%% file: bordism_and_the_CHL_string.tex
\documentclass{amsart}
\input{macros}

\title{Bordism for the 2-group symmetries of the heterotic and CHL strings}
\author{Arun Debray}
\date{\today}

\begin{document}
\maketitle

\begin{abstract}
In the presence of a nonzero B-field, the symmetries of the $\E_8\times\E_8$ heterotic string form a $2$-group,
or a categorified group, as do the symmetries of the CHL string. We express the bordism groups of the corresponding
tangential structures as twisted string bordism groups, then compute them through dimension $11$ modulo a few
unresolved ambiguities. Then, we use these bordism groups to study anomalies and defects for these two string
theories.
\end{abstract}

\tableofcontents

\setcounter{section}{-1}
\section{Introduction}
	\input{intro}

\section{Tangential structures for heterotic and CHL string theories}
	\input{symmetry_types}

\section{Bordism computations}
	\input{computations_section}

\section{Consequences in string theory}
	\input{string_theory_section}
	% TODO: better title?

%\bibliography{references}{}
%\bibliographystyle{alpha}
%\input{bordism_and_the_CHL_string.bbl}
\input{bbl_bib}

\end{document}

%% file: macros.tex
\numberwithin{equation}{section}
\usepackage[utf8]{inputenc}
\usepackage[T1]{fontenc}
\usepackage{lmodern}
\usepackage[usenames, dvipsnames]{xcolor}
\usepackage{adamsmacros}
\usepackage{enumitem}
\usepackage[margin=1.32in]{geometry}
\usepackage{setspace}
\usepackage{graphicx}
\usepackage{mathtools}
\usepackage{amssymb}
\usepackage{amsthm}
\usepackage{xspace}
\usepackage{subcaption}
\usepackage{tikz-cd}
\usepackage{quiver}

\usepackage[backref=page, bookmarks=false]{hyperref}
\usepackage[capitalize, noabbrev]{cleveref}
\usepackage{environ}

\newcommand{\Z}{\mathbb Z}

\newcommand{\R}{\mathbb R}
\newcommand{\C}{\mathbb C}
\newcommand{\T}{\mathbb T}

\newcommand{\GL}{\mathrm{GL}}

\DeclareMathOperator{\Ext}{Ext}
\DeclareMathOperator{\Hom}{Hom}

\newtheorem{introthm}{Theorem}

\newtheorem{lem}[equation]{Lemma}
\newtheorem{cor}[equation]{Corollary}
\newtheorem{prop}[equation]{Proposition}
\newtheorem{thm}[equation]{Theorem}

\theoremstyle{definition}
\newtheorem{exm}[equation]{Example}
\newtheorem{defn}[equation]{Definition}

\newtheorem{conj}[equation]{Conjecture}
\newtheorem{ques}[equation]{Question}

\theoremstyle{remark}
\newtheorem{rem}[equation]{Remark}

\crefname{thm}{Theorem}{Theorems}
\crefname{introthm}{Theorem}{Theorems}
\crefname{lem}{Lemma}{Lemmas}
\crefname{cor}{Corollary}{Corollaries}
\crefname{prop}{Proposition}{Propositions}
\crefname{ex}{Exercise}{Exercises}
\crefname{exm}{Example}{Examples}
\crefname{defn}{Definition}{Definitions}
\crefname{claim}{Claim}{Claims}
\crefname{rem}{Remark}{Remarks}
\crefname{fct}{Fact}{Facts}
\crefname{note}{Note}{Notes}
\crefname{ques}{Question}{Questions}

\newcommand{\term}{\emph} % e.g. "The \term{trace} is defined to be..."

\DeclarePairedDelimiter\abs{\lvert}{\rvert}
\DeclarePairedDelimiter\set{\{}{\}}
\DeclarePairedDelimiter\paren{(}{)}
\DeclarePairedDelimiter\ang{\langle}{\rangle}

\makeatletter
	\let\oldparen\paren
	\def\paren{\@ifstar{\oldparen}{\oldparen*}}
\makeatother

\usepackage{microtype}
\usepackage{hypcap}

\hypersetup{
 colorlinks,
 linkcolor={red!50!black},
 citecolor={green!50!black},
 urlcolor={blue!80!black}
}

\usepackage[mathscr]{eucal}
\makeatletter
\def\instring#1#2{TT\fi\begingroup
  \edef\x{\endgroup\noexpand\in@{#1}{#2}}\x\ifin@}
\def\isuppercase#1{%
  \instring{#1}{ABCDEFGHIJKLMNOPQRSTUVWXYZ}%
}%
\newcommand{\C@lIfUpper}[1]{
 \if\isuppercase{#1}\mathscr{#1}%
 \else #1%
 \fi
}
\newcommand{\cat}[1]{\mathit{\@tfor\next:=#1\do{\C@lIfUpper{\next}}}}
\makeatother

\newcommand{\Ch}{\cat{Ch}}

\newcommand{\Sym}{\mathrm{Sym}}
\newcommand{\fg}{\mathfrak g}

\newcommand{\SO}{\mathrm{SO}}

\renewcommand{\d}{\mathrm d}

\newcommand{\bl}{\text{--}}
\newcommand{\cA}{\mathcal A}
\newcommand{\cQ}{\mathcal Q}
\newcommand{\tmf}{\mathit{tmf}}
\newcommand{\ko}{\mathit{ko}}
\newcommand{\MTString}{\mathit{MTString}}
\newcommand{\spinc}{spin\textsuperscript{$c$}\xspace}
\newcommand{\U}{\mathrm U}
\newcommand{\E}{\mathrm E}
\newcommand{\G}{\mathbb G}
\newcommand{\shortexact}[4]{
\begin{tikzcd}[ampersand replacement=\&]
        0 \& {#1}
        \&  {#2}
        \& {#3}
        \& {0 #4}
        \arrow[from=1-1, to=1-2]
        \arrow[from=1-2, to=1-3]
        \arrow[from=1-3, to=1-4]
        \arrow[from=1-4, to=1-5]
\end{tikzcd}
}
\newcommand{\Aut}{\mathrm{Aut}}
\newcommand{\Spin}{\mathrm{Spin}}
\newcommand{\Pin}{\mathrm{Pin}}

\newcommand{\CW}{\mathrm{CW}}

\newcommand{\rE}{\mathrm E}
\newcommand{\Sq}{\mathrm{Sq}}

\newcommand{\RP}{\mathbb{RP}}
\newcommand{\CP}{\mathbb{CP}}
\newcommand{\HP}{\mathbb{HP}}

\newcommand{\String}{\mathrm{String}}
\newcommand{\pinm}{pin\textsuperscript{$-$}\xspace}
\newcommand{\MO}{\mathit{MO}}
\newcommand{\MTSO}{\mathit{MTSO}}
\newcommand{\MTSpin}{\mathit{MTSpin}}
\newcommand{\inj}{\hookrightarrow}
\renewcommand{\O}{\mathrm O}
\newcommand{\Sph}{\mathbb S}
\newcommand{\ku}{\mathit{ku}}
\newcommand{\SU}{\mathrm{SU}}
\newcommand{\KO}{\mathit{KO}}
\newcommand{\pt}{\mathrm{pt}}
\newcommand{\Mod}{\cat{Mod}}
\newcommand{\Ghet}{\G^{\mathrm{het}}}
\newcommand{\GCHL}{\G^{\mathrm{CHL}}}
\newcommand{\xihet}{{\xi^{\mathrm{het}}}}
\newcommand{\xihetn}{{\xi_n^{\mathrm{het}}}}
\newcommand{\xiCHL}{{\xi^{\mathrm{CHL}}}}
\newcommand{\xiCHLn}{\xi_n^{\mathrm{CHL}}}
\newcommand{\MTxi}{\mathit{MT\xi}}
\newcommand{\MTxihet}{\mathit{MT\xihet}}
\newcommand{\MTxiCHL}{\mathit{MT\xiCHL}}
\newcommand{\cP}{\mathcal P}
\newcommand{\MTO}{\mathit{MTO}}
\newcommand{\DD}{\mathrm{DD}}
\newcommand{\Snb}{S_{\mathit{nb}}^1}

\makeatletter
\DeclareRobustCommand\widecheck[1]{{\mathpalette\@widecheck{#1}}}
\def\@widecheck#1#2{%
    \setbox\z@\hbox{\m@th$#1#2$}%
    \setbox\tw@\hbox{\m@th$#1%
       \widehat{%
          \vrule\@width\z@\@height\ht\z@
          \vrule\@height\z@\@width\wd\z@}$}%
    \dp\tw@-\ht\z@
    \@tempdima\ht\z@ \advance\@tempdima2\ht\tw@ \divide\@tempdima\thr@@
    \setbox\tw@\hbox{%
       \raise\@tempdima\hbox{\scalebox{1}[-1]{\lower\@tempdima\box
\tw@}}}%
    {\ooalign{\box\tw@ \cr \box\z@}}}
\makeatother
\newcommand{\cS}{\mathit{\mathcal Str}}
\newcommand{\pinp}{pin\textsuperscript{$+$}\xspace}
\newcommand{\Bord}{\cat{Bord}}

\newcommand\rightthreearrow{%
        \mathrel{\vcenter{\mathsurround0pt
                \ialign{##\crcr
                        \noalign{\nointerlineskip}$\rightarrow$\crcr
                        \noalign{\nointerlineskip}$\rightarrow$\crcr
                        \noalign{\nointerlineskip}$\rightarrow$\crcr
                }%
        }}%
}

%% file: intro.tex
String theory has long been a place where higher-categorical structures in mathematics meet their applications.
This is true for a few different reasons, but one crucial reason is that many fields in superstring and
supergravity theories have mathematical incarnations that are higher-categorical objects, and so even precisely
setting up mathematical questions coming out of string theory, let alone solving them, often requires engaging with
or developing the foundations of various kinds of geometric objects with higher structure. This paper is concerned
with the appearance of a higher structure called a \term{2-group} in two specific string theories, and how
including this structure affects computations of bordism groups for the tangential structures of these theories. These
bordism groups control anomalies and extended objects for these theories. The main results of this paper are
computations of bordism groups and their generating manifolds through dimension $11$, except for a few ambiguities
we did not addres, for the tangential structures underlying these two string theories.
%\todo{``the introduction section 1 appears to present the application
%to string physics as the main point and comes to mention the mathematical
%content of section 2 only on its third page''}

For the higher structures we investigate in this paper, the story begins with the \term{Kalb-Ramond field}, or the
\term{B-field}. This is an analogue of the field strength of an electromagnetic field, represented as a closed
differential $2$-form with a quantization condition. Locality of quantum field theory means expressing the
field strength of the electromagnetic field as a section of a sheaf, specifically as a connection on a principal
$\T$-bundle, where $\T$ is the circle group. For the B-field, everything is one degree higher: it comes to us as a
closed differential $3$-form with a quantization condition, which we would like to express as a geometric object
that sheafifies. This cannot be a connection on a principal $G$-bundle for a finite-dimensional Lie group $G$;
instead, one models the B-field as a connection on a \term{$\T$-gerbe}, which is a categorification of a principal
$\T$-bundle. A $\T$-gerbe on a manifold $M$ is, roughly speaking, a bundle of groupoids on $M$ which is locally
equivalent to $\pt/\T$. There are several ways to make this precise; we discuss one, Murray's \term{bundle
gerbes}~\cite{Mur96}, in \cref{bundle_gerbe}.

In this article, we consider higher structures in two string theories: the $\rE_8\times\rE_8$ heterotic string, and
the Chaudhuri-Hockney-Lykken (CHL) string. The former is a ten-dimensional superstring theory whose low-energy
limit is ten-dimensional $\mathcal N = 1$ supergravity, and the latter is a nine-dimensional theory obtained from
the $\rE_8\times\rE_8$ heterotic string theory by compactifying on a circle. Both of these theories have
B-fields, but Green and Schwarz~\cite{GS84} showed that in order to cancel an anomaly, the B-field and the
gauge field must satisfy a relation known as a \term{Bianchi identity}. Fiorenza-Schreiber-Stasheff~\cite{FSS12}
and Sati-Schreiber-Stasheff~\cite{SSS12} describe how the Bianchi identity mixes the data of the B-field and
the gauge field into data that can be interpreted as a connection on a principal bundle for a $2$-group $\G$,
specifically a \term{string $2$-group} $\cS(G, \mu)$ associated to the data of a compact Lie group $G$ and a class
$\mu\in H^4(BG;\Z)$; typically, $G$ is the gauge group and $\mu$ is determined by the anomaly polynomial.
%For the $\rE_8\times\rE_8$ heterotic string, $G = (\rE_8\times \rE_8)\rtimes\Z/2$, where $\Z/2$ swaps
%the two factors, and for the CHL string, $G = \rE_8$.

$2$-groups have been used in the theoretical and mathematical physics literature for some time now. This
program began in earnest with work of Baez, Crans, Lauda, Stevenson, and Schreiber~\cite{Bae02, BC04, BL04, BSCS07,
BS07}; more recently, $2$-groups, their symmetries, and their anomalies have made a resurgence in quantum field
theory following work of Córdova-Dumitriescu-Intrilligator~\cite{CDI19} and Benini-Córdova-Hsin~\cite{BCS19}
identifying many examples of $2$-group symmetries in commonly studied QFTs. See also Sharpe~\cite{Sha15} and the
references therein.

In the first part of this article, we introduce the Bianchi identity and $2$-groups, then review work of
Fiorenza-Schreiber-Stasheff~\cite{FSS12} and Sati-Schreiber-Stasheff~\cite{SSS12} mentioned above. These authors
work in the setting of stacks on the site $\cat{Man}$ of smooth manifolds; the data of the B-field $(Q, \Theta_Q)$
and the principal $G$-bundle with connection $(P, \Theta_P)$ on a manifold $M$ refine to maps from $M$ to
classifying stacks of these data.  The data of an identification of two differential characteristic classes
associated to $\Theta_P$ and $\Theta_Q$ gives rise to
\begin{enumerate}
	\item a principal $\cS(G, \mu)$-bundle lifting $P$ for a specified choice of $\mu$ (\cref{2shape}), and
	\item local data of solutions to the Bianchi identity (\cref{2gp_satisfies_bianchi},~\cite[\S 6.3]{FSS12}).
\end{enumerate}

Inspired by this, we introduce the tangential structures $\xihet$ and $\xiCHL$, which are special cases of a
general construction of Sati-Schreiber-Stasheff~\cite[Definition 2.8]{SSS12}: a $\xihetn$-structure on a spin
manifold $M$ is data of a principal $\Ghet_n$-bundle, where $\Ghet_n \coloneqq \cS(\Spin_n\times
(\rE_8\times\rE_8\rtimes\Z/2), c_1 + c_2 - \lambda)$~\eqref{gnhet_defn}, whose associated $\Spin_n$-bundle via the
quotient $\Ghet_n\to\Spin_n$ is the principal $\Spin_n$-bundle of spin frames (\cref{top_het}). This is compatible
as $n$ varies, allowing us to stabilize and define a $\xihet$-structure as usual. The definition of $\xiCHL$ in
\cref{top_CHL}, which coincides with $B\String^{2a}$ in~\cite[(2.18), \S 2.2.3]{SSS12}, is analogous. Related
tangential structures appear in~\cite{Sat11, FSS15, FSS15a, FSS21}.

Given a tangential structure, we can compute bordism groups, and indeed the point of this paper is to compute
$\xihet$ and $\xiCHL$ bordism groups in low dimensions. These bordism groups can then be used to learn more about
the $\rE_8\times\rE_8$ heterotic and CHL strings. We have two primary applications in mind.
\begin{enumerate}
	\item The \term{cobordism conjecture} of McNamara-Vafa~\cite{MV19} is an application to the question of what
	kinds of spacetime backgrounds are summed over in quantum gravity. Such backgrounds are often taken to be
	manifolds or something closely related equipped with data of a tangential structure $\xi$. The cobordism
	conjecture says that if $\xi$ is the most general
	tangential structure which can appear in this way in any particular $d$-dimensional theory of quantum gravity, then
	$\Omega_k^\xi = 0$ for $3\le k\le d-1$. We will see that $\Omega_k^{\xihet}$ and $\Omega_k^{\xiCHL}$ are often
	nonzero in that range. This is consistent with the cobordism conjecture: it suggests that $\xihet$ and $\xiCHL$
	are not the most general tangential structures
%	\todo{here and in \S 3: refer to the class of backgrounds we sum over -- a priori maybe we missed some and
%	we're learning what class of backgrounds we need to sum over}
	that can be summed over. Typically these
	bordism groups are killed by allowing singular manifolds corresponding to considering the theory with branes or
	other defects, so one can use bordism computations to predict new defects in string theories.
%	\todo{``This paragraph seems excessively handwavy even by physics
%standards:
%Recall that the definition and even the basic principles of "quantum gravity"
%(beyond effective perturbation theory) are famously a matter of lively debate,
%with a whole range of partial proposals of widely differing nature competing
%with each other. But the one aspect that all these proposals tend to have in
%common is the notion that quantum spacetimes are no longer described by
%manifolds, but by some "fuzzy" geometry, which might be noncommutative or even
%more exotic.''}
	\item A broad class of $n$-dimensional quantum field theories come with data of an \term{anomaly}, which in
	many cases can roughly be described an
	$(n+1)$-dimensional invertible field theory $\alpha$.
%	\todo{``A wealth of qualifications would be needed to make this sentence well-formed and
%true. If the author does not want to dwell on this fine print, at least a
%warning would be in order. For instance, one could prefix the sentence by the
%words: `Very roughly, it may be expected that for a broad class of quantum field
%theories...' '' So: here and in section 3, clarify (a broad class of QFTs)}
	In some cases one wants to trivialize $\alpha$,
%	\todo{``This sounds misleading. What one strictly needs to trivialize are anomalies that
%are actual obstructions to constructing a QFT at all (like gauge anomalies) .
%However, if a QFT has already been constructed except that its representation of
%composition of cobordisms is projective (which is what it means for the anomaly
%to be given by an invertible cobordism representation in one dimension higher)
%hen that's good enough for many purposes.''}
	meaning exhibiting an isomorphism from $\alpha$ to the trivial field theory. By work of
	Freed-Hopkins-Teleman~\cite{FHT10} and Freed-Hopkins~\cite{FH21}, invertible field theories can be classified
	using bordism group computations. For both the $\rE_8\times\rE_8$ heterotic string and the CHL string, the
	bordism groups indicating a potential anomaly are nonzero, and it would be interesting to check whether the
	corresponding anomalies are nontrivial.
\end{enumerate}
See \S\ref{s_phys}, as well as \cref{ques1,ques2,ques3} below, for more on these applications and what we can learn
from our bordism computations.

%[ok, the first thing we do is determine the tangential structure and see that this relates to the Bianchi identity. not
%new exactly but not explictly cobbled together like this before]
%
%[our main theorems are these two computations, and here's a little blurb about how they're hard]
Our main theorems are the following two computations of the $\xihet$ and $\xiCHL$ bordism groups in low dimensions.
%\todo{``closely related results on string bordism ought to be known and would
%deserve more of a mention in the introduction''}
\begin{introthm}
\label{main_1}
For $k\le 10$, the $\xihet$-bordism groups are:
\begin{alignat*}{2}
        \Omega_0^{\xihet}  &\cong \Z & \qquad\qquad \Omega_6^\xihet &\cong \Z/2\\
        \Omega_1^{\xihet}  &\cong \Z/2\oplus\Z/2 &\Omega_7^\xihet &\cong \Z/16\\
        \Omega_2^{\xihet}  &\cong \Z/2\oplus\Z/2 &\Omega_8^\xihet &\cong \Z^3\oplus (\Z/2)^{\oplus i}\\
        \Omega_3^{\xihet}  &\cong \Z/8&\Omega_9^\xihet &\cong (\Z/2)^{\oplus j}\\
        \Omega_4^{\xihet}  &\cong \Z\oplus\Z/2& \Omega_{10}^\xihet &\cong (\Z/2)^{\oplus k}.\\
        \Omega_5^{\xihet}  &\cong 0%& \Omega_{11}^\xihet &\cong A,
\end{alignat*}
Here, either $i = 1$, $j = 4$, and $k = 4$, or $i = 2$, $j = 6$, and $k = 5$.

$\Omega_{11}^\xihet$ is an abelian group of order $64$ isomorphic to one of $\Z/8\oplus\Z/8$, $\Z/16\oplus\Z/4$,
$\Z/32\oplus\Z/2$, or $\Z/64$.
\end{introthm}
This is a combination of \cref{het_at_2,no_odd_het}. In \S\ref{xihet_gens}, we find manifold representatives for
all classes in $\Omega_k^{\xihet}$ for $k\le 10$ except potentially for two missing classes $X_8$ and $X_9$ of
dimensions $8$, resp.\ $9$ and their products with $\Snb$. These classes may or may not be zero depending on the
fate of an Adams differential. In \S\ref{s:X8}, we find a manifold representing $X_8$: if the unaddressed Adams
differential vanishes, $X_8$ should be added to the list of generators in \S\ref{xihet_gens}, and if the
differential does not vanish, then $X_8$ bounds as a $\xihet$-manifold.

Our calculation of $\xiCHL$-bordism builds on work of Hill~\cite[Theorem 1.1]{Hil09}, who computes
$\Omega_*^{\String}(B\rE_8)$ in dimensions $14$ and below.
\begin{introthm}
\label{main_2}
For $k\le 11$, there is an abstract isomorphism from $\Omega_*^{\xiCHL}$ to the free and $2$-torsion summands of
$\Omega_*^\String(B\rE_8)$. Therefore, by Hill's computation~\cite{Hil09}, there are isomorphisms
\begin{alignat*}{2}
	\Omega_0^{\xiCHL} &\cong \Z & \Omega_6^{\xiCHL} &\cong \Z/2\\
	\Omega_1^{\xiCHL} &\cong \Z/2 \qquad\qquad& \Omega_7^{\xiCHL} &\cong 0\\
	\Omega_2^{\xiCHL} &\cong \Z/2 & \Omega_8^{\xiCHL} &\cong \Z\oplus\Z\oplus\Z/2\\
	\Omega_3^{\xiCHL} &\cong \Z/8 & \Omega_9^{\xiCHL} &\cong \Z/2\oplus\Z/2 \oplus\Z/2\\
	\Omega_4^{\xiCHL} &\cong \Z & \Omega_{10}^{\xiCHL} &\cong \Z/2\oplus\Z/2\\
	\Omega_5^{\xiCHL} &\cong 0 & \Omega_{11}^{\xiCHL} &\cong \Z/8.
\end{alignat*}
\end{introthm}
This is a combination of \cref{2_loc_CHL,no_odd_CHL}. We also obtain some information about manifold
representatives of generators of these groups.

The computational tool we use to prove \cref{main_1,main_2} is standard: the Adams spectral sequence. This spectral
sequence has seen plenty of applications in the mathematical physics literature, and there is a standard procedure
reviewed by Beaudry-Campbell~\cite{BC18} for simplifying the $E_2$-page for a wide class of tangential structures,
namely those which can be described as oriented, \spinc, spin, or string bordism twisted by a virtual vector
bundle. For example, the twisted string bordism computations of~\cite{FK96, Fan99, FW10} make use of this
simplifying technique. Unfortunately, this procedure is unavailable to us: in \cref{simple_shear_no}, we prove that $\xihet$ and
$\xiCHL$ cannot be described as twists of this sort. However, we are still able to describe them as twists in a
more general sense due to Ando-Blumberg-Gepner-Hopkins-Rezk~\cite{ABGHR14a, ABGHR14b}: adapting an argument of
Hebestreit-Joachim~\cite{HJ20}, one learns that the Thom spectra for
$\xihet$ and $\xiCHL$ can be produced as the $\MTString$-module Thom spectra associated to certain maps to
$B\GL_1(\MTString)$. Using this structure, in upcoming joint work with Matthew Yu, we are able to prove a theorem
simplifying the calculation of the $E_2$-page:
\begin{introthm}[Debray-Yu~\cite{DY}]
In topological degrees $15$ and below, the $E_2$-pages of the Adams spectral sequences computing $2$-completed
twisted string bordism for a class of twists including those for $\xihet$ and $\xiCHL$ can be computed as $\Ext$
over the subalgebra $\cA(2)$ of the Steenrod algebra.
\end{introthm}
What we prove is more precise and holds in more generality; see \cref{fake_shearing_DY,fake_shearing} for that
version of the result.\footnote{Since~\cite{DY} is not available yet, we provide a proof sketch of the case we need
in \cref{DY_proof_rem}.}

The $\cA(2)$-module Ext groups we have to compute are simpler than what one a
priori has to work with over the entire Steenrod algebra $\cA$. We do not need this simplification at odd
primes; there the full Adams spectral sequence is easier to work with, and the absence of a simplification does not
hinder us (though see also~\cite[\S 3.2]{DY}).

The reason we computed these bordism groups in this paper is with applications to physics, specifically to
anomalies and the cobordism conjecture, in mind. We discuss some implications of our calculations in
\S\ref{s_phys}; for example, one of the $\Z/2$ summands of $\Omega_1^\xihet$ corresponds to the non-supersymmetric
$7$-brane recently discovered by Kaidi-Ohmori-Tachikawa-Yonekura~\cite{KOTY23}. We end this section of the
introduction with some questions related to these physics predictions.
%\todo{For the heterotic anomaly question, make clear what we do/don't know about $\Omega_{11}^{\xihet}$ and the
%corresponding anomaly}
%\todo{Make explicit the heterotic/M-theory duality question}
\begin{ques}
\label{ques1}
What does the Kaidi-Ohmori-Tachikawa-Yonekura $7$-brane correspond to in Ho\v{r}ava-Witten theory, and what does
this look like in bordism? Ho\v{r}ava-Witten~\cite{HW96b, HW96a, Wit96} proposed that the $\rE_8\times\rE_8$
heterotic string can be identified with a certain limit of M-theory compactified on an interval; thus this ought to
correspond to a notion of bordism of manifolds with boundary. Conner-Floyd~\cite[\S 16]{CF66} define a notion
bordism of compact manifolds with boundary --- is this the correct kind of bordism for applications to
McNamara-Vafa's conjecture?
%In dimensions $k = 1, 3, 7$, $\RP^k$ with $\xihet$-structure induced from the nontrivial double cover and the map
%\todo{``regarding the 2-group "$G_n^het$" This seems (best to make explicit the
%detailed definition of the terms that appear) to essentially be the "twisted
%String" 2-group considered in eq (6) of: arxiv.org/abs/2002.11093
%going back to
%(3.6) of: arxiv.org/abs/1007.5419
%Def. 2.8 of: arxiv.org/abs/0910.4001
%Rem. 4.1.1 in: arxiv.org/abs/1202.2455''}
%a$\Z/2\times\String\to\Ghet$ generates a summand of $\Omega_k^\xihet$. The corresponding bordism class should
%correspond to an $(8-k)$-brane in $\rE_8\times\rE_8$ heterotic bordism, and the $\RP^3$ brane is predicted to be
%distinct from the NS5-brane.  Do these predicted branes in fact exist in $\rE_8\times\rE_8$ heterotic string
%theory?
\end{ques}
We discuss some additional extended objects predicted by our bordism computations to exist in the
$\rE_8\times\rE_8$ heterotic and CHL strings in \S\ref{s_cobordism_conjecture}.
\begin{ques}
\label{ques2}
Is the $\Z/2$ symmetry exchanging the two $\rE_8$-bundles in $\rE_8\times\rE_8$ heterotic string theory anomalous?
Because $\Omega_{11}^{\xihet}$ is nonzero, we were unable to rule out this anomaly.
\end{ques}
Witten~\cite[\S 4]{Wit86} and Tachikawa-Yonekura~\cite{TY21} show that the $\rE_8\times\rE_8$ heterotic string is
anomaly-free in certain cases, but they do not address the $\Z/2$ symmetry. %We show that $\Omega_{11}^{\xihet}$ is
%an extension of $\Z/8$ by $\Z/8$, and the subgroup $\Z/8$ is generated by $\mathrm{Bott}\times\RP^3$
\begin{ques}
\label{ques3}
Does the CHL string have an anomaly? This anomaly could be nontrivial, because $\Omega_{10}^{\xiCHL}\cong
\Z/2\oplus\Z/2$.
\end{ques}
There is another application of twisted string bordism to physics that we did not address in this paper: studying
elliptic genera, the Witten genus and related invariants, along the lines of, e.g., Bunke-Naumann~\cite{BN14},
McTague~\cite{McT14}, Han-Huang-Duan~\cite{HHD21}, and Berwick-Evans~\cite{Ber23}. It would be interesting to study
whether the calculations in this paper could be applied in similar contexts.
%Outline:
%\begin{itemize}
%	\item String theory has long been a place where higher-categorical structures in math meet their applications
%	\item Lots of work has been done on this; we bring it into computations
%	\item Why are we doing those computations?
%	\item Here are the computations we do; note that they're quite a bit harder than typical applications
%	\item Here's what we learned about string theory
%\end{itemize}
%
\subsection*{Outline}
We begin in \S\ref{heterotic_intro} by introducing the fields present in 10d $\mathcal N = 1$ supergravity, the
low-energy limit of heterotic string theory. We discuss how the Green-Schwarz anomaly cancellation condition
imposes an equation called the \term{Bianchi identity}~\eqref{first_bianchi} on the fields in this theory. We then
generalize this to a \term{twisted Bianchi identity}~\eqref{twisted_Bianchi} associated to data of a Lie group $G$
and a class $\mu\in H^4(BG;\Z)$. In \S\ref{higher_stuff}, we relate these Bianchi identities to the presence of a
$2$-group symmetry in this field theory. We begin by reviewing $2$-groups, their principal bundles, and their
connections, and in \cref{string_cover} define the string cover $\mathcal S(G, \mu)$ corresponding to a group $G$
and a class $\mu\in H^4(BG;\Z)$. Then we review work of Fiorenza-Schreiber-Stasheff~\cite{FSS12} and
Sati-Schreiber-Stasheff~\cite{SSS12} relating the Bianchi identity to twisted string structures.
Using this, we define the heterotic tangential structure in \cref{top_het},
which is the topological part of the structure necessary for defining $\mathcal N = 1$ supergravity. Then, in
\S\ref{CHL_intro}, we introduce the CHL string and define the CHL tangential structure using what we learned in
\S\ref{higher_stuff}.

In \S\ref{s_computations}, we compute the bordism groups $\Omega_*^{\xihet}$ and $\Omega_*^{\xiCHL}$ in low
degrees. For the latter we are able to completely compute them in dimensions $11$ and below, but for the former, we
have only partial information above dimension $7$, occluded by Adams differentials and an extension problem we could not
solve. We begin in \S\ref{shearing} by discussing how to simplify the Thom spectra $\MTxihet$ and $\MTxiCHL$;
we prove in \cref{simple_shear_no} that a standard approach does not work, and so we use a different idea:
construct $\MTxihet$ and $\MTxiCHL$ as $\MTString$-module Thom spectra using machinery developed by
Ando-Blumberg-Gepner-Hopkins-Rezk.  We review this machinery and discuss how it leads to \cref{fake_shearing}, a
special case of the main theorem of our work~\cite{DY} joint with Matthew Yu, simplifying the calculation
of the $E_2$-page of the Adams spectral sequence at $2$ for a wide class of twisted string bordism groups.
Next, in \S\ref{xihet_at_2}, we undertake this computation for $\xihet$. We do not have such a simplification
at odd primes, so in \S\ref{xihet_odd} we press ahead directly with the Adams spectral sequence for $\xihet$,
proving in \cref{no_odd_het} that $\Omega_*^{\xihet}$ lacks odd-primary torsion in degrees $11$ and below. Finally,
in \S\ref{s_CHL_bord} we run the analogous calculations for the CHL string, again using \cref{fake_shearing} at $p
= 2$ and arguing more directly at odd primes.

The final section, \S\ref{s_phys}, is about applications to string theory. We first discuss the cobordism
conjecture of McNamara-Vafa~\cite{MV19} in \S\ref{s_cobordism_conjecture}, and go over a few predictions that
follow from the bordism group computations in \S\ref{s_computations}. In \S\ref{anomalies_general}, we briefly
introduce anomalies of quantum field theories and their bordism-theoretic classification, and touch on questions
raised by our bordism computations.

\subsection*{Acknowledgements}
I especially want to thank Miguel Montero both for suggesting this project and for many helpful conversations about
the material in this paper, and Matthew Yu for many helpful discussions relating to~\cite{DY} and other ideas
related to this paper. I also want to thank Markus Dierigl and the anonymous referee for helpful comments on a draft. In addition, this paper benefited from conversations with
Ivano Basile, Matilda Delgado,
Jacques Distler, Dan Freed, Jonathan J. Heckman,
Justin Kaidi,
Jacob McNamara,
Yuji Tachikawa, and Roberto Telléz Domínguez; thank you to all!

Part of this project was completed while AD visited the Perimeter
Institute for Theoretical Physics; research at Perimeter is supported by the Government of
Canada through Industry Canada and by the Province of Ontario through the Ministry
of Research \& Innovation.

%% file: symmetry_types.tex
\label{s_symmetry}
The goal of this section is to define the tangential structures $\xihet$ and $\xiCHL$ that are necessary to
formulate the (low-energy limits of) the $\rE_8\times\rE_8$ heterotic string and the CHL string. By ``tangential
structure'' we mean the topological part of the structure needed on a manifold to define a given field theory; see
\cref{symm_defn} for the precise definition. The presence of a B-field in both theories means that these
tangential
structures arise as classifying spaces of higher groups. First, we introduce the heterotic string in
\S\ref{heterotic_intro}, and see what data and conditions are told to us by Green-Schwarz anomaly cancellation;
then in \S\ref{higher_stuff}, we reinterpret that data as combining the gauge field and the B-field into a
connection for a principal bundle for a higher group. Finally, in \S\ref{CHL_intro}, we use the general theory from
\S\ref{higher_stuff} to determine the tangential structure for the CHL string.

The material in this section is not new, though it was not always stated in this form before. The
fact that a Bianchi identity/Green-Schwarz mechanism is expressing a lift to a connection for a higher-group
principal bundle is well-known; see Fiorenza-Schreiber-Stasheff~\cite{FSS12} and
Sati-Schreiber-Stasheff~\cite{SSS12}. % and [46, 21].

\subsection{The $\rE_8\times\rE_8$ heterotic string}
\label{heterotic_intro}
Heterotic string theories are ten-dimensional superstring theories whose low-energy limits are 10d $\mathcal N = 1$
supergravity theories. These supergravity theories can have Yang-Mills terms, and so are parametrized by the data
of the gauge group $G$, a compact Lie group. However, not all choices of $G$ yield valid supergravity theories;
there is the potential for an anomaly that must be trivialized, and this is quite a strong constraint, implying
that the connected component of the identity in $G$ must be either $\rE_8\times\rE_8$ or $G =
\mathrm{SemiSpin}_{32}$\footnote{The center of $\Spin_{4k}$ is isomorphic to
$\Z/2\times\Z/2$. Quotienting by one copy of $\Z/2$ yields $\SO_{4k}$; the quotients by the two other $\Z/2$
subgroups are isomorphic, and are called $\mathrm{SemiSpin}_{4k}$. See~\cite{McI99}.}~\cite{GS84, ATD10}. The
anomaly cancellation mechanism itself, due to Green-Schwarz~\cite{GS84}, combines the different fields in the
theory into a connection for a principal $\G$-bundle, where $\G$ is a higher group;\footnote{Green-Schwarz' work
only cancels the perturbative part of the anomaly; see \S\ref{anomalies_general} for more information.} we use this
subsection to discuss the fields and the Green-Schwarz condition, and the next subsection to discuss the role of
higher group. In this paper, we will focus solely on the $\rE_8\times\rE_8$ case; it would be interesting to study
the analogues of the computations and applications in this paper in the $\mathrm{SemiSpin}_{32}$ case.

The group $\Z/2$ acts on $\rE_8\times\rE_8$ by exchanging the two factors, and the setup of heterotic string
theory, including the low-energy supergravity limit and Green-Schwarz' anomaly cancellation, is invariant under
this symmetry, so we can expand the gauge group to $G\coloneqq (\rE_8\times\rE_8)\rtimes\Z/2$.\footnote{Though we
often use the standard name ``the $\rE_8\times\rE_8$ heterotic string'' to refer to this theory, we will always
consider the larger gauge group $(\rE_8\times\rE_8)\rtimes\Z/2$.} This appears to have
first been noticed by McInnes~\cite[\S I]{McI99}; see also~\cite[\S 2.2.1]{dBDHKMMS00}.

The fields of 10d $\mathcal N = 1$ supergravity on a manifold $M$ include:
\begin{itemize}
	\item a metric $g$,
	\item a spin structure on $M$,
	\item a principal $G$-bundle $P\to M$ with connection $\Theta_P$,
	\item a \term{B-field} or \term{Kalb-Ramond field}, a gerbe $Q\to M$ with connection $\Theta_Q$, and \item
	several additional fields (the dilaton, dilatino, gravitino, and gaugino) which will not be directly relevant to
	this paper.
\end{itemize}
Let us say more about the B-field, since its model as a gerbe with connection may be less familiar. A
\term{gerbe} is a categorification of the idea of a principal $\T$-bundle; here $\T$ is the circle group.
Thus, for example, a principal $\T$-bundle $P\to M$ is classified by its first Chern class $c_1(P)\in
H^2(M;\Z)$, and a gerbe $Q\to M$ is classified by its \term{Dixmier-Douady class} $\DD(Q)\in
H^3(M;\Z)$~\cite{DD63, Bry93}. A connection on a principal $\T$-bundle has holonomy around loops; a connection on a
gerbe has holonomy on closed surfaces. And so on.

Gerbes were first introduced by Giraud~\cite{Gir71}. There are several different and equivalent ways to precisely
define gerbes and their connections; heuristically you can think of a gerbe on $M$ as a sheaf of groupoids on
$M$ locally equivalent to the trivial sheaf with fiber $\pt/\T$. One way to make this precise is the following.

If $f\colon Y\to X$ is a map, we let $Y^{[n]}\coloneqq Y\times_X Y\times_X \dotsm \times_X Y$; $Y^{[n]}$ is the
space of $n$-simplices in the Čech nerve for $f$.
\begin{defn}[Murray~\cite{Mur96}]
\label{bundle_gerbe}
A \term{bundle gerbe} over a manifold $M$ is a surjective submersion $\pi\colon Y\to M$, a $\T$-bundle $P\to
Y^{[2]}$, and an isomorphism $\mu\colon \pi_{12}^* P\otimes \pi_{23}^*P\overset\cong\to \pi_{13}^*P$ of
$\T$-bundles over $Y^{[3]}$ satisfying the natural associativity condition (see below) over $Y^{[4]}$.
\end{defn}
Given two $\T$-bundles $P_1,P_2\to X$, their \term{tensor product} $P_1\otimes P_2$ is the unit circle bundle
inside the tensor product of the Hermitian line bundles $L_1,L_2\to X$ associated to $P_1$, resp.\ $P_2$. The maps
$\pi_{12},\pi_{23},\pi_{13}\colon Y^{[3]}\rightthreearrow Y^{[2]}$ are the three face maps in the Čech nerve
$Y^\bullet$ associated to $f$, given explicitly by contracting two of the three copies of $Y$ via $Y\times_X Y\to Y$.

The associativity condition in \cref{bundle_gerbe} is a little unwieldy to state explicitly, but can be found in
in~\cite[Definition 4.1(2)]{Mur10}.
\begin{defn}[\cite{Mur96}]
A \term{connection} $\Theta_Q$ on a bundle gerbe $Q = (Y, P, \mu)$ is data of a $2$-form $B\in\Omega^2(Y)$ and a
connection $\Theta_P$ on $P$ such that if $\Omega_P\in\Omega^2(P)$ denotes the curvature of $P$ and
$\pi_1,\pi_2\colon Y^{[2]}\to Y$ are the two projections, then
\begin{equation}
	\Omega_P = \pi_2^*B - \pi_1^*B.
\end{equation}
The \term{curvature} of $\Theta_Q$ is $\Omega_Q\coloneqq \d B$, which is a closed $3$-form.
\end{defn}
The key thing to know about this definition is that, just like a principal $\T$-bundle $P\to M$ with connection
locally has a connection $1$-form $A$ and globally has a curvature $2$-form $\Omega_P$ which locally satisfies
$\Omega_P = \d A$, a gerbe with connection $Q$ locally has a connection $2$-form $B$ and globally has a
curvature $3$-form $\Omega_Q$ which locally satisfies $\Omega_Q = \d B$. For more information, see, e.g.,
Brylinski~\cite[\S 5.3]{Bry93}.

\begin{defn}
\label{char_class_e8}
Because $\rE_8$ is a simple, connected, simply connected, compact Lie group, there is a canonical isomorphism
$H^4(B\rE_8;\Z)\overset\cong\to\Z$. Let $c$ denote the generator corresponding to $1\in\Z$. In
$B(\rE_8\times\rE_8) \simeq B\rE_8\times B\rE_8$, let $c_1$ and $c_2$ denote the copies of $c$ coming from the
first, resp.\ second copies of $B\rE_8$ via the Künneth map.

The class $c_1 + c_2$ is invariant under the $\Z/2$ swapping action, so descends via the Serre spectral sequence to
a class in $H^4(B((\rE_8\times\rE_8)\rtimes\Z/2);\Z)$, which we also call $c_1 +c_2$.
\end{defn}
\begin{defn}
$\Spin_n$ is also a compact, connected, simply connected simple Lie group when $n \ge 3$, and the generator of
$H^4(B\Spin_n;\Z)\overset\cong\to\Z$ corresponding to $1$ is denoted $\lambda$.
\end{defn}
The class $\lambda$ is preserved under the standard embeddings $\Spin_n\inj\Spin_{n+k}$, so we often work with its
stabilized avatar $\lambda\in H^4(B\Spin;\Z)$. We use this to define $\lambda$ for $\Spin_n$ when $n < 3$. Because
$2\lambda = p_1$, the class $\lambda$ is often denoted $\tfrac 12 p_1$. The mod $2$ reduction of $\lambda$ is the
Stiefel-Whitney class $w_4$.
%[\TODO: who first identified $2\lambda = p_1$?]
\begin{lem}[Whitney sum formula]
\label{lambda_whitney}
Let $X$ be a topological space and $E_1,E_2\to X$ be two vector bundles with spin structure. Then
$\lambda(E_1\oplus E_2) = \lambda(E_1) + \lambda(E_2)$.
\end{lem}
\begin{proof}
It suffices to prove the universal case, which amounts to the calculation of the pullback of $\lambda$ by the map
\begin{equation}
	\oplus\colon B\Spin_{k_1}\times B\Spin_{k_2}\longrightarrow B\Spin_{k_1+k_2}.
\end{equation}
For $n\ge 3$, $\Spin_n$ is a connected, simply connected, compact simple Lie group, so $H^\ell(B\Spin_n;\Z)$
vanishes for $\ell = 1,2,3$ and is isomorphic to $\Z$ for $\ell = 0,4$. For $n < 3$, $H^\ast(B\Spin_n;\Z)$ is still
trivial or free abelian in degrees $4$ and below. Therefore by the Künneth formula, for all $k_1,k_2$,
$H^4(B\Spin_{k_1}\times B\Spin_{k_2};\Z)$ is a free abelian group, meaning that if we can show $2\lambda(E_1\oplus
E_2) = 2\lambda(E_1) + 2\lambda(E_2)$, then we can deduce $\lambda(E_1\oplus E_2) = \lambda(E_1) + \lambda(E_2)$.

As $2\lambda = p_1$, we have reduced to the Whitney sum formula for $p_1$. The Whitney sum formula $p_1(E_1\oplus
E_2) = p_1(E_1) + p_1(E_2)$ does not actually hold for all vector bundles, but Brown~\cite[Theorem 1.6]{Bro82} (see
also Thomas~\cite{Tho62}) showed that the difference $p_1(E_1\oplus E_2) - p_1(E_1) - p_1(E_2)$ vanishes when $E_1$
and $E_2$ are orientable, so in our setting of spin vector bundles, we can conclude.
\end{proof}
\begin{rem}
There are other ways to prove \cref{lambda_whitney}: for example, it follows immediately from a result of
Jenquin~\cite[Corollary 4.9]{Jen05} in a simple generalized cohomology theory. Johnson-Freyd and Treumann~\cite[\S
1.4]{JFT20} sketch another proof of \cref{lambda_whitney}.
\end{rem}
Next, we introduce the Chern-Weil homomorphism. Let $G$ be a Lie
group with Lie algebra $\fg$, and let $f\in\Sym^k(\fg^\vee)$, i.e.\ $f$ is a degree-$k$ polynomial function on
$\fg$ which is invariant under the adjoint $G$-action on $\fg$. Given a manifold $M$, a principal $G$-bundle $P\to
M$, and a connection $\Theta$ on $P$, let $\Omega\in\Omega_P^2(\fg)$ denote the curvature $2$-form. Then one can
evaluate $f$ on $\Omega^{\wedge k}\in\Omega_P^{2k}(\fg^{\otimes k})$, producing a form $f(\Omega^{\wedge
k})\in\Omega_P^{2k}$; because $f$ is $\mathrm{Ad}$-invariant, $f(\Omega^{\wedge k})$ descends to a form
$w(\Theta)\in\Omega_M^{2k}$, which is always closed. This defines a ring homomorphism, called the \term{Chern-Weil
homomorphism}~\cite{Car49, Che52},
\begin{subequations}
\begin{equation}
	w\colon \Sym^\bullet(\fg^\vee)\longrightarrow H_{\mathrm{dR}}^*(M),
\end{equation}
which doubles the degree and is natural in $M$; moreover, the de Rham class of $w(\Theta)$ depends on $P$ but not
on the connection. Using de Rham's theorem and naturality, $w$ upgrades to a ring homomorphism
\begin{equation}
	w\colon\Sym^\bullet(\fg^\vee)\longrightarrow H^*(BG;\R),
\end{equation}
\end{subequations}
which Chern and Weil showed is an isomorphism when $G$ is compact~\cite{Che52, Wei49}. Thus, when $G$ is compact, a
class $x\in H^{2*}(BG;\Z)$ defines a polynomial $\CW_x\in\Sym^*(\fg^\vee)$, the $w$-preimage of the de Rham class
of $x$. We will also write $\CW_x(\Theta)$ to denote the form defined by evaluating the polynomial $\CW_x$ on the
curvature form of $\Theta$.

Returning to 10d $\mathcal N = 1$ supergravity, Green-Schwarz~\cite{GS84} noticed that in order to trivialize an
anomaly, one has to impose a relation between $P$ and $Q$ and their connections, so that $Q$ is not quite a gerbe,
but instead something twisted. Specifically, the curvature $\Omega_Q$ is no longer closed, but instead satisfies
the equation
\begin{equation}
\label{first_bianchi}
	\d\Omega_Q = \CW_{c_1+c_2}(\Theta_P) - \CW_\lambda(\Theta^{\mathrm{LC}}),
\end{equation}
%instead of asking for $\Theta_Q$ to satisfy $\d B = \Omega_Q$, we must
%ask that
%\begin{equation}
%\label{first_bianchi}
%	\d B = \Omega_Q - \CS_{c_1 + c_2}(\Theta_P) + \CS_\lambda(\Theta^{\mathrm{LC}}),
%\end{equation}
where $\Theta^{\mathrm{LC}}$ is the Levi-Civita connection on the principal $\Spin_n$-bundle of frames of
$M$.\footnote{Before Green-Schwarz, it was already known that $\CW_{c_1+c_2}(\Theta_P)$ and $\d\Omega_Q$ had to mix
in order to preserve supersymmetry, thanks to work of Bergshoeff-de Roo-de Wit-van Nieuwenhuizen~\cite{BdRdWvN82}
and Chapline-Manton~\cite{CM83}.}
%As $B$ is only defined locally,
%\todo{re: $B$ being defined only locally: ``This is absolutely crucial and would be needed to say with detail and precision.
%The formulas for "d B" are wrong if taken at face value, and it seems unwise,
%certainly unusual,  in a mathematical text to carry such abuse of notation
%along. Also for no apparent reason: The discussion would be more transparent in
%term of the curvatures/fluxes which are differential form representatives of the
%%cohomology classes that the author cares about.
%
%One way to refer to the necessary machinery here us by appeal to "bundle
%2-gerbes". These seem more relevant to the author's cause than the 1-gerbes
%recalled in Def. 1.1.''}we only ask for this equation to hold locally on
%$M$.
This is called a \term{Bianchi identity} in the physics literature, motivating the following definition.
%\todo{``More commonly, what are called Bianchi identities are differential relations on
%the curvatures/fluxes, not on the potentials. The Green-Schwarz Bianchi identity
%is of the form (up to factors) $d H3 = <F^2> - <R^2>$. This is  (the globalization, see below, of the) de Rham differential of the
%formula which the author gives.''}
\begin{defn}
Given data of a compact Lie group $G$ and a class $\mu\in H^4(BG;\Z)$, the \term{twisted Bianchi identity} is the
equation
\begin{equation}
\label{twisted_Bianchi}
	\d H = \CW_\mu(\Theta_P),
\end{equation}
where $H$ is a $3$-form and $\Theta_P$ is a connection on a principal $G$-bundle.
\end{defn}
As in the case of~\eqref{first_bianchi}, we think of this as mixing the data of two connections, one on a principal
$G$-bundle and one on a gerbe. In the next section, we interpret twisted Bianchi identities as coming from
connections on higher groups.

%{\color{gray}
%Outline of this part.
%\begin{itemize}
%	\item Quick summary: what are we doing in this section?
%	\item Introduce the fields and the Bianchi identity. Where did the Bianchi identity come from?
%	\item How these come together to form a $2$-group. What is a $2$-group, and what does a ``$2$-group symmetry''
%	give you in physics? E.g.\ currents and the like. Example for String $2$-groups.
%	\item Now, turn back on $\tr(R^2)$, so we have a twisted $2$-group symmetry, much like a \spinc structure is a
%	twisted $\U_1$ gauge symmetry.
%\end{itemize}
%
%Note: some work of Sati and Schreiber \url{https://arxiv.org/pdf/0910.4001.pdf} is relevant here!
%
%\url{https://arxiv.org/abs/2206.00660} may be relevant on the definition of the classifying space of a topological
%$2$-group. \url{https://arxiv.org/abs/2206.01287} says something about $2$-groups of the form $G\ltimes BA$.
%
%
%
%Goal: discuss the field content of the $\E_8\times \E_8$ heterotic string and the various simplifications thereof.
%The fields are:
%\begin{itemize}
%	\item Two principal $\E_8$-bundles $P_1$ and $P_2$ with connections $\Theta_1$ and $\Theta_2$. These have
%	characteristic classes $c(P_1), c(P_2)\in H^4(M;\Z)$.
%	\item We also want to consider the $\Z/2$ symmetry which switches the $\E_8$-bundles. Thus we really have a
%	principal $G$-bundle for $G\coloneqq (\E_8\times\E_8)\ltimes\Z/2$ and a connection for this $G$-bundle.
%	\item A $B$-field, which is a gerbe $B$ with connection, defining a class in $\check H^2(M;\Z)$.
%\end{itemize}
%These mix in a way giving a $2$-group symmetry\dots{} (will need some background info on $2$-groups).
%}
%\subsection{$2$-group symmetries}
\subsection{From the Bianchi identity to higher groups}
\label{higher_stuff}
In this section, we show that the twisted Bianchi identity~\eqref{twisted_Bianchi} is a natural consequence of
combining a principal $G$-bundle and a gerbe, each with connections, into a principal $\G$-bundle, where $\G$ is a
certain Lie $2$-group built from $G$ and $\mu$, together with additional data that we think of as a connection on
$\G$. First we introduce $2$-groups and their principal bundles; then, following~\cite{FSS12, SSS12}, we recover
the twisted Bianchi identity. As a result, we can precisely define the tangential structure for the
$\rE_8\times\rE_8$ heterotic string, i.e.\ the topological part of the data which, when put on a manifold $M$,
allows one to study $\rE_8\times\rE_8$ heterotic string theory on that manifold.%We will then compute the bordism
%groups of manifolds with this structure.

%The goal of this section is to interpret some of the fields for the $\E_8\times\E_8$ heterotic string as gauge
%fields for something called a \term{2-group}.
\begin{defn}
A \term{$2$-group} $\G$ is a group object in the bicategory of small categories.
\end{defn}
%There are several equivalent characterizations:
%\begin{enumerate}
%	\item A $2$-group is a category object in groups.
%	\item A $2$-group is a monoidal groupoid whose objects are invertible under tensor product.
%\end{enumerate}
\begin{defn}
A \term{Lie $2$-group} is a $2$-group $\G$ whose underlying category has been given the structure of a category
object in smooth manifolds.
\end{defn}
This means that the sets of objects and morphisms are smooth manifolds, and assignments such as the source of a map
or the composition of two maps are smooth. $2$-groups were first introduced by Hoàng Xuân Sính in her
thesis~\cite{Hoa75}, and Lie $2$-groups were introduced by Baez~\cite[\S 2]{Bae02}.

We call a $2$-group \term{strict} if it is strict as a monoidal category, i.e.\ its associators and unitors are all
identity maps. Mac Lane's coherence theorem~\cite[Chapter 7]{ML71} implies every $2$-group is equivalent to a
strict $2$-group, but the analogous statement is false for Lie $2$-groups; see \cref{strict_string}.
\begin{exm}
If $G$ is a group, it defines a monoidal groupoid with $G$ as its set of objects, tensor product $g\otimes
h\coloneqq gh$, and only the identity morphisms. This is a $2$-group, and inherits the structure of a Lie $2$-group
if $G$ is a Lie group.
\end{exm}
This procedure embeds the bicategory of groups, group homomorphisms, and identity $2$-morphisms into the bicategory
of $2$-groups, and we will therefore abuse notation and call this $2$-group $G$ again.
\begin{exm}
Let $A$ be an abelian group, and let $A[1]$ denote the monoidal groupoid with a single object $*$ and $\Hom_{A[1]}(*,
*) \coloneqq A$. This is a $2$-group, and if $A$ is Lie, $A[1]$ is a Lie $2$-group.
\end{exm}
It turns out every $2$-group $\G$ factors as an extension of these examples. Let $e$ be the identity object
of $\G$ and $\pi_0(\G)$ be the group of isomorphism classes of objects in $\G$. Then there is a short exact sequence
of $2$-groups
\begin{equation}
\label{2gp_extn}
	\shortexact{\Aut_{\G}(e)[1]}{\G}{\pi_0(\G)}.
\end{equation}
The Eckmann-Hilton theorem guarantees $\Aut_\G(e)$ is abelian. Extensions~\eqref{2gp_extn} are classified by the
data of:
\begin{enumerate}
	\item an action of $\pi_0(\G)$ on $\Aut_\G(e)$, and
	\item a cohomology class $k\in H^3(B\pi_0(\G); \Aut_\G(e))$, called the \term{$k$-invariant} of $\G$.
\end{enumerate}
When $\G$ has the discrete topology, this is unambiguous, but when $\G$ is a Lie $2$-group, one must be careful
what kind of cohomology is used here. The correct notion of cohomology is the \term{Segal-Mitchison
cohomology}~\cite{Seg70, Seg75} of $\pi_0(\G)$ valued in the abelian Lie group $\Aut_{\G}(e)$, as shown by
Schommer-Pries~\cite[Theorem 1]{SP11}.

Now we want to discuss principal $\G$-bundles. The idea is that if $G$ is a group, a principal $G$-bundle is a
submersion which is locally trivial, and whose fibers are $G$-torsors. For a Lie $2$-group $\G$, we need
the fibers to locally look like $\G$, meaning they must be categorified somehow.
\begin{defn}[{Bartels~\cite{Bar06}, Nikolaus-Waldorf~\cite[Definition 6.1.5]{NW13}}]
\label{2_bun_defn}
Let $\G$ be a Lie $2$-group. A \term{principal $\G$-bundle} over a smooth manifold $M$ is a Lie groupoid $\mathcal
P$ with a surjective submersion $\mathrm{obj}(P)\to M$ and a smooth right action $\rho$ of $\G$ on $\mathcal P$
such that the map
\begin{equation}
	(\mathrm{pr_1}, \rho)\colon \mathcal P\times\G\longrightarrow \mathcal P\times_M\mathcal P
\end{equation}
is a weak equivalence of Lie groupoids.
\end{defn}
See Nikolaus-Waldorf~\cite[\S 6]{NW13} for more details. The principal $\G$-bundles on a manifold $M$ form a
$2$-groupoid $\cat{Bun}_\G(X)$~\cite[Theorem 6.2.1]{NW13}.
\begin{defn}
Let $\G$ be a $2$-group, and let $C_\G$ be the bicategory with a single object $*$ and morphism category
$\Hom_{C_\G}(*, *)\coloneqq \G$. The \term{classifying space} of $\G$, denoted $B\G$, is the geometric realization
of the nerve of $C_\G$.\footnote{There are many different definitions of the nerve of a bicategory; the fact that
their geometric realizations are canonically homotopy equivalent is a theorem of
Carrasco-Cegarra-Garzón~\cite{CGG10}, allowing us to speak about $B\G$ without specifying which kind of
bicategorical nerve to use.}

When $\G$ is a Lie $2$-group, we make the same definition. This time $C_\G$ is a topological bicategory, so its nerve
is a simplicial space, and geometrically realizing, we obtain the space $B\G$.
\end{defn}
\begin{thm}[{Nikolaus-Waldorf~\cite[Theorems 4.6, 5.3.2, 7.1]{NW13}}]
\label{2_bun_class}
If $\G$ is a strict Lie $2$-group, then there is a natural equivalence $[X, B\G]\overset\simeq\to
\pi_0(\cat{Bun}_\G(X))$.
\end{thm}
% \TODO: does this generalize to Fréchet Lie 2-groups?
Nikolaus-Waldorf's proof builds on Baez-Stevenson's related but distinct characterization of $[X,
B\G]$~\cite[Theorem 1]{BS09} in terms of nonabelian Čech cohomology.

When $G$ is an ordinary group, if $G$ has the discrete topology, $BG$ has only one nonzero homotopy group, which is
$\pi_1(BG) = G$; likewise if $\G$ is a discrete $2$-group, $\pi_i(B\G)$ is nontrivial only for $i = 1,2$;
$\pi_1(B\G) = \pi_0(\G)$ and $\pi_2(B\G) = \Aut_\G(e)$. When $\G$ is a Lie
$2$-group, we have no control over its homotopy groups in general, just like $BG$ when $G$ is positive-dimensional.

If $\G$ has the discrete topology, the data classifying~\eqref{2gp_extn}, namely the action of $\pi_0(\G)$ on
$\Aut_\G(e)$ and the $k$-invariant, is equivalent to the Postnikov data of $B\G$, worked out by Mac
Lane-Whitehead~\cite{MW50}: this data classifies fibrations over $BG$ with fiber the Eilenberg-Mac Lane space
$K(\Aut_\G(e), 2)$. The total space of the fibration with this Postnikov data is homotopy equivalent to $B\G$.
\begin{exm}
\label{string_cover}
Let $G$ be a compact Lie group; then, the Segal-Mitchison cohomology group $H^3_{\mathrm{SM}}(G; \T)$ classifying
Lie $2$-group extensions of $G$ by $\T[1]$ is naturally isomorphic to $H^4(BG;\Z)$~\cite[Corollary 97]{SP11}.
Therefore given a class $\mu\in H^4(BG;\Z)$, we obtain a Lie $2$-group $\cS(G, \mu)$ fitting into a central
extension
\begin{equation}
\label{string_SES}
	\shortexact{\T[1]}{\cS(G, \mu)}{G},
\end{equation}
which is sometimes called the \term{string $2$-group} or \term{string cover} associated to $G$ and $\lambda$. Of
all the string covers, the most commonly studied one is $\String_n\coloneqq \cS(\Spin_n, \lambda)$, which is called 
\emph{the} string $2$-group.
\end{exm}
This class of $2$-groups was first studied by Baez-Lauda~\cite[\S 8.5]{BL04}.

The sequence~\eqref{string_SES} implies that upon taking classifying spaces,
\begin{equation}
\label{2gp_fiber}
	B\G\longrightarrow BG\overset{\mu}{\longrightarrow} K(\Z, 4)
\end{equation}
is a fibration.
\begin{rem}
\label{strict_string}
\Cref{2_bun_class} classified principal $\G$-bundles when $\G$ is a strict $2$-group, but it is a theorem of
Baez-Lauda~\cite[Corollary 60]{BL04} that there is no strict Lie $2$-group model for $\cS(G, \mu)$ when $G$ is
simply connected and $\mu\ne 0$. However, there is a fix: in the setting of \term{Fréchet Lie $2$-groups}, where we
allow the spaces of objects and morphisms of $\G$ to be Fréchet manifolds, there is a strict model for $\cS(G,
\mu)$~\cite{BSCS07, LW23}, so $B\cS(G, \mu)$ actually classifies principal $\cS(G, \mu)$-bundles. This suffices for
studying bordism groups.
\end{rem}
Following Sati-Schreiber-Stasheff~\cite{SSS12}, we now relate $\cS(G, \mu)$ to the twisted Bianchi identity for $G$
and $\mu$. To do so, we use the language of stacks and differential cohomology, following~\cite{HS05, FH13, Sch13,
BNV16, ADH21}. Make the category $\cat{Man}$ into a site by defining the covers to
be surjective submersions, and define a \term{stack} to be a functor of $\infty$-categories
$\cat{Man}^{\cat{op}}\to\cat{Top}$ which satisfies descent for hypercovers. This defines a presentable
$\infty$-category $\cat{St}$ of stacks~\cite[Proposition 6.5.2.14]{Lur09b}, and the Yoneda embedding $h\colon
\cat{Man}\to\cat{St}$ embeds $\cat{Man}$ as a full subcategory. We will often simply write $M$ for the stack
$h(M)$; we never compare these two notions directly, so this will not introduce confusion.

For any space $X$, the functor $\mathrm{Map}(\bl, X)\colon\cat{Man}\to\cat{Top}$ is a sheaf, and this procedure
defines a functor of $\infty$-categories $\Gamma^*\colon\cat{Top}\to\cat{St}$. The values of
the stacks produced by $\Gamma^*$ evaluated on manifolds $M$ are homotopy-invariant in $M$. $\Gamma^*$ has a left
adjoint $\Gamma_\sharp\colon\cat{St}\to\cat{Top}$ (see Dugger~\cite[Proposition 8.3]{Dug01},
Morel-Voevodsky~\cite[Proposition 3.3.3]{MV99}, and~\cite[Proposition 4.3.1]{ADH21}); $\Gamma_\sharp(\mathbf X)$
for a stack $\mathbf X$ can be thought of as the best approximation to $\mathbf X$ by a stack whose values on
manifolds are homotopy-invariant.

Let $\Delta_{\mathrm{alg}}^n\coloneqq\set{(t_0,\dotsc,t_n)\mid t_0 + \dotsb + t_n = 1}\subset\R^{n+1}$. These
``algebraic $n$-simplices'' assemble into a cosimplicial manifold $\Delta_{\mathrm{alg}}^\bullet$,
and~\cite[Corollary 5.1.4]{ADH21} there is a natural homotopy equivalence $\Gamma_\sharp(\mathbf X)\simeq
\abs{\mathbf X(\Delta_{\mathrm{alg}}^\bullet)}$, where as usual $\abs{\bl}$ denotes geometric realization.

Thus, for a manifold $M$, there is a natural homotopy equivalence $\Gamma_\sharp(M)\overset\simeq\to
M$, so a map $M\to\mathbf X$ naturally induces a
map $M\to \Gamma_\sharp(\mathbf X)$.
\begin{lem}
\label{shape_pullback}
Suppose $\mathbf X\to \mathbf Y\gets\mathbf Z$ is a diagram in $\cat{St}$, and that $\mathbf
Y(\Delta_{\mathrm{alg}}^n)$ and $\mathbf Z(\Delta_{\mathrm{alg}}^n)$ are connected for all $n$. Then
\begin{equation}
	\Gamma_\sharp(\mathbf X\times_{\mathbf Y}\mathbf Z)\simeq \Gamma_\sharp(\mathbf X)\times_{\Gamma_\sharp(\mathbf
	Y)} \Gamma_\sharp(\mathbf Z).
\end{equation}
%$\Gamma_\sharp$ commutes with pullbacks.
\end{lem}
\begin{proof}
Pullbacks of sheaves can be computed pointwise, then sheafifying, so given a pullback
$\mathbf X\to\mathbf Y\gets\mathbf Z$ in $\cat{St}$, for each $n\ge 0$ the pullback of
\begin{equation}\label{LHP}
\mathbf X(\Delta_{\mathrm{alg}}^n) \longrightarrow \mathbf Y(\Delta_{\mathrm{alg}}^n) \longleftarrow
\mathbf Z(\Delta_{\mathrm{alg}}^n)
\end{equation}
is $(\mathbf X\times_{\mathbf Y}\mathbf Z)(\Delta_{\mathrm{alg}}^n)$. The Bousfield-Friedlander theorem~\cite{BF78,
Bou01} implies that, given the hypotheses on $\mathbf Y$ and $\mathbf Z$ in the theorem statement, the homotopy
pullback of the geometric realizations of $\mathbf X$, $\mathbf Y$, and $\mathbf Z$ is the geometric realization of
the levelwise homotopy pullback~\eqref{LHP} (see~\cite[p.\ 14-9]{War20} for this specific consequence of the
Bousfield-Friedlander theorem).
\end{proof}
\begin{exm}
For $G$ a Lie group, there is a stack $B_\nabla G$ whose value on a manifold $M$ is the geometric realization of
the nerve of the groupoid of principal $G$-bundles on $M$ with connection~\cite{FH13}. This object is denoted
$\mathbf B G_{\mathrm{conn}}$ in~\cite{FSS12, SSS12, Sch13}, $\mathbb BG^\nabla$ in~\cite[\S 5]{BNV16}, and
$\mathrm{Bun}_G^\nabla$ in~\cite{ADH21}.

There is a natural homotopy equivalence $\Gamma_\sharp(B_\nabla G)\overset\simeq\to BG$~\cite[Corollary
13.3.29]{ADH21}, which can be interpreted as forgetting from a principal bundle with connection to a principal
bundle.
\end{exm}
\begin{exm}
For $k\ge 0$, there is a stack $B_\nabla^k \T$ whose value on a manifold $M$ is the geometric realization of the
nerve of the $\infty$-groupoid of cocycles for the differential cohomology group $\check H^{k+1}(M;\Z)$. This
object is studied in~\cite{FSS12, SSS12, Sch13}, where it is denoted $\mathbf B^k U(1)_{\mathrm{conn}}$.
\end{exm}
\begin{lem}
There is a homotopy equivalence $\Gamma_\sharp(B_\nabla^k\T)\simeq K(\Z, k+1)$.
\end{lem}
\begin{proof}
Schreiber~\cite[Observation 1.2.134]{Sch13} produces the following pullback square in $\cat{St}$:
% https://q.uiver.app/?q=WzAsNCxbMCwwLCJCX1xcbmFibGFea1xcVCJdLFsxLDAsIlxcT21lZ2Ffe1xcbWF0aGl0e2NcXGVsbH19XntrKzF9Il0sWzEsMSwiSyhcXFIsIGsrMSkiXSxbMCwxLCJLKFxcWiwgaysxKSJdLFszLDJdLFswLDNdLFswLDFdLFsxLDJdLFswLDIsIiIsMSx7InN0eWxlIjp7Im5hbWUiOiJjb3JuZXIifX1dXQ==
\begin{equation}
\label{diffcoh_diagram}
\begin{tikzcd}
	{B_\nabla^k\T} & {\Omega_{\mathit{c\ell}}^{k+1}} \\
	{K(\Z, k+1)} & {K(\R, k+1),}
	\arrow[from=2-1, to=2-2]
	\arrow[from=1-1, to=2-1]
	\arrow[from=1-1, to=1-2]
	\arrow[from=1-2, to=2-2]
	\arrow["\lrcorner"{anchor=center, pos=0.125}, draw=none, from=1-1, to=2-2]
\end{tikzcd}
\end{equation}
where $\Omega_{\mathit{c\ell}}^{k+1}$ is the stack of closed $(k+1)$-forms. For that stack and $K(\R, n+1)$, the
values on each $\Delta_{\mathrm{alg}}^n$ are connected spaces, so \cref{shape_pullback} identifies
$\Gamma_\sharp(B_\nabla^k\T)\simeq K(\Z, k+1)\times_{K(\R, k+1)} \Gamma_\sharp(\Omega_{\mathit{c\ell}}^{k+1})$. To
finish, observe that, essentially by the de Rham theorem, the map
$\Omega_{\mathit{c\ell}}^{k+1}\to K(\R, k+1)$ passes to a homotopy equivalence after applying $\Gamma_\sharp$. This
follows from~\cite[Lemma 7.15]{BNV16} together with the Dold-Kan theorem.
\end{proof}
These stacks are the universal setting for the Chern-Weil map.
\begin{thm}[{Cheeger-Simons~\cite[Theorem 2.2]{CS85}, Bunke-Nikolaus-Völkl~\cite[\S 5.2]{BNV16}}]
\label{CW_lifts}
Let $G$ be a compact Lie group and $c\in H^k(BG;\Z)$, where $k$ is even. Then there is a map $\check c\colon
B_\nabla G\to B^{k-1}_\nabla \T$ natural in $(G, c)$ such that for any manifold $M$ and map $f\colon M\to B_\nabla
G$, interpreted as a principal $G$-bundle $P\to M$ with connection $\Theta$,
\begin{enumerate}
	\item if $\mathit{char}\colon \check H^*(\bl;\Z)\to H^*(\bl;\Z)$ denotes the characteristic class map, then
	$\mathit{char}(\check c\circ f) = c(P)$, and
	\item if $\mathit{curv}\colon \check H^*(\bl;\Z)\to \Omega^*_{\mathit{c\ell}}$ denotes the curvature map, then
	$\mathit{curv}(\check c\circ f) = \CW_c(\Theta)$.
\end{enumerate}
\end{thm}
Cheeger-Simons lifted the Chern-Weil map to differential cohomology; Bunke-Nikolaus-Völkl recast it in terms of
$B_\nabla G$. The map $\mathit{char}$ in \cref{CW_lifts} is the map down the left of the
square~\eqref{diffcoh_diagram}; $\mathit{curv}$ is the map across the top of~\eqref{diffcoh_diagram}.
\begin{defn}[{Fiorenza-Schreiber-Stasheff~\cite[\S 6.2]{FSS12}}]
Given a compact Lie group $G$ and a class $\mu\in H^4(BG;\Z)$, let $\mathbf{BStr}(G, \mu)$ denote the fiber of the map
$\check\mu\colon B_\nabla G\to B_\nabla^3 \T$.
\end{defn}
We will see momentarily that maps to $\mathbf{BStr}(G, \mu)$ lead to solutions to the twisted Bianchi identity for $G$
and $\mu$.
\begin{prop}
\label{2shape}
There is a natural homotopy equivalence $\Gamma_\sharp(\mathbf{BStr}(G, \mu))\simeq B\cS(G, \mu)$.
\end{prop}
For this reason we think of $\mathbf{BStr}(G, \mu)$ as the classifying stack of principal $\cS(G, \mu)$-bundles
with connection, though this is only a heuristic.\footnote{There are at least \emph{five} notions of a connection
on principal $\G$-bundles for $\G$ a $2$-group: three are discussed by Waldorf~\cite[\S 5]{Wal18}, a fourth by
Rist-Saemann-Wolf~\cite{RSW22}, and a fifth, defined only for $\G = \String_n$, by Waldorf~\cite{Wal13}.}
\begin{proof}
Apply \cref{shape_pullback} to the diagram
\begin{equation}
	B_\nabla G\overset{\check\mu}{\longrightarrow} B^3_\nabla \T \longleftarrow *,
\end{equation}
as the values of both $*$ and $B^3_\nabla\T$ are connected on $\Delta_{\mathrm{alg}}^n$ for each $n$. This implies
that $\Gamma_\sharp(\mathbf{BStr}(G, \mu))$ is the fiber of $\mu\colon BG\to K(\Z, 4)$, which we identified with $B\G$
in~\eqref{2gp_fiber}.
\end{proof}
\begin{prop}[{Fiorenza-Schreiber-Stasheff~\cite[\S 6.3]{FSS12}}]
\label{2gp_satisfies_bianchi}
Let $G$ be a compact Lie group, $U\subset\R^n$ be an open set, and $P\to U$ be a principal $G$-bundle with connection
$\Theta$. A lift of the corresponding map $f_{P,\Theta}\colon U\to B_\nabla G$ to a map $\widetilde
f_{P,\Theta}\colon U\to \mathbf{BStr}(G, \mu)$ induces a form $H\in\Omega^3(U)$ such that $H$ and $\Theta$ satisfy the
twisted Bianchi identity~\eqref{twisted_Bianchi}.
\end{prop}
The idea here is that we have specified a trivialization of the differential characteristic class $\check \mu(P,
\theta)$. Applying the curvature map $\mathit{curv}\colon B_\nabla^3\T\to\Omega^4_{\mathit{c\ell}}$, we have also
specified a trivialization of $\CW_\mu(\Theta)$, which locally is the data $H$ showing that $\CW_\mu(\Theta)$ is exact.

A map to $\mathbf{BStr}(G, \mu)$ is more data than what we get from \cref{2gp_satisfies_bianchi}, as we have
trivialized not just the Chern-Weil form, but also the differential characteristic class. This can be interpreted
as saying the data $H$ specifying the trivialization is quantized to form a twisted version of a gerbe with
connection.
%\begin{rem}
%Let us dig into the idea that the data of a trivialization is a twisted gerbe with connection. Just as isomorphism
%classes of principal $\T$-bundles are classified by $H^2(\bl;\Z)$, and equivalence classes of gerbes are classified
%by $H^3(\bl;\Z)$, isomorphism classes of principal $\T$-bundles with connection are classified by $\check
%H^2(\bl;\Z)$, and equivalence classes of gerbes with connection are classified by $\check H^3(\bl;\Z)$~\cite[\S
%5.3]{Bry93}. So $B_\nabla^2\T$ is a classifying stack for gerbes with connection. Since $\pi\colon \mathbf{BStr}(G,
%\mu)\to B_\nabla G$ is the fiber of a map to $B_\nabla^3\T$, the fiber of $\pi$ is $B_\nabla^2\T$, meaning that
%$\mathbf{BStr}(G, \mu)$ mixes the data of a gerbe and a principal $G$-bundle with connection in a nontrivial way.
%\end{rem}

To summarize, given a map $M\to \mathbf{BStr}(G, \mu)$, the stack which we think of as modeling $\cS(G,\mu)$-bundles
with connection, we obtain:
\begin{enumerate}
	\item a principal $\cS(G, \mu)$-bundle $\cP\to M$ by \cref{2shape}, and
	\item a ``twisted gerbe with connection,'' i.e.\ local data of a gerbe $Q\to M$ such that $\Omega_Q$ and the
	$G$-connection $\Theta$ induced by the map $\mathbf{BStr}(G, \mu)\to B_\nabla G$ satisfy the twisted Bianchi
	identity~\eqref{twisted_Bianchi} by \cref{2gp_satisfies_bianchi}.
\end{enumerate}
Motivated by this, we define
of the tangential structure for the $\rE_8\times\rE_8$ heterotic string. This first appears in~\cite[\S
3.2]{SSS12}, with~\cite{Sat11, FSS15} considering some related examples.
\begin{defn}
\label{diff_het}
Let $G\coloneqq (\rE_8\times\rE_8)\rtimes\Z/2$. A \term{differential $\xihetn$-structure} on a manifold $M$ is the
following data:
\begin{enumerate}
	\item a Riemannian metric and spin structure on $M$,
	\item a principal $G$-bundle $P\to M$ with connection $\Theta$, and
	\item a lift of
	\begin{equation}
		((B_{\mathrm{Spin}}(M), \Theta^{\mathrm{LC}}), (P, \Theta))\colon M\longrightarrow
		B_\nabla(\Spin_n\times G)
	\end{equation}
	to a map $M\to \mathbf{BStr}(\Spin_n\times G, c_1 + c_2 - \lambda)$.
\end{enumerate}
Here $B_{\Spin}(M)\to M$ is the principal $\Spin_n$-bundle of frames of $M$, and $\Theta^{\mathrm{LC}}$ denotes its
Levi-Civita connection.
\end{defn}
For bordism groups we want the topological version of this.
\begin{defn}
\label{symm_defn}
A \term{tangential structure} is a space $B$ and a map $\xi\colon B\to B\O$. Given a tangential structure $\xi$, a
\term{$\xi$-structure} on a virtual vector bundle $E\to X$ is a lift of the classifying map $f_E\colon X\to B\O$ to
a map $\widetilde f_E\colon X\to B$ such that $\xi\circ \widetilde f_E = f_E$. A \term{$\xi$-structure} on a
manifold $M$ is a $\xi$-structure on its tangent bundle.

We make the analogous definition with maps $\xi_n\colon B_n\to B\O_n$; in this case, we only refer to
$\xi_n$-structures on $n$-manifolds.
\end{defn}
Lashof~\cite{Las63} defined bordism groups $\Omega_*^\xi$ of manifolds with $\xi$-structure, and Boardman~\cite[\S
V.1]{Boa65} defined a Thom spectrum $\mathit{MT\xi}$ whose homotopy groups are naturally isomorphic to
$\Omega_*^\xi$ via the Pontrjagin-Thom construction.\footnote{In homotopy theory, it is common to study the Thom
spectra $\mathit{M\xi}$ representing $\xi$-structures on the stable \emph{normal} bundle $\nu_M$ of a manifold $M$,
and indeed many of the results we cite about $\MTSO$, $\MTString$, etc.\ are stated for $\mathit{MSO}$,
$\mathit{MString}$, etc., or about Thom spectra $\mathit{M\xi}$ in general. This is not a problem: for any
tangential
structure $\xi$, there is a tangential structure $\xi^\perp$ such that a $\xi$-structure on $TM$ is equivalent data to a
$\xi^\perp$-structure on $\nu_M$ and vice versa, so that $\mathit{MT\xi}\simeq\mathit{M\xi^\perp}$, so the general
theory is the same. And for $\xi = \O$, $\SO$, $\Spin$, $\Spin^c$, and $\String$, $\xi\simeq\xi^\perp$ and in those
cases we can ignore the difference between $\mathit{M\xi}$ and $\mathit{MT\xi}$.} We think of the category of
tangential structures as the category of spaces over $B\O$, and bordism groups and Thom spectra are functorial in this
category. That is, taking bordism groups and Thom spectra is functorial as long as one commutes with the map down
to $B\O$.

The following definition is a special case of a definition due to Sati-Schreiber-Stasheff~\cite[Definition
2.8]{SSS12}. See~\cite{Sat11, FSS15, FSS21} for other related examples.
\begin{defn}
\label{top_het}
Let $G_n \coloneqq \Spin_n \times (\rE_8\times\rE_8)\rtimes\Z/2$ and
\begin{equation}
\label{gnhet_defn}
	\Ghet_n \coloneqq \cS(G_n, c_1 + c_2 -
\lambda).
\end{equation}
The \term{$\rE_8\times\rE_8$ heterotic tangential structure} is the tangential structure
\begin{equation}
	\xihetn\colon B\Ghet_n\longrightarrow B\Spin_n\longrightarrow B\O_n,
\end{equation}
where the first map comes from the quotient of $\Ghet$ by $\T[1]$, followed by projection onto the $\Spin_n$ factor
in $G_n$. We also define $\Ghet$ and $\xihet$ analogously by stabilizing in $n$.
\end{defn}
In other words: a differential $\xihetn$-structure is a lift of a map to $B_\nabla (\Spin\times G)$ to
$\mathbf{BStr}(G, \mu)$; by \cref{2shape}, a topological $\xihetn$-structure is the image of this data under
$\Gamma_\sharp$. In particular, a $\xihetn$-structure on an $n$-manifold $M$ includes data of a principal $\Ghet_n$-bundle $\mathcal P\to
M$.

Taking the quotient of $\Ghet$ by $\T[1]$ induces a map of tangential structures
\begin{equation}
\label{forget_string}
	\phi\colon B\Ghet\longrightarrow B\Spin\times B(\rE_8^2\rtimes\Z/2).
\end{equation}
Thus, much like a \spinc manifold $M$ has an associated $\T$-bundle $P$ with $c_1(P)\bmod 2 = w_2(M)$, a
$\xihet$-manifold has associated $(\rE_8^2\rtimes\Z/2)$-bundle $P$. From this perspective, a $\xihet$-structure on
a manifold $M$ is the following data:
\begin{itemize}
	\item a spin structure on $M$,
	\item a double cover $\pi\colon \widetilde M\to M$,
	\item two principal $\rE_8$-bundles $P,Q\to \widetilde M$ which are exchanged by the nonidentity deck
	transformation of $\pi$, and
	\item a trivialization of the class $\lambda(M) - (c(P) + c(Q))\in H^4(M;\Z)$.
\end{itemize}
By a \term{trivialization} of a cohomology class $\alpha\in H^n(M; A)$ we mean a null-homotopy of the classifying
map $f_\alpha\colon M\to K(A, n)$. Thus orientations are identified with trivializations of $w_1$, etc. To make the
trivialization of $\lambda(M) - (c(P) + c(Q))$ precise, we have to descend the class $c(P) + c(Q)$, a priori an
element of $H^4(\widetilde M;\Z)$, to $H^4(M;\Z)$. We can do this because, as noted in \cref{char_class_e8}, the
class $c_1 + c_2$ descends through the Serre spectral sequence to the base.
%Quotienting further by $\rE_8^2$, we obtain a
%double cover $\widetilde M\to M$; it is possible to think of $P$ as the data of two $\rE_8$-bundles on $\widetilde
%M$ which are exchanged by the nonidentity deck transformation.
%\TODO: what is a $\xihet$-structure? First $\phi$ tells us the two $\rE_8$-bundles on the orientation double cover;
%then we have the trivialization of the degree-$4$ characteristic class
\begin{rem}
\label{xihet_char_class}
We can combine some the data of a $\xihet$ structure on $M$ into a twisted characteristic class. Let $\Z^\sigma$ be
the $\Z[\Z/2]$-module isomorphic to $\Z^2$ as an abelian group, and in which the nontrivial element of $\Z/2$ swaps
the two factors. Then, let $\Z^\sigma_\pi$ denote the local system on $M$ which is the associated bundle
$\widetilde M\times_{\Z/2} \Z^\sigma$. A pair of classes $x,y\in H^k(\widetilde M;\Z)$ exchanged by the deck
transformation thus define a class in $H^k(M; \Z^\sigma_\pi)$, so, the classes $c(P)$ and $c(Q)$ in
$H^4(\widetilde M;\Z)$ together define a class $\widetilde c(P, Q)\in H^4(M;\Z^\sigma_\pi)$, which is a characteristic
class of an $((\rE_8\times\rE_8)\rtimes\Z/2)$-bundle.

If $\Z$ denotes the $\Z[\Z/2]$-module isomorphic to $\Z$ as an abelian group and with trivial $\Z/2$-action, then
taking the quotient of $\Z^\sigma$ by the submodule generated by $(1, -1)$ defines a map of $\Z[\Z/2]$-modules $q\colon
\Z^\sigma\to\Z$, hence also a map between the corresponding twisted cohomology groups, and this map sends
$\widetilde c(P, Q)\mapsto c(P) + c(Q)$. Therefore one could recast a $\xihet$-structure on a spin manifold $M$ as
the data of a principal $((\rE_8\times\rE_8)\rtimes\Z/2)$-bundle $(P,Q,\pi)$ together with a trivialization of
$\lambda(M) - q(\widetilde c(P, Q))$.

Bott-Samelson~\cite[Theorems IV, V(e)]{BS58} showed that the map $B\rE_8\to K(\Z, 4)$ defined by the characteristic
class $c$ is $15$-connected. This implies that up to isomorphism, a principal
$((\rE_8\times\rE_8)\rtimes\Z/2)$-bundle on a manifold of dimension $15$ or lower is equivalent data to its
characteristic class $\widetilde c$.
\end{rem}
\begin{rem}
\label{no_Z2}
One might want to simplify by restricting to the special case where $\pi\colon\widetilde M\to M$ is trivial (as done
in, e.g.,~\cite{Wit86}), in which case the data of a $\xihet$-structure simplifies to the data of a spin structure
on $M$, two principal $\rE_8$-bundles $P,Q\to M$, and a trivialization of $\lambda(M) - c(P) - c(Q)$.  This
corresponds to the tangential structure $\xi^{r,\mathrm{het}}\colon B\cS(\Spin\times \rE_8\times\rE_8, c_1 + c_2 -
\lambda)\to B\Spin\to B\O$.
\end{rem}

%\TODO: principal $\G$-bundles and connections,and they're maps to $B\G$. Ideally, I'd express all these things and
%currents and the like, and match the physics and math words!
%
%Forgetting about $\tr(R^2)$ for a moment, the Bianchi identity describes for us a $2$-group structure mixing
%$B\U_1$ and $G$. An extension
%
%There are various simplifications one can make to this tangential structure.
%\begin{enumerate}
%	\item One can ignore the $\Z/2$-action switching the two $\E_8$-bundles.
%	\item One can ``turn off'' the $B$-field, i.e.\ assume that it is zero.
%\end{enumerate}
%

%\subsection{From the Bianchi identity to $2$-group connections}
%\input{bianchi}

\subsection{The CHL string}
\label{CHL_intro}
Eleven-dimensional $\mathcal N = 1$ supergravity admits a time-reversal symmetry, allowing it to be defined on
\pinp $11$-manifolds.\footnote{In addition to the \pinp structure, one needs the additional data of a lift of
$w_4(TM)$ to $w_1(TM)$-twisted integral cohomology. See~\cite{Wit97, Wit16, FH21b}.} Therefore we can compactify it
on a Möbius strip with certain boundary data to obtain a nine-dimensional supergravity theory; the goal of this
subsection is to determine the tangential structure of this theory. Eleven-dimensional $\mathcal N = 1$
supergravity is expected to be the low-energy limit of a theory called M-theory,\footnote{M-theory is expected to
require additional data on top of the tangential structure described above for $11$-dimensional $\mathcal N = 1$
supergravity. See~\cite[Table 1]{FSS20} and the references listed there.} and compactifying M-theory on the Möbius
strip is expected to produce a string theory called the \term{Chaudhuri-Hockney-Lykken (CHL) string}~\cite{CHL95}
whose low-energy limit is the $9$-dimensional supergravity theory described above; we study the tangential
structure of this supergravity theory in this subsection with the aim of also learning about the CHL string.

However, we do not want our perspective on the CHL string to be overly one-sided. There is another way to produce
the CHL string by compactifying: consider the circle with its nontrivial principal $\Z/2$-bundle $P\to S^1$. Via
the map $\Z/2\inj \Spin\times ((\rE_8\times\rE_8)\rtimes\Z/2)$, this bundle defines a
$\Spin\times((\rE_8\times\rE_8)\rtimes\Z/2)$-structure on $S^1$ for which $\lambda$ and $c_1 + c_2$ are both
trivial, so this structure lifts to define a $\xihet$-structure on $S^1$. We will call the circle with this
$\xihet$-structure $\RP^1$, as $S^1\cong\RP^1$ as manifolds and the $\xihet$-structure comes from the double cover
$S^1\to\RP^1$. The CHL string is precisely what one obtains by compactifying the $\rE_8^2$ heterotic string on
$\RP^1$.

We want to determine the tangential structure $\xiCHL$ such that the product of $\RP^1$ with a manifold with
$\xiCHL$-structure has an induced $\xihet$-structure. In general, keeping track of how the tangential structure
changes under compactification can be subtle; for a careful analysis, see Schommer-Pries~\cite[\S 9]{SP18}. But for
the CHL string, we can get away with a more ad hoc approach: following Chaudhuri-Polchinski~\cite{CP95} (see
also~\cite[\S 2.2.1]{dBDHKMMS00}) we restrict to the case where the principal $\Z/2$-bundle on $\RP^1\times M$
obtained by the quotient map~\eqref{forget_string} is the pullback of the Möbius bundle $S^1\to\RP^1$ along the
projection $\mathrm{pr}_1\colon \RP^1\times M\to\RP^1$.
\begin{prop}
Let $M$ be a spin manifold and $P\to M$ be a principal $\rE_8$-bundle. The data of a trivialization $\mathfrak s$
of $\lambda(M) - 2c(P)$ induces a $\xihet$-structure on $\RP^1\times M$ whose associated principal $\Z/2$-bundle is
the Möbius bundle $S^1\times M\to\RP^1\times M$. Moreover, if $\dim(M)\le 14$, this assignment is a natural
bijection from the set of isomorphism classes of data $(P, \mathfrak s)$ to the set of $\xihet$-structures on
$\RP^1\times M$ whose associated $\Z/2$-bundle is $S^1\times M\to \RP^1\times M$.
\end{prop}
\begin{proof}
Let $\pi\colon S^1\times M\to M$ be the projection onto the second factor. Given $P\to M$ and $\mathfrak s$, the
pair of $\rE_8$-bundles $(\pi^*P, \pi^*P)\to S^1\times M$ are exchanged by the deck transformation for $S^1\times
M\to\RP^1\times M$, and $(c_1+c_2)$ evaluated on the pair $(\pi^*P, \pi^*P)$ is $2c(P)\in H^4(\RP^1\times M;\Z/2)$.
Choosing the string structure on $\RP^1$ induced from the bounding framing, we obtain a canonical trivialization of
$\lambda(\RP^1\times M) - \lambda(M)\in H^4(\RP^1\times M;\Z)$ from the two-out-of-three property of string
structures. Putting all of this together, we see that we have data of two $\rE_8$-bundles on $S^1\times M$
exchanged by the deck transformation, and a trivialization of $\lambda - (c_1 + c_2)$ on $\RP^1\times M$, thus
defining a $\xihet$-structure as claimed.

To see that this produces all $\xihet$-structures associated with $S^1\times M\to\RP^1\times M$, recall from
\cref{xihet_char_class} that the
$((\rE_8\times\rE_8)\rtimes\Z/2)$-bundle associated to a $\xihet$-structure is classified by a characteristic class
in twisted cohomology. The assumption that the associated $\Z/2$-bundle is $S^1\times M\to\RP^1\times M$ implies
this class belongs to $H^4(\RP^1\times M; \underline{\Z\oplus\Z})$, where a generator of $\pi_1(\RP^1)$ acts on
$\underline{\Z\oplus\Z}$ by swapping the two factors, and $\pi_1(M)$ acts trivially. The twisted Künneth
formula~\cite[Theorem 1.7]{Gre06} gives us an isomorphism
\begin{equation}
	H^4(\RP^1\times M; \underline{\Z\oplus\Z}) \overset\cong\longrightarrow H^4(M; \Z),
\end{equation}
meaning that the pair of $\rE_8$-bundles on the orientation double cover $S^1\times M$ pull back from bundles on
$M$, which must be isomorphic in order to be exchanged by the $\Z/2$-action.
\end{proof}
%This considerably simplifies the situation: it means the pair of associated $\rE_8$-bundles $P,Q\to \RP^1\times M$
%are isomorphic. It is also typical to assume that these $\rE_8$-bundles are trivial on $\RP^1$, i.e.\ they pull
%back from $\rE_8$-bundles on $M$.
The Bianchi identity corresponding to this data can therefore be simplified to use a single bundle $P\to M$ and the
class $c(P) + c(P)$: we obtain
\begin{equation}
	\d H = \CW_{2c}(\Theta_P) - \CW_\lambda(\Theta^{\mathrm{LC}}),
%	\d B = \Omega_Q - \CS_{2c}(\Theta_P) + \CS_\lambda(\Theta^{\mathrm{LC}}),
\end{equation}
i.e.\ the twisted Bianchi identity for $G = \Spin\times\rE_8$ and $\mu = 2c-\lambda$. Then, following
\cref{top_het,diff_het}, we make the following definitions.
\begin{defn}
\label{diff_CHL}
A \term{differential $\xiCHLn$-structure} on a manifold $M$ is the
following data:
\begin{enumerate}
	\item a Riemannian metric and spin structure on $M$,
	\item a principal $\rE_8$-bundle $P\to M$ with connection $\Theta$, and
	\item a lift of
	\begin{equation}
		((B_{\mathrm{Spin}}(M), \Theta^{\mathrm{LC}}), (P, \Theta))\colon M\longrightarrow
		B_\nabla(\Spin_n\times \rE_8)
	\end{equation}
	to a map $M\to \mathbf{BStr}(\Spin_n\times \rE_8, 2c - \lambda)$.
\end{enumerate}
\end{defn}
What we call $B\GCHL_n$
coincides with Sati-Schreiber-Stasheff's $B\String^{2a}$~\cite[(2.18), \S 2.3.3]{SSS12} and also appears in work of
Fiorenza-Sati-Schreiber~\cite[Remark 4.1.1]{FSS15}, though those papers do not discuss its relationship with
the CHL string.
\begin{defn}[{Sati-Schreiber-Stasheff~\cite[(2.18), \S 2.3.3]{SSS12}}]
\label{top_CHL}
Let
\begin{equation}
	\GCHL_n \coloneqq \cS(\Spin_n\times\rE_8, 2c - \lambda).
\end{equation}
The \term{CHL tangential structure} is the
tangential structure
\begin{equation}
	\xiCHLn\colon B\GCHL_n\longrightarrow B\Spin_n\longrightarrow B\O_n,
\end{equation}
where the first map comes from the quotient of $\GCHL$ by $\T[1]$, followed by projection onto the $\Spin_n$ factor.
Stabilizing in $n$, we also obtain $\GCHL$ and a tangential structure $\xiCHL$.
\end{defn}
A $\xiCHL$-structure on an $n$-manifold $M$ in particular comes with data of a principal $\GCHL_n$-bundle $\mathcal
P\to M$,
%\begin{defn}[Sati-Schreiber-Stasheff (\textit{ibid.})]
%\label{diff_CHL}
%A \term{differential $\xiCHLn$-structure} is a
%Riemannian metric and a $\xiCHLn$-structure on $M$ together with a connection $\Theta_{\mathcal P}$ lifting the
%Levi-Civita connection on the frame bundle of $M$.
%\end{defn}
and can be formulated as the data of a principal $\rE_8$-bundle $P\to M$ and a trivialization of
$\lambda(M) - 2c(P)\in H^4(M;\Z)$.
\begin{rem}
\label{spinw4}
Since a $\xiCHL$ structure includes data identifying $\lambda$ as twice another class, it induces a trivialization
of the mod $2$ reduction of $\lambda$, which is $w_4$. That is, a $\xiCHL$ structure induces a
$\Spin\ang{w_4}$ structure, where $B\Spin\ang{w_4}$ is the homotopy fiber of $w_4\colon B\Spin\to K(\Z, 4)$. This
structure has been studied in, e.g.~\cite{Wit97, KS04, FH21b} for applications to M-theory.
\end{rem}
\begin{rem}[Variation of the tangential structure along the moduli space]
There is a moduli space of CHL string theories, not just one, and the gauge group depends on
where in the moduli space one is; this moduli space was first studied by Chaudhuri-Polchinski~\cite{CP95}. At a
generic point, the gauge group is broken to $\T^8$, and at various special points the gauge group enhances to
$\rE_8$ or other nonabelian groups: see~\cite[Table 3]{FFGNP21}. We work only at the $\rE_8$ point of the moduli
space in this paper; it would be interesting to apply the techniques in this paper to other points in the CHL
moduli space.

There has been quite a bit of recent research studying the moduli spaces of compactifications of the
$\rE_8\times\rE_8$ heterotic string and the CHL string, and investigating which gauge groups can
occur~\cite{FGN18, CDLZ20, FFGNP20, CDLZ21, FFGNP21, FP21, MV21, CDLZ22, CGH22, CMM22, FP22, Fre22, MF22}.
\end{rem}
%\begin{rem}[A categorical symmetry]
%
%\end{rem}

%% file: computations_section.tex
\label{s_computations}
Now it is time to compute. We will use the Adams spectral sequence to compute $\Omega_*^\xihet$ and
$\Omega_*^\xiCHL$; this is a standard tool in computational homotopy theory and more recently appears frequently in
the mathematical physics literature, and we point the interested reader to Beaudry-Campbell's introductory
article~\cite{BC18}.

Applications of the Adams spectral sequence to mathematical physics questions tend to follow the same formula.
Suppose that we want to compute $\Omega_*^\xi$ for some tangential structure $\xi$.
\begin{enumerate}
	\item First, express $\xi$ as a ``twisted $\xi'$-structure,'' where $\xi'$ is one of $\SO$, $\Spin$, $\Spin^c$,
	or $\String$: prove that a $\xi$-structure on a vector bundle $E\to M$ is equivalent data to an auxiliary
	vector bundle $V\to M$ and a $\xi'$-structure on $E\oplus V$.
	
	This implies that $\mathit{MT\xi}\simeq \mathit{MT\xi'}\wedge X$ for some Thom spectrum $X$ that is
	usually not too complicated.
	\item Next, invoke a change-of-rings theorem to greatly simplify the calculation of the $E_2$-page for
	$\xi'$-bordism of spaces or spectra. Then run the Adams spectral sequence, taking advantage of the extra
	structure afforded by the change-of-rings theorem.
\end{enumerate}
This recipe goes back to work of Anderson-Brown-Peterson~\cite{ABP69} and Giambalvo~\cite{Gia73a, Gia73b, Gia76}
computing twisted spin bordism. It is most commonly used in the case $\xi' = \Spin$, where it has been frequently
used to compute bordism groups for tangential structures representing field theories with fermions; $\xi' = \String$ is
less common but still appears in physically motivated examples, including the tangential structure of the Sugimoto
string~\cite{Sug99} and $\xi = \String^c$~\cite{CHZ11, Sat11}.

Unfortunately, $\xihet$ and $\xiCHL$ do not belong to this class of examples: we will see in \cref{simple_shear_no}
that there is no way to write these tangential structures as twisted string structures in the sense
above.\footnote{The presence of the B-field, and how the Bianchi identity mixes it with the principal
$\Spin_n$-bundle of frames, rules out $\xi' = \SO$, $\Spin$, or $\Spin^c$.}\textsuperscript{,}\footnote{This
problem also happens to the tangential structures studied in~\cite{FH21b, DY22}.} So we have to do something different.

At odd primes, we plow ahead with the unsimplified Adams spectral sequence, though since we only care about
dimensions $11$ and below the computations are very tractable. At $p = 2$, though, we can modify the above strategy
to simplify the computation: in \S\ref{shearing}, we generalize the notion of ``twisted string bordism'' for which
the change-of-rings trick works to include string covers (in the sense of \cref{string_cover}) of groups of the
form $\Spin\times G$. This applies to both $\xihet$ and $\xiCHL$, and so we are off to the races.
\begin{rem}
We are far from the first to compute bordism groups for a tangential structure $\xi\colon B\to B\O$ where $B$ is
the classifying space of a $2$-group. For example, $\Omega_*^\String$ has been calculated in a range of degrees
by~\cite{Gia71, HR95, MG95, Hov08}; other examples include~\cite{Hil09, KT17, WW19b, WW19a, WWZ19, Tho20, LT21,
Yu21, DL23}.
\end{rem}
\subsection{Twists of string bordism}
\label{shearing}
\begin{quote}
\textit{
	``Started out with a twist, how did it end up like this?\\
	It was only a twist, it was only a twist\dots{}''
}
\end{quote}

Once the tangential structure for a bordism question is known, the next step is typically to prove a ``shearing'' theorem
simplifying the tangential structure. For example, the usual route to computing \pinm bordism~\cite[\S 7]{Pet68} first
establishes an isomorphism between \pinm bordism and the spin bordism of the Thom spectrum $\Sigma^{-1}\MO_1$, and
then computes the latter groups using something like the Adams or Atiyah-Hirzebruch spectral sequence.

There are a few different approaches to shearing theorems, such as those in~\cite{FH21,DDHM22a}, but generally they
work with Thom spectra of vector bundles; for example, the above simplification of \pinm bordism begins with the
observation that a \pinm structure on a bundle $E\to M$ is equivalent data to a real line bundle $L\to M$ and a
spin structure on $E\oplus L$, which follows from a characteristic class computation, and then passes the data of
``$L$ and a spin structure on $E\oplus L$'' through the Pontrjagin-Thom theorem.

This approach does not work for the heterotic and CHL tangential structures.
\begin{lem}
\label{simple_shear_no}
There is no spin vector bundle $V$ on $B((\rE_8\times\rE_8)\rtimes\Z/2)$ such that $\lambda(V) = c_1 + c_2$, and
there is no spin vector bundle $W$ on $B\rE_8$ such that $\lambda(V) = 2c$.
\end{lem}
This means there is no way to express a $\xihet$-structure as ``a $G$-bundle and a string structure on $E$ plus
some associated bundle,'' and likewise for $\xiCHL$.
\begin{proof}
Let $G$ be a compact, simple, simply connected Lie group and $\rho\colon G\to\SU_n$ be a representation.
$H^4(BG;\Z)$ and $H^4(B\SU_n;\Z)$ are both canonically isomorphic to $\Z$, so the pullback map $\rho^*$ on $H^4$ is
a map $\Z\to\Z$, necessarily multiplication by some integer $\delta(\rho)$. Because
$\SU_n$ is compact, connected, and simply connected, the standard inclusion $\SU_n\to\GL_{2n}(\R)$ lifts to a map
$\SU_n\to\Spin_{2n}$. Choices of this lift are a torsor over $H^1(B\SU_n;\Z/2) = 0$, meaning that the
characteristic class $\lambda$ is uniquely defined for $\SU_n$-representations. Moreover, $\lambda$ of the defining
representation is a generator of $H^4(B\SU_n;\Z)$; because $H^4(B\SU_n;\Z)$ is torsion-free, it suffices to show
$2\lambda = p_1$ is twice a generator, which is standard. The \term{Dynkin index} of $G$ is the minimum value of
$\abs{\delta(\rho)}$ over all such representations $\rho$. Laszlo-Sorger~\cite[Proposition 2.6]{LS97} show that the
Dynkin index of $\rE_8$ is $60$, meaning that for any vector bundle $V\to B\rE_8$ with $\SU$-structure induced from
a representation, $\lambda(V)$ is at least $60$ times a generator.%\footnote{}

We would like to generalize to real representations.
\begin{lem}
The complexification map $\Spin_n\to\O_n\to\U_n$ has image contained in $\SU_n$.
\end{lem}
\begin{proof}
A lift of a representation $\rho\colon G\to\U_n$ has image contained in $\SU_n$ if and only if $c_1$ of the complex
vector bundle associated to $\rho$ vanishes. When one pulls back across the complexification map $B\O_n\to B\U_n$,
$c_1$ is sent to the image of $w_1$ under the Bockstein map $\beta\colon H^1(B\O_n;\Z/2)\to H^2(B\O_n;\Z)$; when we
pull back further to $B\Spin_n$, $w_1\mapsto 0$, so $c_1 = \beta w_1\mapsto 0$ too.
\end{proof}
%because $c_1$ of a complexified representation is the Bockstein of $w_1$, and $w_1$ vanishes on
%spin representations.
Thus the Dynkin index fact we mentioned above applies to complexifications of representations landing in $\Spin_n$.

If $V$ is a real representation of a group $G$, $V\otimes\C\cong V\oplus V$ as real
representations, so using the Whitney sum formula for $\lambda$ (\cref{lambda_whitney}), $\lambda(V\otimes\C)
=2\lambda(V)$. Therefore if $V$ is any real spin representation of $\rE_8$, $\lambda(V\otimes\C)$ is at least $60$
times a generator, so $\lambda(V)$ is at least $30$ times a generator. Thus the class defining $\GCHL$, which is
twice a generator, is not $\lambda$ of any spin representation of $\rE_8$; likewise for $\Ghet$, as one could
restrict to either factor of $\rE_8$ inside $\rE_8^2\rtimes\Z/2$ and obtain a representation with $\lambda$ equal
to the generator.

Finally, the Atiyah-Segal completion theorem extends this from representations to all vector bundles. Because
$\lambda$ is additive (\cref{lambda_whitney}), it factors through the Grothendieck group $\mathit{KSpin}(BG)$ of
spin vector bundles on $BG$, and similarly, evaluated on spin representations, $\lambda$ factors through the
corresponding Grothendieck group $\mathit{RSpin}(G)$. Atiyah-Segal~\cite[\S 7, \S 8]{AS69} show that taking the
associated bundle of an arbitrary representation exhibits the Grothendieck ring $\KO^0(BG)$ of all vector bundles
on $BG$ as the completion of the representation ring $\mathit{RO}(G)$ at its augmentation ideal. Thus given a
$\Z$-valued characteristic class $c$ of arbitrary vector bundles of $G$ which satisfies the Whitney sum formula,
passing from representations of $G$ to vector bundles on $BG$ does not decrease the minimal value of $\abs c$.

In order to use the Atiyah-Segal theorem, we need to get from spin representations and vector bundles to arbitrary
ones. We will do so, at the cost of lowering the minimum value of $\lambda$ a little bit. For any vector bundle
$V$, $V^{\oplus 4}$ admits a canonical spin structure: the Whitney sum formula for Stiefel-Whitney classes shows a
spin structure exists; then choose a spin structure universally over $B\O$. Therefore we can define $\lambda$ of an
arbitrary representation of $\rE_8$ or vector bundle on $B\rE_8$ by $\lambda(V)\coloneqq \tfrac 14
\lambda(V^{\oplus 4})$, valued in $\tfrac 14\Z$. Therefore passing from $\mathit{RO}(\rE_8)\to\KO^0(B\rE_8)$ to
$\mathit{RSpin}(\rE_8)\to\mathit{KSpin}(B\rE_8)$ divides the minimal value of $\lambda$ by at most $4$, and now we
can invoke Atiyah-Segal, so it is still not possible to get $2c$ and $\xiCHL$; and likewise for
$\rE_8^2\rtimes\Z/2$ in place of $\rE_8$ to show that the characteristic class for $\xihet$ cannot be achieved.
\end{proof}
So we take a different approach: we cannot get Thom spectra corresponding to vector bundles, but we can still
obtain $\MTString$-module Thom spectra. We accomplish this using the theory of
Ando-Blumberg-Gepner-Hopkins-Rezk~\cite{ABGHR14a, ABGHR14b} (ABGHR), which we briefly summarize.

The idea behind the ABGHR perspective on Thom spectra is to generalize the notion of local coefficients to
generalized cohomology theories. Given a based, connected space $X$ and a homomorphism $\rho\colon
\pi_1(X)\to\GL_1(\Z)\cong\set{\pm 1}$, one obtains a local coefficient system $\Z_\rho$ on $X$: this is a bundle on
$X$ with fiber $\Z$, and whose monodromy around a loop $\gamma\in\pi_1(X)$ is precisely $\rho(\gamma)$. Given
$\Z_\rho$, we can take twisted cohomology groups: if $\widetilde X\to X$ denotes the universal cover, then the
cochain complex $C^\ast(\widetilde X;\Z)$ has a $\pi_1(X)$-action induced from the $\pi_1(X)$-action on $\widetilde
X$. If $C^\ast(X; \Z_\rho)$ denotes the subcomplex of $C^\ast(\widetilde X;\Z)$ of cochains which transform under
this $\pi_1(X)$-action by $\rho$, then $H^\ast(X;\Z_\rho)\coloneqq H^\ast(C^\ast(X;\Z_\rho))$.

Another way to say this is that if $\pt/G$ denotes the category with one object $\ast$ and $\Hom(\ast, \ast)
= G$, $\rho$ defines a $\pt/\pi_1(X)$-shaped diagram of chain complexes of abelian groups:
\begin{equation}
\label{local_coh}
	\pt/\pi_1(X) \overset\rho\longrightarrow \pt/\set{\pm 1}\longrightarrow \Ch_\Z,
\end{equation}
sending $\pt$ to $C^\ast(\widetilde X;\Z)$, and sending $g\in\pi_1(X)$ to the action by $\rho(g)$. The subcomplex
of cochains that transform by $\rho$ is precisely the limit of this diagram. For functoriality reasons, we envision
this complex as cochains on some object $\mathcal X$ which is a \emph{co}limit of a diagram akin
to~\eqref{local_coh}.

To summarize, twisted cohomology, i.e.\ cohomology of the Thom spectrum, is expressed as a colimit of a diagram of
chain complexes of $\Z$-modules induced from a map $X\to B\Aut(\Z)$. Ando-Blumberg-Gepner-Hopkins-Rezk lift this to
spectra. Specifically, given a ring spectrum $R$,
%\footnote{\TODO: what level of commutativity do they assume?
%Mention this and then say that we will only need $E_\infty$}
Ando-Blumberg-Gepner-Hopkins-Rezk naturally associate
a topological group\footnote{$\GL_1(R)$ is not exactly a topological group, but the
homotopy-coherent version thereof: a grouplike $A_\infty$-space.} $\GL_1(R)$, thought of as the group of units or
group of automorphisms of $R$. The classifying space
$B\GL_1(R)$ carries the universal local system of $R$-lines; a local system of $R$-lines over $X$ is equivalent
data to a map $X\to B\GL_1(R)$.
\begin{defn}[{Ando-Blumberg-Gepner-Hopkins-Rezk~\cite[Definition 2.20]{ABGHR14a}}]
\label{ABGHR_Thom}
The \term{Thom spectrum} $Mf$ associated to a map $f\colon X\to B\GL_1(R)$ is the colimit of the diagram $X\to
B\GL_1(R)\to\Mod_R$, where we think of $X$ as its fundamental $\infty$-groupoid.
\end{defn}
When $R = \Sph$, this is due to Lewis~\cite[Chapter IX]{LMSM86}. In \cref{ABGHR_Thom}, we have to consider the
fundamental $\infty$-groupoid, rather than just $\pi_1$, because $R$ can have higher automorphisms, because spectra
are derived objects.

The Thom spectrum of a map to $B\GL_1(R)$ is an $R$-module.
\begin{exm}[Twisted ordinary cohomology]
\label{twisted_ord_coh}
It turns out $B\GL_1(H\Z)\simeq K(\Z/2, 1)$, so the ABGHR viewpoint recovers $\Aut(\Z)$ and the usual notion of
cohomology twisted by a local system. To prove this homotopy equivalence, use the homotopy pullback
square of $E_\infty$-spaces~\cite[Definition 2.1]{ABGHR14b}
\begin{equation}
	% https://q.uiver.app/?q=WzAsNCxbMCwwLCJcXEdMXzEoSFxcWikiXSxbMSwwLCJcXE9tZWdhXlxcaW5mdHkgSFxcWiJdLFswLDEsIihcXHBpXzAgKEhcXFopKV5cXHRpbWVzIl0sWzEsMSwiXFxwaV8wKEhcXFopIl0sWzAsMV0sWzAsMl0sWzEsMywiXFxwc2kiXSxbMiwzXV0=
\begin{tikzcd}
	{\GL_1(H\Z)} & {\Omega^\infty H\Z} \\
	{(\pi_0 (H\Z))^\times} & {\pi_0(H\Z).}
	\arrow[from=1-1, to=1-2]
	\arrow["\varphi", from=1-1, to=2-1]
	\arrow["\psi", from=1-2, to=2-2]
	\arrow[from=2-1, to=2-2]
\end{tikzcd}
\end{equation}
$\Omega^\infty H\Z\simeq\Z$ as $E_\infty$-spaces, and $\psi$ is a homotopy equivalence of $E_\infty$-spaces.
Therefore $\varphi$ is also a homotopy equivalence of $E_\infty$-spaces, and we conclude.
\end{exm}
\begin{exm}[Thom spectra from vector bundles]
Boardman's original definition of Thom spectra~\cite[\S V.1]{Boa65} associates them to virtual vector bundles $V\to
X$. Let us connect this to the ABGHR definition. Virtual vector bundles are classified by maps $f_V\colon X\to
B\O$, and one avatar of the $J$-homomorphism~\cite{Whi42} is a map $J\colon \O\to\GL_1(\Sph)$~\cite[Example
3.15]{ABG10}, which deloops to a map of spaces $BJ\colon B\O\to B\GL_1(\Sph)$. A map with this signature is a
natural assignment from virtual vector bundles $V\to X$ to local systems of invertible $\Sph$-modules, and $BJ$
assigns to $V$ the local system with fiber $\Sph^{V_x}$ at each $x\in X$. Putting these maps together, we have an
$X$-shaped diagram
\begin{equation}
\label{OG_Thom}
	X\overset{f_V}\longrightarrow B\O \overset{BJ}{\longrightarrow} B\GL_1(\Sph)\longrightarrow \cat{Sp},
\end{equation}
and the colimit of this diagram, which is a Thom spectrum in the ABGHR sense, coincides with the Thom spectrum
$X^V$ in the usual sense. This is a combination of theorems of Lewis~\cite[Chapter IX]{LMSM86} and
Ando-Blumberg-Gepner-Hopkins-Rezk~\cite[Corollary 3.24]{ABGHR14a}.
\end{exm}
This approach to Thom spectra plays well with multiplicative
structures. If $R$ is an $E_\infty$-ring spectrum, then the grouplike $A_\infty$-structure on $\GL_1(R)$
refines to a grouplike $E_\infty$-structure, making $\GL_1(R)$ and therefore $B\GL_1(R)$ into infinite loop spaces.
For $1\le k\le\infty$, if $X$ is a $k$-fold loop space and $f\colon X\to B\GL_1(R)$ is a $k$-fold loop map, then
the Thom spectrum $Mf$ inherits the structure of an $E_k$-ring spectrum. This is a theorem of Lewis~\cite[Theorem
IX.7.1]{LMSM86} for $R = \Sph$ and Ando-Blumberg-Gepner~\cite[Theorem 1.7]{ABG18} for more general $R$.

$B\O$ has an infinite loop space structure coming from the addition-like operation on $B\O$ of direct sum of vector
bundles. The $J$-homomorphism $BJ\colon B\O\to B\GL_1(\Sph)$ is an infinite loop map, so we get an $E_\infty$-ring
structure on $\mathit{MT}\xi$ if $\xi$ is a tangential structure satisfying a \term{2-out-of-3 property}, i.e.\ whenever
any two of $E$, $F$, and $E\oplus F$ have a $\xi$-structure, the third has an induced $\xi$-structure. The idea is
that the 2-out-of-3 property implies that $\xi\colon B\to B\O$ is an infinite loop map, so passing to
$B\GL_1(\Sph)$ and taking the Thom spectrum, we obtain an $E_\infty$-ring spectrum. This applies to $\MTO$,
$\MTSO$, $\MTSpin^c$, $\MTSpin$, and $\MTString$; however, some commonly considered tangential structures appearing in
physics do not have this property, including $B\mathrm{Pin}^\pm$.

%Often, direct product of manifolds makes the bordism groups $\Omega_\ast^\xi$ into a graded ring. This occurs for
%$\xi = \O$, $\SO$, $\Spin$, $\Spin^c$, $\U$, $\SU$, $\String$, and many other examples, but there are also plenty
%of tangential structures for which this does not work, including $\Pin^\pm$. This ring structure tends to be less useful
%in physics applications of bordism, but plays an important role in the ABGHR approach to Thom spectra.
%
%The first thing to do is to lift the ring structure on $\Omega_\ast^\xi$ to a homotopical version of a ring
%structure on the Thom spectrum $\mathit{MT\xi}$. In stable homotopy theory, there are multiple notions of
%``commutative ring;'' today, we only need the strictest one, an \term{$E_\infty$-ring structure} on
%$\mathit{MT\xi}$. [\TODO: did LMS or MQRT do some of this before ABGHR?] Given an $E_\infty$-ring spectrum $R$, one
%can study ``$R$-lines:'' $R$-modules which are isomorphic to $R$.
\begin{prop}
\label{general_twisting}
Let $B$ and $X$ be infinite loop spaces and $\xi\colon B\to B\O$ and $f\colon B\to X$ be infinite loop maps, so
that the fiber $\eta\colon F\to B$ of $f$ is also a map of infinite loop spaces. This data naturally defines twists
of the Thom spectrum $M(\xi\circ\eta)$ over $X$, i.e.\ a map $X\to B\GL_1(M(\xi\circ\eta))$.
\end{prop}
\begin{proof}
The fiber of $\eta\colon F\to B$ is another infinite loop map $\zeta\colon \Omega X\to F$, so the induced map of
Thom spectra (where the maps down to $B\O$ are $\xi\circ\eta\circ\zeta$ and $\xi\circ\eta$ respectively) is a map
of $E_\infty$-ring spectra. Because $\xi\circ\eta\circ\zeta$ is nullhomotopic, its Thom spectrum is a suspension
spectrum, so we have a map of $E_\infty$-ring spectra $\Sigma_+^\infty \Omega X\to M(\xi\circ\eta)$.

Ando-Blumberg-Gepner-Hopkins-Rezk~\cite[(1.4), (1.7)]{ABGHR14b} prove that $\Sigma^\infty_+$ and $\GL_1$ are an
adjoint pair on the categories of infinite loop spaces and $E_\infty$-ring spectra. Applying this adjunction, we
have a map of infinite loop spaces $\Omega X\to \GL_1(M(\xi\circ\eta))$; deloop to obtain the map in the theorem
statement.
\end{proof}
\begin{thm}[{Beardsley~\cite[Theorem 1]{Bea17}}]
\label{universal_Thom}
With notation as in \cref{general_twisting}, the Thom spectrum of the ``universal twist'' $X\to
B\GL_1(M(\xi\circ\eta))$ is canonically equivalent to $M\xi$.
\end{thm}
%This amounts to a computation of the colimit in~\eqref{OG_Thom} in this example.
\begin{cor}\hfill
\label{bordism_twists}
\begin{enumerate}
	\item There is a map $\widehat w_1\colon K(\Z/2, 1)\to B\GL_1(\MTSO)$ which, after taking the quotient
	$\MTSO\to H\Z$, passes to the usual homotopy equivalence $K(\Z/2, 1)\to B\GL_1(H\Z)$ from
	\cref{twisted_ord_coh}.
	\item There is a map $\widehat w_2\colon K(\Z/2, 2)\to B\GL_1(\MTSpin)$ which, after composing with the
	Atiyah-Bott-Shapiro map $\MTSpin\to\ko$~\cite{ABS, Joa04}, is the usual map $K(\Z/2, 2)\inj
	B\GL_1(\ko)$~\cite{DK70, HJ20}.
	\item\label{k_twist} There is a map $\widehat{\beta w_2}\colon K(\Z, 3)\to B\GL_1(\MTSpin^c)$ which, after
	composing with the Atiyah-Bott-Shapiro map $\MTSpin^c\to\ku$~\cite{ABS, Joa04, AHR10}, is the usual twist of
	$K$-theory by degree-$3$ classes $K(\Z, 3)\to B\GL_1(\ku)$~\cite{DK70, Ros89, AS04, ABG10}.
	\item There is a map $\widehat\lambda\colon K(\Z, 4)\to B\GL_1(\MTString)$ which, when composed with the
	Ando-Hopkins-Rezk orientation $\MTString\to\tmf$~\cite{AHR10}, is the Ando-Blumberg-Gepner map $K(\Z, 4)\to
	B\GL_1(\tmf)$~\cite[Proposition 8.2]{ABG10}.
\end{enumerate}
\end{cor}
Part~\eqref{k_twist} is a theorem of Hebestreit-Joachim~\cite[Appendix C]{HJ20}. The other parts are surely
known, though we were unable to find them in the literature.
\begin{proof}
Apply \cref{general_twisting} to the four maps
\begin{enumerate}
	\item $w_1\colon B\O\to K(\Z/2, 1)$, whose fiber is $B\SO$;
	\item $w_2\colon B\SO\to K(\Z/2, 2)$, whose fiber is $B\Spin$;
	\item $\beta\circ w_2\colon B\SO\to K(\Z, 3)$, whose fiber is $B\Spin^c$, where $\beta\colon H^k(\bl;\Z/2)\to
	H^{k+1}(\bl;\Z)$ is the Bockstein; and
	\item $\lambda\colon B\Spin\to K(\Z, 4)$, whose fiber is $B\String$.
\end{enumerate}
All four of these are infinite loop maps, because these characteristic classes are additive in direct sums. For
compatibility with preexisting twists, we use the fact that in the \spinc and string cases,
Ando-Blumberg-Gepner~\cite[\S 7, \S 8]{ABG10} construct the desired twists $K(\Z, 3)\to B\GL_1(\ku)$ and $K(\Z,
4)\to B\GL_1(\tmf)$ in the same way as we construct the twists of $\MTSpin^c$ and $\MTString$, so compatibility
follows from functoriality. The cases of $\ko$ and $H\Z$ are analogous.
\end{proof}
The homotopy groups of the Thom spectra of the twists \cref{bordism_twists} have bordism interpretations. Looking
at $\widehat w_2$ for example, a spin structure on an oriented manifold is a trivialization of $w_2(TM)$, but given
a space $X$ and a degree-$2$ cohomology class $B$, thought of as a map $f_B\colon X\to K(\Z/2, 2)$, the homotopy
groups of $\mathit{MT}(\widehat w_2\circ f_B)$ are the bordism groups of oriented manifolds $M$ together with a map
$g\colon M\to X$ and a trivialization of $w_2(TM) + g^*B$, as was shown by Hebestreit-Joachim~\cite[Corollary 3.3.8]{HJ20}. The other three cases are analogous; in particular, we
have described the Thom spectra for $\xihet$ and $\xiCHL$ as $\MTString$-module Thom spectra.

These kinds of twisted bordism have been studied before: \spinc structures twisted by a degree-$3$ cohomology class
were first studied by Douglas~\cite[\S 5]{Dou06}, and they appear implicitly in work of Freed-Witten~\cite{FW99} on
anomaly cancellation. Twisted spin and string structures of the sort appearing in \cref{bordism_twists} were first
considered by B.L.\ Wang~\cite[Definitions 8.2, 8.4]{Wan08}. See~\cite{DFM11a, DFM11b, Sat11a, Sat11b, Sat12,
Sat15, SW15, LSW20, SY21} for more examples of twisted generalized cohomology theories from a similar point of view
and some applications in physics.

The first case, involving twists of $\MTSO$ by degree-$1$ $\Z/2$-cohomology classes, is the notion of a twisted
orientation from the beginning of this section: given a real line bundle $L\to X$, we ask for data of a map
$g\colon M\to X$ and an orientation on $TM\oplus g^*(L)$. In the ABGHR perspective this says that the map $\widehat
w_1$ factors through $B\O_1$ as
\begin{equation}
\label{MTSO_twist_not_new}
	K(\Z/2, 1)\overset\simeq\to B\O_1\inj B\O\to B\GL_1(\Sph)\to B\GL_1(\MTSO).
\end{equation}
But the others do not factor this way.
\begin{rem}
There is a complex version of~\eqref{MTSO_twist_not_new}. Let $\mathcal W$ denote \term{Wall's bordism
spectrum}~\cite{Wal60}, whose homotopy groups are the bordism groups of manifolds with an integral lift of $w_1$.
Explicitly, if $\xi\colon F\to B\O$ is the fiber of $\beta w_1\colon B\O\to K(\Z, 2)$, then $\mathcal W\coloneqq
\mathit{MT\xi}$.  \Cref{general_twisting} then produces a map $\widehat{\beta w_1}\colon K(\Z, 2)\to
B\GL_1(\mathcal W)$, but degree-$2$ cohomology classes are equivalent to complex line bundles, and $\widehat{\beta
w_1}$ factors as
\begin{equation}
\label{Wall_factor}
	K(\Z, 2) \overset\simeq\to B\T\to B\O_2\to B\O\to B\GL_1(\Sph)\to B\GL_1(\mathcal W).
\end{equation}
\end{rem}
\begin{rem}
\label{universal_twist}
One consequence of the fact that $\widehat w_1$ (resp.\ $\widehat{\beta w}_1$) factors as
in~\eqref{MTSO_twist_not_new} (resp.\ \eqref{Wall_factor}), i.e.\ as a twist associated to a real (resp.\ complex)
line bundle $L\to X$ is that the associated $\MTSO$-module (resp.\ $\mathcal W$-module) Thom spectrum splits as
$\MTSO\wedge X^{L-1}$ (resp.\ $\mathcal W\wedge X^{L-2}$). Working universally over $B\O_1$ and $B\T$,
\cref{universal_Thom} gives us homotopy equivalences $\MTSO\wedge (B\O_1)^{L-1}\simeq\MTO$ and $\mathcal W\wedge
(B\T)^{L-2}\simeq\MTO$; the former is a theorem of Atiyah~\cite[Proposition 4.1]{Ati61}.
\end{rem}

We will apply \cref{bordism_twists} to the degree-$4$ characteristic classes that the Bianchi identity told us for
the heterotic and CHL tangential structures. Given a space $X$ with a class $\mu\in
H^4(X;\Z)$, let $\mathcal B(X)$ denote the homotopy fiber of $\lambda + \mu\colon B\Spin\times X\to K(\Z, 4)$, and
let $\xi^\mu$ denote the tangential structure
\begin{equation}
	\xi^\mu\colon \mathcal B(X)\longrightarrow B\Spin\times X\longrightarrow B\O.
\end{equation}
$\mathit{MT\xi}^\mu$ is equivalent to the $\MTString$-module Thom spectrum associated to the twist
$\widehat\lambda\circ\mu\colon X\to B\GL_1(\MTString)$. If $X = BG$ for a Lie group $G$, $\mathcal B(X)$ is the
classifying space of the string $2$-group $\mathcal S(\Spin\times G, \lambda + \mu)$.
Let $\cA$ denote the $2$-primary Steenrod algebra and for $n\ge 0$, let $\cA(n)$ denote the subalgebra of $\cA$
generated by $\Sq^1,\dotsc,\Sq^{2^n}$.
In work to appear joint with
Matthew Yu~\cite{DY}, we compute the $\cA$-module structure on $H^*(\mathit{MT\xi}^\mu;\Z/2)$.
\begin{defn}
Let $R$ denote the $\Z/2$-algebra $\cA(1)[S]$, i.e.\ the algebra with generators $\Sq^1$, $\Sq^2$, and $S$, and
with Adem relations for $\Sq^1$ and $\Sq^2$. Given $X$ and $\mu$ as above, define the $\cA(1)$-module $T(X,
\mu)\coloneqq H^*(X;\Z/2)$, and give $T(X, \mu)$ an $R$-module structure by defining
\begin{equation}
\label{not_Cartan}
	S(x) \coloneqq \mu x + \Sq^4(x).
\end{equation}
\end{defn}
We want to think of $S$ as $\Sq^4$ and $T(X, \mu)$ as an $\cA(2)$-module, but a priori it is not clear that this
$S$-action satisfies the Adem relations. 
\begin{thm}[\cite{DY}]\hfill
\label{fake_shearing_DY}
\begin{enumerate}
	\item\label{adem_part} The $R$-module structure on $T(X, \mu)$ satisfies the Adem relations for $\Sq^1$,
	$\Sq^2$, and $\Sq^4 = S$, hence induces an $\cA(2)$-module structure on $T(X, \mu)$.
	\item\label{A_mod_part} There is an map of $\cA$-modules
	\begin{equation}
	\label{the_goal_isom}
		H^*(\mathit{MT\xi}^\mu; \Z/2) \longrightarrow \cA\otimes_{\cA(2)} T(X, \mu),
	\end{equation}
	natural in the data $(X, \mu)$, which is an isomorphism in degrees $15$ and below.
\end{enumerate}
\end{thm}
As~\cite{DY} is not yet available, we describe a proof of this theorem in \cref{DY_proof_rem}.
%\begin{cor}
%\label{tmu_cor}
%In degrees $15$ and below, there is an isomorphism of $\cA$-modules
%\todo{``The first lines of Corollary 2.21 are really a definition, and as such crucially
%used in what follows . Maybe better to make this explicit as a definition herr.
%For instance, a reader who happens upon Prop. 2.30 on p.23 may have a hard time
%figuring out that "T(...)" refers to the first lines of Corollary 2.21 (I assume
%it does?)''}
%\todo{``Regarding this definition: it says
%"plus x(? mod 2)"
%which I am unsure how to parse. Maybe it's a typo for
%  "plus ? mod 2"
%But why is "plus" here written in words. Better to write out the full intended 
%formula in mathematical form, of tjis crucial definition.''}
%\todo{``Also: With such modified definition, best to comment on why/that the module
%structure obtained is still well-defined.'' (Note: would also help to mention that it's an $\cA(2)$-module but not,
%in general, an $\cA$-module)}
%\begin{equation}
%	H^*(\mathit{MT\xi^\mu};\Z/2) \overset\cong\longrightarrow \cA\otimes_{\cA(2)} T(\mu).
%\end{equation}
%\end{cor}
%\begin{proof}
%This follows from \cref{fake_shearing_DY} and the fact that in degrees $15$ and below, $H^*(B\String;\Z/2)$
%coincides with $\cA\otimes_{\cA(2)}\Z/2$, a theorem of Bahri-Mahowald~\cite{BM80}.
%\end{proof}
\begin{cor}
\label{fake_shearing}
For $t-s\le 15$, the $E_2$-page of the Adams spectral sequence computing $2$-completed $\xi^\mu$-bordism is
\begin{equation}
	E_2^{t,s} = \Ext_{\cA(2)}^{s,t}(T(X, \mu), \Z/2).
\end{equation}
\end{cor}
As $\cA(2)$ is much smaller than $\cA$, this is much easier to work with.
\begin{proof}
This follows from the change-of-rings formula: if $\mathcal B$ is a graded Hopf algebra, $\mathcal C$ is a graded
Hopf subalgebra of $\mathcal B$, and $M$ and $N$ are graded $\mathcal B$-modules, then there is a natural
isomorphism
\begin{equation}
\label{change-of-rings}
	\Ext_{\mathcal B}^{s,t}(\mathcal B\otimes_{\mathcal C}M, N)\overset\cong\longrightarrow \Ext_{\mathcal
	C}^{s,t}(M, N).
\end{equation}
This you can think of as the derived version of a maybe more familiar isomorphism
\begin{equation}
	\Hom_{\mathcal B}(\mathcal B\otimes_{\mathcal C}M, N)\overset\cong\longrightarrow \Hom_{\mathcal C}(M, N).
\end{equation}
In our example, $\mathcal B$ is the Steenrod algebra, which is a Hopf algebra, and $\mathcal C$ is $\mathcal A(2)$,
which is indeed a Hopf subalgebra of $\mathcal A$, so we can invoke~\eqref{change-of-rings} and conclude.
\end{proof}
We will use this simplification in the cases $\xi^\mu = \xihet,\xiCHL$ to run the Adams spectral sequences
computing $\Omega_*^{\xihet}$ and $\Omega_*^{\xiCHL}$ at $p = 2$.
\begin{rem}[Proof sketch of \cref{fake_shearing_DY}]
\label{DY_proof_rem}
To prove~\eqref{adem_part}, check the Adem relations for $\cA(2)$ directly. The first step in proving
part~\eqref{A_mod_part} is to establish a Thom isomorphism for mod $2$ cohomology. We make use of the \term{Thom
diagonal}, a map of $\MTString$-modules
\begin{equation}
\label{Thom_diagonal}
	\mathit{MT\xi}^\mu\overset{\Delta^t}{\longrightarrow} \mathit{MT\xi}^\mu \wedge \MTString\wedge \Sigma_+^\infty X
\end{equation}
defined as follows: the diagonal map $\Delta\colon X\to X\times X$ is a map of spaces over $B\GL_1(\MTString)$,
if we give $X$ the map $\widehat\lambda\circ\mu$ to $B\GL_1(\MTString)$  and we give $X\times X$ the map
$(\widehat\lambda\circ\mu, \ast)$. Applying the $\MTString$-module Thom spectrum functor to $\Delta$ produces
\eqref{Thom_diagonal}.
Smash~\eqref{Thom_diagonal} with $H\Z/2$. The result is the Thom diagonal for a twist of $H\Z/2$, but all such
twists are trivializable (i.e.\ all $H\Z/2$-bundles admit an orientation). Therefore by~\cite[Proposition
3.26]{ABGHR14b} the following composition is an equivalence:
\begin{equation}
	\mathit{MT\xi}^\mu\wedge H\Z/2\overset{\Delta^t}{\longrightarrow}
		\mathit{MT\xi}^\mu\wedge \Sigma_+^\infty X \wedge H\Z/2\longrightarrow
		\MTString\wedge \Sigma_+^\infty X\wedge H\Z/2,
\end{equation}
which is the $\Z/2$-homology Thom isomorphism. The analogous fact is true for mod $2$ cohomology.
%
%To do this, use the map of ring
%spectra $q\colon \MTSpin\to H\Z/2$ defined by first forgetting the string structure, which is an $E_\infty$-ring
%map $\MTSpin\to\MTO$, then taking the coconnective quotient of $\MTO$, which is an $E_\infty$-ring map $\MTO\to
%H\Z/2$. Letting $q$ also denote the induced map on $B\GL_1(\bl)$, it follows from~\cite[\S 1.4]{ABGHR14b}there is
%an equivalence of $H\Z/2$-module spectra
%\begin{equation}
%	M(q\circ \widehat\lambda\circ\mu) \simeq H\Z/2\wedge M(\widehat\lambda\circ\mu).
%\end{equation}
%As $B\GL_1(H\Z/2)$ is contractible, the twist $q\circ\widehat\lambda\circ\mu\colon X\to B\GL_1(H\Z/2)$ is trivial,
%i.e.\ $M(q\circ\widehat\lambda\circ\mu)\simeq H\Z/2\wedge (\mathcal B(X))$.

The Thom diagonal makes $H^*(\mathit{MT\xi}^\mu;\Z/2)$ into a free, rank-$1$ module over $H^*(\mathcal B(X);\Z/2)$,
generated by the Thom class $U$. As the Thom diagonal is a map of spectra, we may use the Cartan formula to compute
the Steenrod squares of an arbitrary element of $H^*(\mathit{MT\xi}^\mu;\Z/2)$ in terms of Steenrod squares in
$\mathcal B(X)$ and $\Sq(U)$. As both $\Sq(U)$ and our desired isomorphism in~\eqref{the_goal_isom} are natural in
$X$ and $\mu$, it suffices to understand the universal case, where $X = K(\Z, 4)$ and $\mu$ is the tautological
class $\tau\in H^4(K(\Z, 4);\Z)$. In this case, \cref{universal_Thom} implies $\mathit{MT\xi}^\mu\simeq \MTSpin$.
By work of Anderson-Brown-Peterson~\cite{ABP67}, if $J$ is the $\cA(1)$-module $\cA(1)/\Sq^3$ and $M$ is the
$\cA(1)$-module $\Z/2\oplus \Sigma^8\Z/2\oplus \Sigma^{10}J$, then there is a map of $\cA$-modules
\begin{equation}
	H^*(\MTSpin;\Z/2)\longrightarrow \cA\otimes_{\cA(1)} M
\end{equation}
which is an isomorphism in degrees $15$ and below. And Giambalvo~\cite[Corollary 2.3]{Gia71} shows that there is a
map $H^*(\MTString;\Z/2)\to \cA\otimes_{\cA(2)}\Z/2$ which is also an isomorphism in degrees $15$ and below.
Therefore by the change-of-rings theorem~\eqref{change-of-rings} it suffices to exhibit a map of $\cA(2)$-modules
\begin{equation}
	T(K(\Z, 4), \tau) \longrightarrow \cA(2)\otimes_{\cA(1)} M
\end{equation}
which is an isomorphism in degrees $15$ and below. This can be verified directly, using as input the
$\cA(2)$-module structure on $H^*(K(\Z, 4);\Z/2)$ calculated by Serre~\cite[\S 10]{Ser53}.
\end{rem}
\subsection{$\xihet$ bordism at $p = 2$}
	\input{e8_semi}

\subsection{$\xihet$ bordism at odd primes}
	\input{odd_primary}
\subsection{$\xiCHL$ bordism}
	\input{chl_bordism}

%% file: e8_semi.tex
\label{xihet_at_2}
In this section we will first compute $H^\ast(BG;\Z/2)$ as an $\cA(2)$-module in low degrees, where $G\coloneqq
\rE_8^2\rtimes\Z/2$; then, using \cref{fake_shearing}, we run the Adams spectral sequence computing $2$-completed
$\xihet$ bordism in degrees $11$ and below.

First, though, we reformulate the problem slightly. Consider the tangential structure $\xihet'\colon
B^{\mathrm{het}'}\to B\O$ defined in the same manner as $\xihet$, but with $K(\Z, 4)$ replacing $B\rE_8$. In a
little more detail, $\Z/2$ acts on $K(\Z, 4)\times K(\Z, 4)$ by swapping the two factors; taking the Borel
construction
\begin{equation}
	B\coloneqq (K(\Z, 4)\times K(\Z, 4))\times_{\Z/2} E\Z/2
\end{equation}
produces a fiber bundle
\begin{equation}
	K(\Z, 4)\times K(\Z, 4)\longrightarrow B\longrightarrow B\Z/2.
\end{equation}
For $i = 1,2$, let $c_i\in H^4(K(\Z, 4)\times K(\Z, 4);\Z)$ be the tautological class for the $i^{\mathrm{th}}$
$K(\Z, 4)$ factor. The class $c_1 + c_2$ is invariant under the $\Z/2$-action, so we can follow it through the
Serre spectral sequence to learn that it defines a nonzero class $c_1 + c_2\in H^4(B;\Z/2)$. Define $f\colon
B^{\mathrm{het}'}\to B\Spin \times B$ to be the fiber of $\lambda - (c_1 + c_2)\colon B\Spin\times B\to K(\Z, 4)$;
then the tangential structure $\xihet'$ is the composition
\begin{equation}
\begin{tikzcd}
	{B^{\mathrm{het}'}} & {B\Spin\times B} & B\Spin & B\O.
	\arrow[from=1-3, to=1-4]
	\arrow["{\mathrm{pr}_1}", from=1-2, to=1-3]
	\arrow["f", from=1-1, to=1-2]
	\arrow["{\xihet'}"', curve={height=12pt}, from=1-1, to=1-4]
\end{tikzcd}
\end{equation}
That is, a $\xihet'$ structure on a manifold $M$ is a spin structure, a principal $\Z/2$-bundle $P\to M$, two
classes $c_1,c_2\in H^4(P;\Z)$ which are exchanged under the deck transformation, and a trivialization of
$\lambda(M) - (c_1 + c_2)$ (where the latter class is descended to $M$). This is the same data as a
$\xihet$ structure, except that we do not ask for $c_1$ or $c_2$ to come from principal $\rE_8$-bundles; therefore
there is a map of tangential structures $\widetilde c\colon \xihet\to\xihet'$, i.e.\ a map of spaces $B\Ghet\to
B^{\mathrm{het}'}$ commuting with the maps down to $B\O$. Like for $\xihet$, a $\xihet'$-structure is a twisted
string structure in the sense of \cref{bordism_twists}, via the class $\lambda - (c_1 + c_2)\colon B\to K(\Z, 4)$.

Bott-Samelson~\cite[Theorems IV, V(e)]{BS58} showed that the characteristic class $c\in H^4(B\rE_8;\Z)$ we defined
in \cref{char_class_e8}, interpreted as a map $c\colon B\rE_8\to K(\Z, 4)$, is $15$-connected. This means that the
homomorphism $\widetilde c$ induces on bordism groups, $\widetilde c\colon \Omega_k^{\xihet}\to\Omega_k^{\xihet'}$,
is an isomorphism in degrees $14$ and below. For our string-theoretic purposes, we only care about $k\le 12$, so we
may as well compute $\xihet'$-bordism. In the rest of this subsection, we often blur the distinction between
$\xihet$ and $\xihet'$; we will point out where it matters which one we are looking at.
\begin{rem}
\label{restricted_het_is_easy}
Turning off the $\Z/2$ symmetry switching the two $\rE_8$ factors, i.e.\ passing to a
$\xi^{r,\mathrm{het}}$-structure as in \cref{no_Z2}, simplifies this story considerably: the bordism groups were
known decades ago. Specifically, replace $B\rE_8$ with $K(\Z, 4)$ in the definition of $\xi^{r,\mathrm{het}}$ to
define a tangential structure $\xi^{r,\mathrm{het}}{}'$, which on a manifold $M$ consists of a spin structure on
$M$, two classes $c_1,c_2\in H^4(M;\Z)$, and a trivialization of $\lambda(M) - c_1 - c_2$. As
Witten~\cite[\S 4]{Wit86} noticed, this data is equivalent to a spin structure and the single class $c_1$, which
may be freely chosen; then $c_2$ must be $\lambda(M) - c_1$.  Therefore the tangential structure
$\xi^{r,\mathrm{het}}{}'$-structure is simply $B\Spin\times K(\Z, 4)\to B\O$, and just as for $\xihet$, the map
$\mathit{MT\xi}^{r,\mathrm{het}}\to\mathit{MT\xi}^{r,\mathrm{het}}{}'\simeq \MTSpin\wedge K(\Z, 4)_+$ is an
isomorphism on homotopy groups in degrees $14$ and below. Stong~\cite{Sto86} computes $\Omega_*^\Spin(K(\Z,
4))$ in degrees $12$ and below.
\end{rem}
As we discussed in \S\ref{higher_stuff}, the data of a trivial principal $\Z/2$-bundle on a manifold $M$ and two
principal $\rE_8$-bundles $P,Q\to M$ define a principal $\rE_8^2\rtimes\Z/2$-bundle on $M$ with $c_1 + c_2$ equal
to $c(P) + c(Q)$; data trivializing $c(P) + c(Q) - \lambda(M)$ therefore defines a $\xihet$ structure. Analogously,
the trivial $\Z/2$-bundle and a pair $c_1,c_2\in H^4(M;\Z)$ with a trivialization of $c_1 + c_2 - \lambda$ define a
$\xihet'$ structure.
\begin{lem}
\label{spin_to_xihet}
A spin manifold $M$ has a canonical $\xihet'$ structure specified as above by the trivial principal $\Z/2$-bundle,
the cohomology classes $c_1 = \lambda$ and $c_2 = 0$, and the canonical trivialization of $\lambda-\lambda = 0\in
H^4(M;\Z)$.
\end{lem}
This defines a map of tangential structures and therefore a map of Thom spectra $s_1\colon \MTSpin\to\MTxihet'$. A
$\xihet'$-structure includes data of a spin structure; forgetting the rest of the $\xihet'$-structure defines a map
$s_2\colon \MTxihet'\to\MTSpin$. The composition of $s_1$ and $s_2$ is homotopy equivalent to the identity, because
the underlying spin structure of the $\xihet'$ manifold built in \cref{spin_to_xihet} is the same spin structure we
began with.
\begin{cor}
\label{spin_splits_off}
There is a spectrum $\cQ$ and a splitting
\begin{equation}
	(s_2, q)\colon \MTxihet'\overset\simeq\longrightarrow \MTSpin\vee\cQ.
\end{equation}
\end{cor}
We will use this later to reduce the amount of spectral sequence computations we have to make.

Both \cref{spin_to_xihet,spin_splits_off} require us to use $\xihet'$ and not $\xihet$, though of course the
consequence on low-degree bordism groups is true for both.

%Then,
%consequences later down the line: once we know the cohomology, we know what's from $\MTSpin$, greatly simplifying
%the $\cA(2)$-module we have to display.

When $K$ is a finite group, Nakaoka~\cite[Theorem 3.3]{Nak61} proved that there is a ring isomorphism from the mod
$2$ cohomology of $B(\Z/2\ltimes (K\times K))$ to the $E_2$-page of the Serre spectral sequence
\begin{equation}
\label{Nakaoka_serre}
	E_2^{p,q} = H^p(B\Z/2; \underline{H^q(BK\times BK;\Z/2)}) \Longrightarrow H^{p+q}(B(\Z/2\ltimes(K\times
	K));\Z/2).
\end{equation}
Here the underline denotes the local coefficient system arising from the $\Z/2$-action on $BK\times BK$ by
switching the two factors. Since this local coefficient system can be nontrivial, one has to be careful defining
the multiplicative structure on the $E_2$-page of~\eqref{Nakaoka_serre}, but here it can be made explicit. As a
$\Z/2[\Z/2]$-module, $\underline{H^\ast(BK\times BK;\Z/2)}$ is a direct sum of:
\begin{itemize}
	\item the subalgebra $\mathcal H_1$ of classes fixed by $\Z/2$, which are of the form $x\otimes x$ for $x\in
	H^\ast(BK;\Z/2)$; and
	\item the submodule $\mathcal H_2$ spanned by classes of the form $x\otimes y$ where $x$ and $y$ are linearly
	independent.% [\TODO: that is the precise condition, right?]
\end{itemize}
Since $\Z/2$ acts trivially on $\mathcal H_1$ and $\mathcal H_1$ is a ring, $H^\ast(B\Z/2; \mathcal H_1)$ has a
ring structure. And as a $\Z/2[\Z/2]$-module, $\mathcal H_2$ is of the form $M\oplus M$ where $\Z/2$ acts by
swapping the two factors, so $H^\ast(B\Z/2; \mathcal H_2)$ vanishes in positive degrees.\footnote{To see this,
first observe that mod $2$ group cohomology for $G$ is additive in the $\Z/2[G]$-module of coefficients, so it
suffices to prove that $H^*(B\Z/2; M\oplus M)$ vanishes in positive degrees when $M = \Z/2$. But $\Z/2\oplus\Z/2$
is isomorphic to $\Z/2[\Z/2]$ as $\Z/2[\Z/2]$-modules (i.e.\ as vector spaces with $\Z/2$-representations,
$\Z/2\oplus\Z/2$ is isomorphic to the vector space of functions on the group $\Z/2$), and group cohomology valued
in the group ring is trivial, e.g.\ because the group ring is its own free resolution.} In degree zero, we obtain
invariants, spanned by elements of the form $x\otimes y + y\otimes x$, with $x,y\in H^\ast(BK;\Z/2)$. $\mathcal
H_1\oplus (\mathcal H_2)^{\Z/2} = E_2^{0,\bullet}$ is a subalgebra of $H^\ast(BK\times BK;\Z/2)$.

So far we have specified ring structures on $H^\ast(B\Z/2;\mathcal H_1) \supsetneq E_2^{>0, \bullet}$ and
$\mathcal H_1\oplus (\mathcal H_2)^{\Z/2} = E_2^{0, \bullet}$, and these ring structures agree where they
overlap. Therefore to specify a ring structure on the entirety of the $E_2$-page, it suffices to write down the
product of an element in $(\mathcal H_2)^{\Z/2}$ and an element in positive $p$-degree. We say that all such
products vanish; this is the ring structure that appears in Nakaoka's theorem.%[\TODO: double-check this]

Of course, $\rE_8$ is not a finite group. Nakaoka's theorem is true in quite great generality~\cite{Eve65,
Kah84, Lea97}; the version we need is proven by Evens~\cite{Eve65}, who proves the same ring isomorphism when $K$
is a compact Lie group. Thus this applies to $\xihet$, and not necessarily to $\xihet'$, but since their cohomology
rings are isomorphic in degrees $14$ and below, it does not matter which one we use in this calculation.

Now we make this ring structure and $\cA(2)$-module structure explicit. Since $c\colon B\rE_8\to K(\Z, 4)$ is
$15$-connected, it induces an isomorphism in cohomology in degrees $14$ and below, so we can use the cohomology of
$K(\Z, 4)$ as a stand-in for the cohomology of $B\rE_8$.
Serre~\cite[\S 10]{Ser53} computed the mod $2$ cohomology of $K(\Z, 4)$. It is an infinitely generated polynomial
algebra; in degrees $12$ and below the generators are: the tautological class $D\in H^4(K(\Z, 4);\Z/2)$,
$F\coloneqq \Sq^2 D$, $G\coloneqq\Sq^3D$, $J\coloneqq\Sq^4F$, and $K\coloneqq\Sq^5F$. %, and $M\coloneqq\Sq^6G$.

If $C$ is one of $D$, $F$, $G$, $J$, or $K$, we let $C_1$ denote the class coming from the
first copy of $B\rE_8$ and $C_2$ denote the class coming from the second copy. Thus we have the following additive
basis for the low-degree cohomology of $BG$:
\begin{enumerate}
	\item In $\mathcal H_1$, $D_1D_2x^k$ and $F_1F_2x^k$ for $k\ge 0$.
	\item In $(\mathcal H_2)^{\Z/2}$, $D_1 + D_2$, $F_1 + F_2$, $G_1 + G_2$, $D_1^2 + D_2^2$, $J_1 +J_2$, $D_1F_1 +
	D_2F_2$, $D_1F_2 + D_2F_1$, $D_1G_1 + D_2G_2$, $D_1G_2 + D_2G_1$, $K_1 + K_2$, $F_1^2 + F_2^2$, $D_1^3 +
	D_2^3$, and $D_1^2D_2 + D_1D_2^2$. %, $M_1 + M_2$, $F_1G_1 + F_2G_2$, and $F_1G_2 + F_2G_1$, $G_1^2 + G_2^2$,
	%$D_1J_1 + D_2J_2$, $D_1J_2 + D_2J_1$, $D_1^2F_1 + D_1^2F_2$, $D_1^2F_2 + D_2^2F_1$.
\end{enumerate}
Next, we determine the $\cA(2)$-module structure using a theorem of Quillen.
\begin{thm}[{Quillen's detection theorem~\cite[Proposition 3.1]{Qui71}}]
Let $X$ be a space and let $\Z/k$ act on $X^k$ by cyclic permutations. Let $Y\coloneqq E\Z/k\times_{\Z/k} X^k$,
which is a fiber bundle over $B\Z/k$ with fiber $X^k$. Let $i_1\colon X^k\to Y$ be inclusion of the fiber at the
basepoint and $i_2\colon B\Z/k\times X\to Y$ be induced by the diagonal map; then
\begin{equation}
	(i_1^\ast, i_2^\ast)\colon H^\ast(Y; \Z/k) \longrightarrow H^\ast(X^k;\Z/k)\oplus H^\ast(B\Z/k\times X;\Z/k)
\end{equation}
is injective.
\end{thm}
For us, $k = 2$, $X = B\rE_8$, and $Y = BG$. Thus, to compute Steenrod squares for classes in $H^\ast(BG;\Z/2)$, we
can assume we are in $B\rE_8^2$ if the class is in $(\mathcal H_2)^{\Z/2}$; for $\mathcal H^1$, we also need to
know $\Sq(x)$, and $i_2^\ast$ tells us $\Sq(x) = x+x^2$. Thus we can compute the $\cA(2)$-module structure on
$H^*(BG;\Z/2)$, hence also on $T(-(c_1 + c_2))$; we focus on the latter. Like most calculations of this form, it is
a little tedious but straightforward, and can be done by hand in a reasonable length of time. After working through
the calculation, we have learned the following.
\begin{prop}
\label{e8_semi_coh}
Let $\mathcal M$ be the quotient of $T(-(c_1 +c_2))$ by all elements in degrees $14$ and
higher. Then $\mathcal M$ is the direct sum of the following submodules.
\begin{enumerate}
	\item $\textcolor{BrickRed}{M_1}$, the summand containing the Thom class $U$.
	\item $\textcolor{RedOrange}{M_2}\coloneqq \widetilde H^\ast(\RP^\infty;\Z/2)$ modulo those elements in degrees
	$13$ and above.
	\item $\textcolor{Goldenrod!67!black}{M_3}$, the summand containing $U(D_1^2 + D_2^2)$.
	\item $\textcolor{Green}{M_4}$, the summand containing $UD_1D_2$.
	\item $\textcolor{PineGreen}{M_5}$, the summand containing $UD_1D_2x$.
	\item $\textcolor{MidnightBlue}{M_6}$, the summand containing $U(D_1F_1 + D_2F_2)$.
	\item $\textcolor{Fuchsia}{M_7}$, the summand containing $U(D_1D_2^2 + D_1^2D_2)$.
\end{enumerate}
\end{prop}
We draw this decomposition in \cref{heterotic_A2_module}.

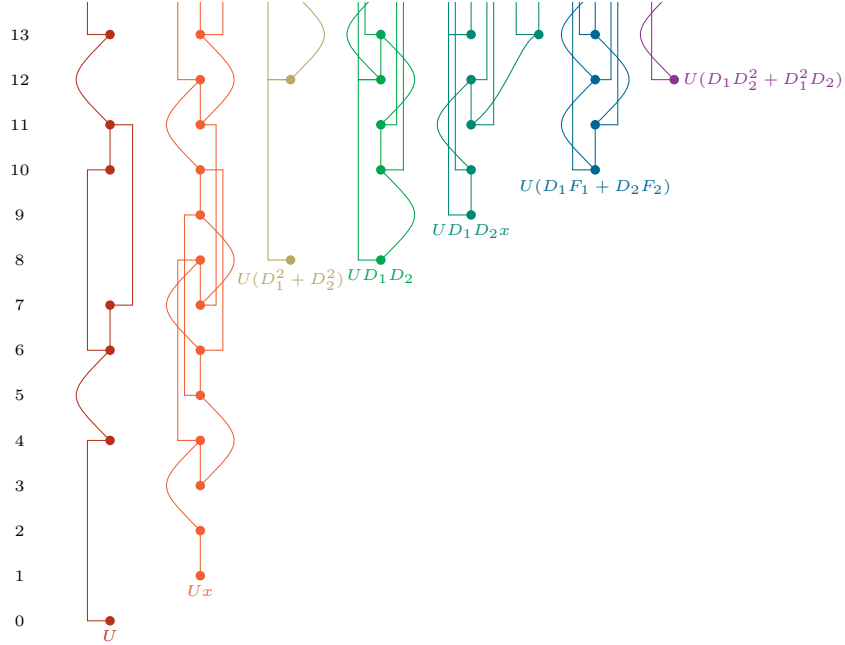
\begin{figure}[h!]
\begin{tikzpicture}[scale=0.6, every node/.style = {font=\tiny}]
	\foreach \y in {0, ..., 13} {
		\node at (-2, \y) {$\y$};
	}
	\begin{scope}
	\clip (-1, -1) rectangle (16.5, 13.7);

	\begin{scope}[BrickRed]
		\tikzptB{0}{0}{$U$}{};
		\foreach \y in {4, 6, 7, 10, 11, 13} {
			\tikzpt{0}{\y}{}{};
		}
		\sqfourL(0, 0);
		\sqtwoL(0, 4);
		\sqone(0, 6);
		\sqfourL(0, 6);
		\sqfourR(0, 7);

		\sqone(0, 10);
		\sqtwoL(0, 11);
		%\sqfourRtwo(0, 10);
		%\sqtwoR(0, 11);
		%\sqfourR(0, 11);
		%\sqtwoL(0, 12);
		%\sqfourL(0, 12);
		%\sqone(0, 13);
		\sqfourL(0, 13);
	\end{scope}

	\begin{scope}[RedOrange]
		\tikzptB{2}{1}{$Ux$}{};
		\foreach \y in {2, ..., 13} {
			\tikzpt{2}{\y}{}{};
		}
		\foreach \y in {1, 3, ..., 13} {
			\sqone(2, \y);
		}
		\foreach \y in {2, 6, 10} {
			\sqtwoL(2, \y);
		}
		\foreach \y in {3, 7, 11} {
			\sqtwoR(2, \y);
		}
		\sqfourL(2, 4);
		\sqfourLtwo(2, 5);
		\sqfourR(2, 6);
		\sqfourRtwo(2, 7);

		\sqfourL(2, 12);
		\sqfourR(2, 13);
	\end{scope}

	\begin{scope}[Goldenrod!67!black]
		\tikzptB{4}{8}{$U(D_1^2 + D_2^2)$}{};
		\tikzpt{4}{12}{}{};
		\sqfourL(4, 8);
		\sqfourL(4, 12);
		\sqtwoR(4, 12);
	\end{scope}

	\begin{scope}[Green]
		\tikzptB{6}{8}{$UD_1D_2$}{};
		\foreach \y in {10, 11, 12, 13} {
			\tikzpt{6}{\y}{}{};
		}
		\sqtwoR(6, 8);
		\sqfourL(6, 8);
		\sqfourR(6, 10);
		\sqone(6, 10);
		\sqfourRtwo(6, 11);
		\sqtwoR(6, 11);
		\sqfourL(6, 12);
		\sqtwoL(6, 12);
		\sqone(6, 12);
		\sqfourLtwo(6, 13);
	\end{scope}
	\begin{scope}[PineGreen]
		\tikzptB{8}{9}{$UD_1D_2x$}{};
		\foreach \y in {10, 11, 12, 13} {
			\tikzpt{8}{\y}{}{};
		}
		\tikzpt{9.5}{13}{}{};
		\sqone(8, 9);
		\sqfourL(8, 9);
		\sqtwoL(8, 10);
		\sqfourLtwo(8, 10);
		\sqone(8, 11);
		\sqtwoCR(8, 11);
		\sqfourR(8, 11);
		\sqone(8, 13);
		\sqone(9.5, 13);
		\sqfourL(8, 13);
		\sqfourRtwo(8, 12);
		\sqfourL(9.5, 13);
	\end{scope}

	\begin{scope}[MidnightBlue]
		\tikzptB{10.75}{10}{$U(D_1F_1 + D_2F_2)$}{};
		\tikzpt{10.75}{11}{}{};
		\tikzpt{10.75}{12}{}{};
		\tikzpt{10.75}{13}{}{};
		
		\sqone(10.75, 10);
		\sqtwoL(10.75, 10);
		\sqtwoR(10.75, 11);
		\sqtwoL(10.75, 12);
		\sqfourRtwo(10.75, 12);
		\sqfourL(10.75, 10);
		\sqfourR(10.75, 11);
		\sqone(10.75, 13);
		\sqfourLtwo(10.75, 13);
	\end{scope}
	\begin{scope}[Fuchsia]
		\tikzptR{12.5}{12}{$U(D_1D_2^2 + D_1^2D_2)$}{};
		\sqtwoL(12.5, 12);
		\sqfourL(12.5, 12);
	\end{scope}

	\end{scope}%for clipping

\end{tikzpicture}
\caption{The $\cA(2)$-module $T(-(c_1 +c_2))$ in low degrees. The pictured submodule contains all classes in
degrees $12$ and below.}
\label{heterotic_A2_module}
\end{figure}

Recall from \cref{spin_splits_off} that $\MTxihet'$ splits as $\MTSpin\vee \cQ$. Since
$\Omega_*^{\xihet}\cong\Omega_*^{\xihet'}$ in the range we need and $\Omega_*^\Spin$ is known thanks to work of
Anderson-Brown-Peterson~\cite{ABP67}, we focus on $\pi_*(\cQ)$. To do so, we will identify the submodule of the
$E_2$-page of the Adams spectral sequence for $\xihet'$ coming from spin bordism via $s_1\colon
\MTSpin\to\MTxihet'$; the $E_2$-page for $\cQ$ is then a complementary submodule.

The canonical $\xihet'$-structure on a spin manifold from \cref{spin_to_xihet} can be rephrased as follows: a spin
structure on a manifold $M$ is equivalent data to: a spin structure on $M$, a map $c\colon M\to K(\Z, 4)$, and a
trivialization of $c - \lambda(M)$. Thus spin structures are twisted string structures in the sense of
\cref{bordism_twists} (in fact the universal twist in the sense of \cref{universal_twist}), so the map
\begin{equation}
	(1, 0)\colon K(\Z, 4)\longrightarrow (K(\Z, 4)\times K(\Z, 4))\times_{\Z/2}E\Z/2 = B
\end{equation}
lifts to a map of $\MTString$-module Thom spectra $s_1\colon \MTSpin\to\MTxihet'$. Naturality of
\cref{fake_shearing_DY}
then tells us the image of $s_1^*$ on mod $2$ cohomology, allowing us to determine which of the summands in
\cref{e8_semi_coh} correspond to $\MTSpin$ and which correspond to $\cQ$. Specifically, the pullback map sends
$x\mapsto 0$, is nonzero on $D_1$, $F_1$, $G_1$, etc., and sends $D_2$, $F_2$, $G_2$, etc., to zero. This implies
that in the direct-sum decomposition $\MTxihet'\simeq\MTSpin\vee \cQ$, the summands $\textcolor{BrickRed}{M_1}$,
$\textcolor{Goldenrod!67!black}{M_3}$, and $\textcolor{MidnightBlue}{M_6}$ come from the cohomology of $\MTSpin$,
and the remaining summands come from the cohomology of $\cQ$.

In order to run the Adams spectral sequence for $\cQ$, we need to compute the Ext of $\textcolor{RedOrange}{M_2}$,
$\textcolor{Green}{M_4}$, $\textcolor{PineGreen}{M_5}$, and $\textcolor{Fuchsia}{M_7}$ over $\cA(2)$. After we
compute this, we will display the $E_2$-page in \cref{heterotic_SS}. For an $\cA(2)$-module $M$,
$\Ext_{\cA(2)}^{*,*}(M, \Z/2)$, which we will usually denote $\Ext_{\cA(2)}(M)$ or $\Ext(M)$, is a bigraded module
over the bigraded $\Z/2$-algebra $\Ext_{\cA(2)}(\Z/2)$; both the algebra and module structures arise from the
\term{Yoneda product}~\cite[\S 4]{Yon54} (see~\cite[\S 4.2]{BC18} for a review). This module structure is helpful
for determining differentials in the Adams spectral sequence: differentials are equivariant with respect to the
action. The module structure also constrains extensions on its $E_\infty$-page.

May (unpublished) and Shimada-Iwai~\cite[\S 8]{SI67} determined the algebra $\Ext_{\cA(2)}(\Z/2)$. We will only
need to track the actions of three elements: $h_0\in \Ext_{\cA(2)}^{1,1}(\Z/2)$, $h_1\in\Ext_{\cA(2)}^{1,2}(\Z/2)$,
and $h_2\in\Ext_{\cA(2)}(\Z/2)$. These elements are in the image of the map
$\Ext_{\cA}(\Z/2)\to\Ext_{\cA(2)}(\Z/2)$ induced by the quotient $\cA\to\cA(2)$, so we do not have to worry
about whether \cref{fake_shearing} is compatible with the $\Ext_{\cA(2)}(\Z/2)$-action on the $E_2$-page of the
Adams spectral sequence. (It is, though.) When we draw Ext charts as in \cref{heterotic_SS}, we denote
$h_0$-actions as vertical lines, $h_1$-actions as diagonal lines with slope $1$, and $h_2$-actions as diagonal
lines with slope $1/3$. When one of these lines is not present, the corresponding $h_i$ acts as $0$.

Often one computes Ext groups of $\cA(2)$-modules using computer programs developed by Bruner~\cite{Bru18} and
Chatham-Chua~\cite{CC21}, or tools such as the May spectral sequence~\cite{May66} or the Davis-Mahowald spectral
sequence~\cite{DM82, MS87} (see also~\cite[Chapter 2]{BR21}) to compute Ext groups of $\cA(2)$-modules, but for the
four modules we care about, we can get away using simpler calculations by hand and computations already in the
literature.
\begin{enumerate}
	\item Davis-Mahowald~\cite[Table 3.2]{DM78} compute $\Ext_{\cA(2)}(\textcolor{RedOrange}{M_2})$ in the
	degrees we need.
	\item In degrees $13$ and below, $\textcolor{Green}{M_4}$ is isomorphic to $\Sigma^8
	(\cA(2)\otimes_{\cA(0)} \Z/2)$; therefore the Ext groups of these two $\cA(2)$-modules, as algebras over
	$\Ext_{\cA(2)}(\Z/2)$, are isomorphic in topological degrees $12$ and below. Thus we can compute with the
	change-of-rings theorem~\eqref{change-of-rings}: as $\Ext_{\cA(2)}(\Z/2)$-algebras,
	\begin{equation}
		\Ext_{\cA(2)}(\cA(2)\otimes_{\cA(0)}\Z/2) \cong \Ext_{\cA(0)}(\Z/2) \cong \Z/2[h_0],
	\end{equation}
	with $h_0\in\Ext^{1,1}$. This identification of $\Ext_{\cA(0)}(\Z/2)$ follows from Koszul duality~\cite[Example
	4.5.5]{BC18}.
	\item $\textcolor{PineGreen}{M_5}$ looks a lot like $\textcolor{RedOrange}{M_2}$, which gives us a technique to
	compute $\Ext_{\cA(2)}(\textcolor{PineGreen}{M_5})$. Specifically, if $\tau_{\le k}M$ denotes the quotient of
	an $\cA(2)$-module $M$ by the submodule of elements in degrees greater than $k$, then there is a short exact
	sequence of $\cA(2)$-modules
	\begin{equation}
	\label{A2_SES}
		\shortexact{\textcolor{RubineRed}{\Sigma^{13}\Z/2}}{\tau_{\le 13}\textcolor{PineGreen}{M_5}}{ \tau_{\le
		13}\textcolor{Periwinkle}{\Sigma^8 M_2}}.
	\end{equation}
	We draw this sequence in \cref{M2M5_LES}, left. \eqref{A2_SES} induces a long exact sequence in Ext groups;
	passage between $M$ and $\tau_{\le 13}M$ does not change Ext groups in degrees $12$ and below, and since we
	only care about degrees $12$ and below, we can and do pass between $\tau_{\le 13}M$ and $M$ without comment.
	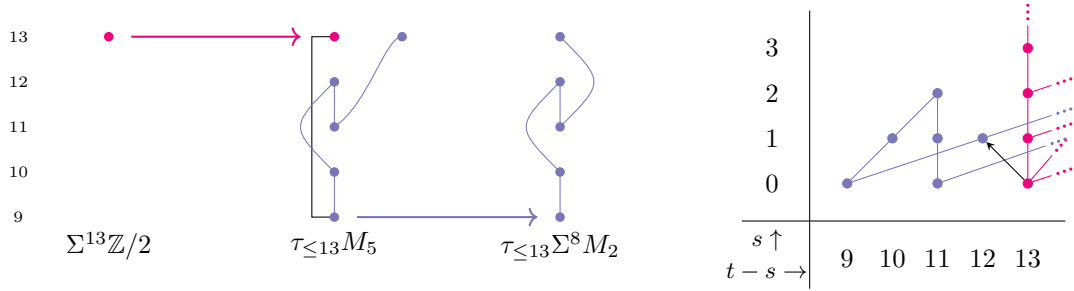
\begin{figure}[h!]
		\begin{subfigure}[c]{0.63\textwidth}
			\begin{tikzpicture}[scale=0.6]
				\begin{scope}[every node/.style = {font=\tiny}]
					\foreach \y in {9, ..., 13} {
						\node at (-7, \y-9) {$\y$};
					}
				\end{scope}
				\sqfourL(0, 0);
				\begin{scope}[RubineRed]
					\foreach \x in {-5, 0} {
						\tikzpt{\x}{4}{}{};
					}
				\draw[->, thick] (-4.5, 4) -- (-0.75, 4);
				\end{scope}
				\begin{scope}[Periwinkle]
					% first the quotient
					\foreach \y in {0, ..., 4} {
						\tikzpt{5}{\y}{}{};
					}
					\sqone(5, 0);
					\sqone(5, 2);
					\sqtwoL(5, 1);
					\sqtwoR(5, 2);

					% then the middle bit
					\foreach \y in {0, ..., 3} {
						\tikzpt{0}{\y}{}{};
					}
					\tikzpt{1.5}{4}{}{};
					\sqone(0, 0);
					\sqtwoL(0, 1);
					\sqone(0, 2);
					\sqtwoCR(0, 2);

					\draw[->, thick] (0.5, 0) -- (4.5, 0);
				\end{scope}
				\node[below=2pt] at (-5, 0) {$\Sigma^{13}\Z/2$};
				\node[below=2pt] at (0, 0) {$\tau_{\le 13}M_5$};
				\node[below=2pt] at (5, 0) {$\tau_{\le 13}\Sigma^8M_2$};
			\end{tikzpicture}
		\end{subfigure}
		\begin{subfigure}[c]{0.3\textwidth}
		%\begin{sseqdata}[name=mod2LES, Adams grading, classes=fill, xrange={9}{13}, yrange={0}{3}, scale=0.6,
		%		x label = {$\displaystyle{s\uparrow \atop t-s\rightarrow}$},
		%		x label style = {font = \small, xshift = -15ex, yshift=6ex}, >=stealth]
		%	\begin{scope}[RubineRed]
		%		\class(13, 0)\AdamsTower{}
%
%				\class(14, 1)\structline(13, 0)(14, 1)
%				\class(16, 1)
%				\class(16, 2)
%				\class(16, 3)
%				\structline(13, 0)(16, 1)
%				\structline(13, 1)(16, 2)
%				\structline(13, 2)(16, 3)
%			\end{scope}
%			\begin{scope}[Periwinkle]
%				\class(9, 0)
%				\class(10, 1)\structline
%				\class(11, 2)\structline
%				\class(11, 1)\structline
%				\class(11, 0)\structline
%				\class(12, 1)\structline(9, 0)(12, 1)
%
%				\class(14, 1)\structline(11, 0)(14, 1)
%				\class(15, 2)\structline(12, 1)(15, 2)
%			\end{scope}
%			\d1(13, 0)
%		\end{sseqdata}
%		\printpage[name=mod2LES, page=1]
		\includegraphics{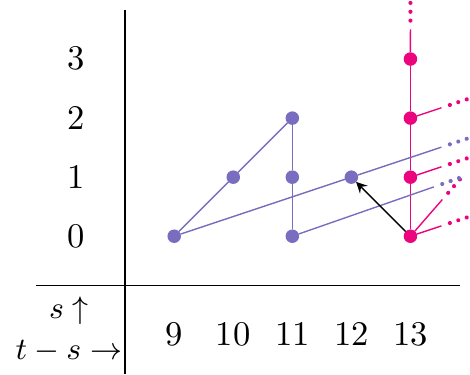}
		\end{subfigure}
	\caption{Left: the short exact sequence~\eqref{A2_SES} of $\cA(2)$-modules. Right: the associated long exact
	sequence in Ext. See the discussion after~\eqref{A2_bdry} for why the pictured boundary map (black arrow) is
	nonzero.}
	\label{M2M5_LES}
	\end{figure}

	We already know $\Ext_{\cA(2)}(\textcolor{RubineRed}{\Z/2})$ and $\Ext_{\cA(2)}(\textcolor{Periwinkle}{M_2})$,
	so we can run the long exact sequence associated to~\eqref{A2_SES} to compute
	$\Ext_{\cA(2)}(\textcolor{PineGreen}{M_5})$ in degrees $12$ and below; we draw this long exact sequence in
	\cref{M2M5_LES}, right. In the range we care about, there is exactly one boundary map that is not forced to be
	zero for degree reasons, namely
	\begin{equation}
	\label{A2_bdry}
		\partial\colon\Ext_{\cA(2)}^{0,13}(\textcolor{RubineRed}{\Sigma^{13}\Z/2})\longrightarrow
		\Ext_{\cA(2)}^{1,13} (\textcolor{Periwinkle}{\Sigma^8 M_2});
	\end{equation}
	it must be nonzero, because that is the only way to obtain $\Ext_{\cA(2)}^{0,13}(\textcolor{PineGreen}{M_5}) =
	\Hom_{\cA(2)}(\textcolor{PineGreen}{M_5}, \Sigma^{13}\Z/2) = 0$, and by inspection of
	\cref{heterotic_A2_module} this Hom group vanishes.
	\item If $C\eta\coloneqq \Sigma^{-2}\widetilde H^\ast(\CP^2; \Z/2)$, there is a $14$-connected quotient map
	$\textcolor{Fuchsia}{M_7}\to \Sigma^{12}C\eta$, so $\Ext_{\cA(2)}(\Sigma^{12} C\eta)$ and
	$\Ext_{\cA(2)}(\textcolor{Fuchsia}{M_7})$ do not differ in the range we care about. Bruner-Rognes~\cite[Figure
	0.15]{BR21} compute $\Ext_{\cA(2)}(C\eta)$.
\end{enumerate}
Using these computations, we obtain the following description of the $E_2$-page of the Adams spectral sequence for
the summand $\cQ$ of $\MTxihet'$.
\begin{prop}
The $E_2$-page of the Adams spectral sequence for $\cQ$ in topological degrees $12$ and below is as given in
\cref{heterotic_SS}. In particular, in this range, the $E_2$-page is generated as an $\Ext_{\cA(2)}(\Z/2)$-module
by eight elements: $p_1\in\Ext^{0,1}$, $p_3\in\Ext^{0,3}$, $p_7\in\Ext^{0,7}$, $a\in\Ext^{0,8}$, $b\in\Ext^{2,10}$,
$c\in\Ext^{0,9}$, $d\in\Ext^{0,11}$, and $e\in\Ext^{0,12}$.
\end{prop}

%Thus we can draw the $E_2$-page of the Adams spectral sequence in \cref{heterotic_SS}. In the range $t-s\le 12$,
\begin{figure}[h!]
\includegraphics{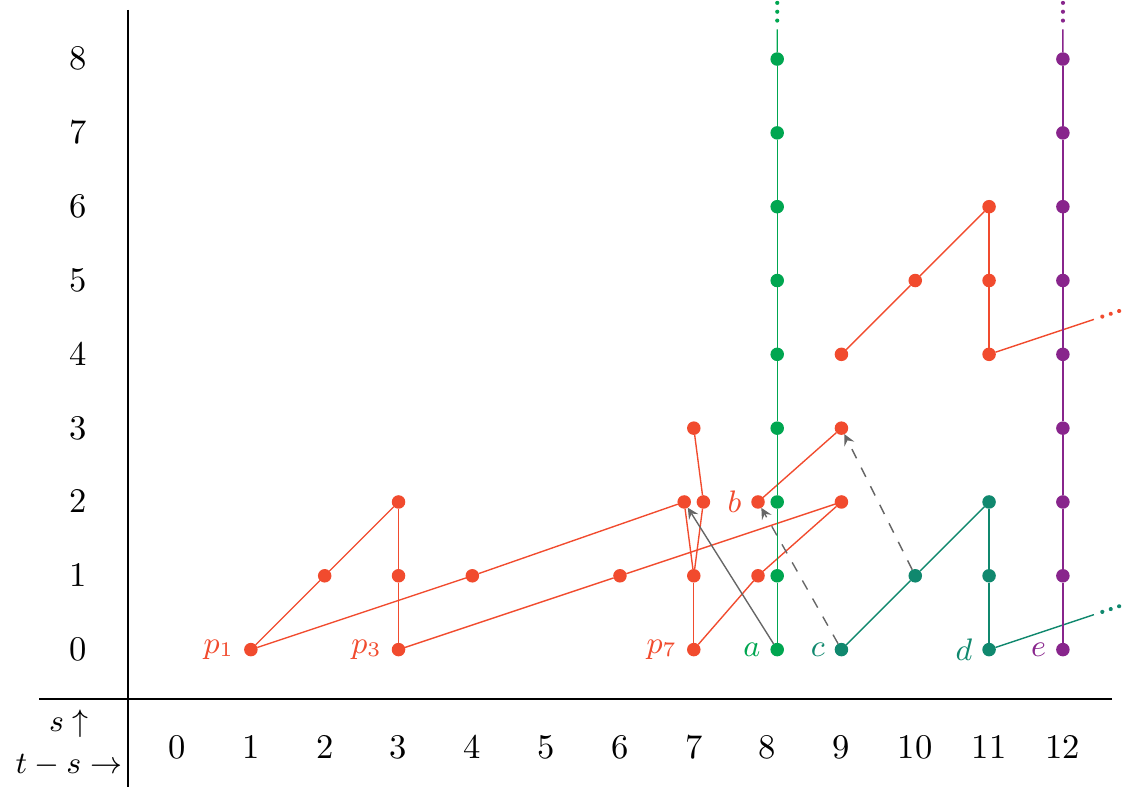}
\caption{In \cref{spin_splits_off}, we showed $\MTxihet'\simeq\MTSpin\vee\cQ$; this figure denotes the $E_2$-page
of the Adams spectral sequence computing $\pi_*(\cQ)$ in degrees $12$ and below. This corresponds to a subset of the
summands in \cref{heterotic_A2_module}. In \cref{g_to_o_nonzero}, we show that the solid gray differential
beginning at $a$ is nonzero; we leave open the other two differentials, which are dashed in this figure.}
%\cref{g_to_o_nonzero,orange_teal_diffs}.}
%$\Ext_{\cA(2)}(T(-(c_1 + c_2)))$, which by \cref{fake_shearing} coincides with the $E_2$-page of
%the Adams spectral sequence computing $\Omega_*^{\xihet}$ in degrees $15$ and below.}
\label{heterotic_SS}
\end{figure}

There are plenty of differentials in this Adams spectral sequence which could be nonzero, even when
we take into account the fact that Adams differentials commute with $h_0$, $h_1$, and $h_2$:
\begin{enumerate}[label=(D\arabic*)]
	\item\label{green_to_orange_d2_1} $d_2\colon E_2^{0,8}\to E_2^{2,9}$, whose value on $a$
	could be $h_2^2 p_1$, $h_0^2p_7$, or a linear combination of
	those two elements.
	\item\label{green_to_orange_d2_2} $d_2\colon E_2^{1,9}\to E_2^{3,10}$, which could send
	$h_0a$ or $h_1p_7$ to $h_0^3 p_7$.
	\item\label{h1_orange_d2_family} $d_2\colon E_2^{0,9}\to E_2^{2,8}$ and $d_2\colon E_2^{1,11}\to E_2^{3,12}$,
	intertwined by an $h_1$-action, which could send $c\mapsto b$ and $h_1c\mapsto h_1b$.
	\item\label{purple_d2} $d_2\colon E_2^{0,12}\to E_2^{2,13}$, which could send $e\mapsto h_1^2c = h_0^2d$.
	\item\label{green_to_orange_d3} If the differentials in~\ref{green_to_orange_d2_1}
	and~\ref{green_to_orange_d2_2} vanish, $d_3\colon E_3^{0,8}\to E_3^{3,10}$ could be nonzero on $a$.
	\item\label{deep_purple} If the differential in~\ref{purple_d2} vanishes, $d_5\colon E_5^{0,12}\to
	E_5^{5,16}$ (and its image under $h_0$) or $d_6\colon E_6^{0,12}\to E_6^{6,17}$ could be nonzero.
\end{enumerate}
% orange
\begin{lem}
\label{orange_diffs}
The differentials~\ref{green_to_orange_d2_2}, \ref{green_to_orange_d3}, and~\ref{deep_purple} vanish.
\end{lem}
\begin{proof}
Our strategy is to use the fact that $\Ghet\to\Z/2$ splits to zero out differentials. This splitting does not
extend to a splitting of $\MTxihet$, but it will be close enough.

The inclusion $\iota\colon \Z/2\inj \Ghet$ defines a map $\iota'\colon \MTString\wedge B\Z/2\to\MTxihet$ which on
Adams $E_2$-pages is precisely the inclusion of the summand $\Ext(\textcolor{RedOrange}{M_2})$. Quotienting $\Ghet$
by $\T[1]$, then by $\rE_8\times\rE_8$, produces a map
\begin{equation}
	p\colon \MTxihet\underset{\eqref{forget_string}}{\overset{\phi}{\longrightarrow}} \MTSpin\wedge
	(B((\rE_8\times\rE_8)\rtimes\Z/2))_+\longrightarrow \MTSpin\wedge (B\Z/2)_+,
\end{equation}
and $p\circ\iota\colon\MTString\wedge (B\Z/2)\to\MTSpin\wedge (B\Z/2)_+$ is the usual map $\MTString\to\MTSpin$
together with the addition of a basepoint. This means that any element of $\widetilde\Omega_*^\String(B\Z/2)$ whose
image in $\widetilde\Omega_*^\Spin(B\Z/2)$ is nonzero must also be nonzero in $\Omega_*^{\xihet}$, which kills many
differentials to or from $\Ext(\textcolor{RedOrange}{M_2})$. To produce such elements, study the map of Adams
spectral sequences induced by $p\circ\iota$, which on $E_2$-pages is the map
\begin{equation}
\label{ExtA2A1map}
	\Ext_{\cA(2)}(\widetilde H^*(B\Z/2;\Z/2))\longrightarrow \Ext_{\cA(1)}(H^*(B\Z/2;\Z/2)).
\end{equation}
Davis-Mahowald~\cite[Table 3.2]{DM78} compute $\Ext_{\cA(2)}(H^*(B\Z/2;\Z/2))$ in the degrees we need, and
Gitler-Mahowald-Milgram~\cite[\S 2]{GMM68} compute $\Ext_{\cA(1)}(H^*(B\Z/2;\Z/2))$. We draw the
map~\eqref{ExtA2A1map} in \cref{orange_diffs_pic}. All differentials in the spectral sequence over $\cA(1)$ vanish
using $h_0$- and $h_1$-equivariance, and by inspection there are no hidden extensions. Therefore we can identify
some classes which survive $p\circ\iota$ and use this to trivialize some differentials in \cref{heterotic_SS}.
\begin{itemize}
	\item By computing the image of $p\circ\iota$ on Ext groups, we learn that the map
	$\widetilde\Omega_7^\String(B\Z/2)\to\widetilde\Omega_7^\Spin(B\Z/2)$ can be identified with the map
	$\Z/16\oplus\Z/2\to\Z/16$ sending $(1, 0)\mapsto 1$ and $(0, 1)\mapsto 0$.\footnote{Alternatively, one could
	show that the $\Z/16\subset\widetilde\Omega_7^\String(B\Z/2)$ is mapped injectively into
	$\widetilde\Omega_7^\Spin(B\Z/2)$ by checking on a generator. One can show that $\RP^7$ admits a string
	structure; then the generator of that $\Z/16$ subgroup of $\widetilde\Omega_7^\String(B\Z/2)$ is $\RP^7$ with
	its nontrivial principal $\Z/2$-bundle. Its image in $\widetilde\Omega_7^\Spin(B\Z/2)$ has order at least $16$,
	because the $\eta$-invariant of a suitable twisted Dirac operator associated to the $\Z/2$-bundle defines a
	bordism invariant $\Omega_7^\Spin(B\Z/2)\to\R/\Z$, and on $(\RP^7, S^7\to\RP^7)$, this $\eta$-invariant is
	$\ell/16\bmod 1$ for some odd $\ell$, as follows from a formula of Donnelly~\cite[Proposition 4.1]{Don78}.}
	Therefore, any differential to or from the four summands in topological degree $7$ linked by $h_0$-actions must
	vanish, including~\ref{green_to_orange_d2_2} and~\ref{green_to_orange_d3}.
	\item Similarly, the map $\widetilde\Omega_{11}^\String(B\Z/2)\to\widetilde\Omega_{11}^\Spin(B\Z/2)$ can be
	identified with the inclusion $\Z/8\inj \Z/128 \oplus \Z/8\oplus \Z/2$ sending $1\mapsto (16, 0, 0)$, which
	follows either by computing $p\circ\iota$ on Ext groups or computing $\eta$-invariants on the generator of
	$\widetilde\Omega_{11}^\String(B\Z/2)$, which can be taken to be the product of $\RP^3$ with a Bott
	manifold.\footnote{All orientable $3$-manifolds have trivializable tangent bundles, hence string structures;
	for a construction of a Bott manifold with string structure, see~\cite[\S 5.3]{FH21b}.}
	Thus~\ref{deep_purple} vanishes.
	\qedhere
\end{itemize}
\end{proof}
\begin{figure}[h!]
\begin{subfigure}[c]{0.48\textwidth}
%\begin{sseqdata}[name=StringZ, classes = fill, scale=0.5, xrange={1}{11}, yrange={0}{7}, Adams grading,
%>=stealth,
%x label = {$\displaystyle{s\uparrow \atop t-s\rightarrow}$},
%x label style = {font = \small, xshift = -24ex, yshift=6ex}]
%\begin{scope}[RedOrange]
%	\class(1, 0)
%	\class(2, 1)\structline
%	\class(3, 2)\structline
%	\class(3, 1)\structline
%	\class(3, 0)\structline
%
%	\class[fill=none](4, 1)\structline(1, 0)(4, 1)
%	\class[fill=none](6, 1)\structline(3, 0)(6, 1)
%	\class(7, 0)
%	\class(7, 1)\structline
%	\class[fill=none](7, 2)
%	\class(7, 2)
%	\class(7, 3)
%	\structline(7, 1)(7, 2, -1)
%	\structline(7, 1)(7, 2, 1)
%	\structline(7, 2, -1)(7, 3)
%	\structline(4, 1)(7, 2, 1)
%	\class[fill=none](8, 1)\structline(7, 0)(8, 1)
%	\class[fill=none](9, 2)\structline
%	\class[fill=none](8, 2)
%	\class[fill=none](9, 3)\structline
%	\structline(6, 1)(9, 2)
%	\class(9, 4)
%	\class(10, 5)\structline
%	\class(11, 6)\structline
%	\class(11, 5)\structline
%	\class(11, 4)\structline
%	\class(14, 5)\structline
%\end{scope}
%\end{sseqdata}
%\printpage[name=StringZ, page=2]
\includegraphics{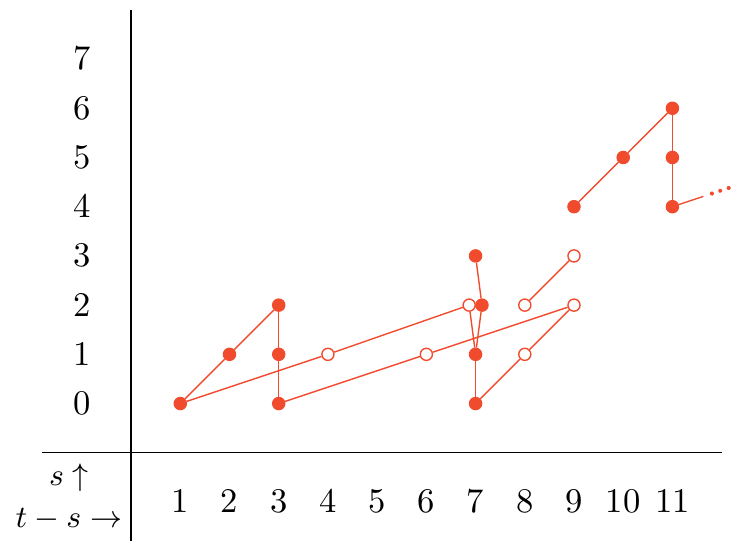}
\end{subfigure}
\begin{subfigure}[c]{0.48\textwidth}
%\begin{sseqdata}[name=SpinZ, classes = fill, scale=0.5, xrange={1}{11}, yrange={0}{7}, Adams grading,
%>=stealth,
%x label = {$\displaystyle{s\uparrow \atop t-s\rightarrow}$},
%x label style = {font = \small, xshift = -24ex, yshift=6ex}]
%\begin{scope}[RedOrange]
%	\class(1, 0)
%	\class(2, 1)\structline
%	\class(3, 2)\structline
%	\class(3, 1)\structline
%	\class(3, 0)\structline
%
%	\class(7, 0)
%	\class(7, 1)\structline
%	\class(7, 2)
%	\class(7, 3)
%	\structline(7, 1)(7, 2, 1)
%	\structline(7, 2, -1)(7, 3)
%	\class(9, 4)
%	\class(10, 5)\structline
%	\class(11, 6)\structline
%	\class(11, 5)\structline
%	\class(11, 4)\structline
%	\class[fill=gray!40!white](11, 3)\structline
%	\class[fill=gray!40!white](11, 2)\structline
%	\class[fill=gray!40!white](11, 1)\structline
%	\class[fill=gray!40!white](11, 0)\structline
%\end{scope}
%\end{sseqdata}
%\printpage[name=SpinZ, page=2]
\includegraphics{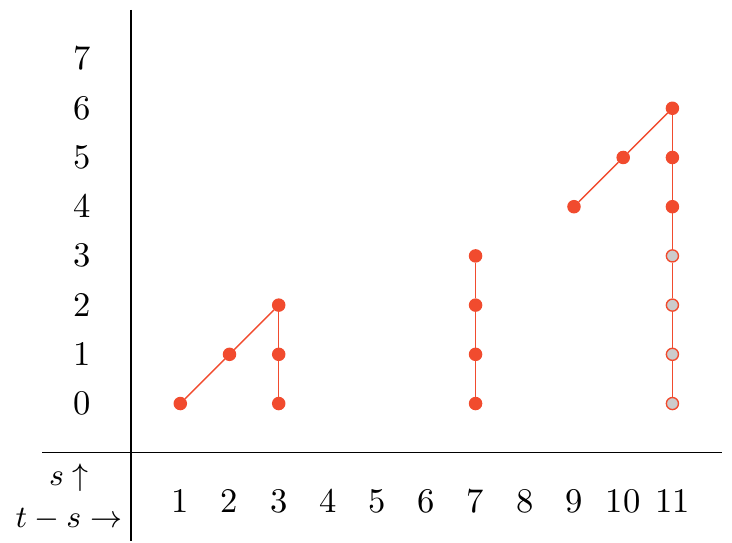}
\end{subfigure}
\caption{Left: $\Ext_{\cA(2)}(\widetilde H^*(B\Z/2; \Z/2), \Z/2)$, the $E_2$-page of the Adams spectral sequence
computing $\widetilde\Omega_*^\String(B\Z/2)_2^\wedge$. Filled dots have nonzero image in $\Ext_{\cA(1)}$; unfilled
dots are the kernel.
Right: $\Ext_{\cA(1)}(\widetilde H^*(B\Z/2; \Z/2), \Z/2)$, a summand of the $E_2$-page of the Adams spectral
sequence computing $\widetilde\Omega_*^\Spin(B\Z/2)_2^\wedge$. Filled dots are in the image of the map from
$\Ext_{\cA(2)}$; gray dots are the cokernel.
This map of spectral sequences is used in the proof of \cref{orange_diffs}.
}
\label{orange_diffs_pic}
\end{figure}

\begin{lem}
\label{g_to_o_nonzero}
The differential~\ref{green_to_orange_d2_1} is nonzero; specifically, $d_2(a) = h_2^2p_1$.
\end{lem}
We will deduce this from the following fact.
\begin{prop}
\label{4d_goings_on}
The map $\Omega_4^{\xi^{r,\mathrm{het}}}\to\Omega_4^{\xihet}$ is surjective, at least after $2$-completion.
\end{prop}
Recall that $\xi^{r,\mathrm{het}}$ is the analogue of $\xihet$ but with $(\rE_8\times\rE_8)\rtimes\Z/2$ replaced
with $\rE_8\times\rE_8$.
\begin{proof}[Proof of \cref{g_to_o_nonzero} assuming \cref{4d_goings_on}]
In this proof, implicitly $2$-complete all abelian groups.
If $d_2(a) = 0$, then $h_2^2 p_1\in E_2^{2,9}$ survives to the $E_\infty$-page, so the $h_2$-action $E_\infty^{1,5}\to
E_\infty^{2,9}$ is nonzero. This lifts to imply that taking the product with $S^3$ with string structure induced
from its Lie group framing, which defines a map $\Omega_4^{\xihet}\to\Omega_7^{\xihet}$, is also nonzero. Direct
products with framed manifolds correspond to action by elements of $\pi_*(\mathbb S)$ on homotopy groups, so this
product with $S^3$ is natural with respect to maps of spectra.

Since $\Omega_4^{\xi^{r,\mathrm{het}}}\to\Omega_4^{\xihet}$ is surjective, we may compute the product with $S^3$ as
a map
\begin{equation}
	\bl\times S^3\colon \Omega_4^{\xi^{r,\mathrm{het}}}\longrightarrow \Omega_7^{\xi^{r,\mathrm{het}}}
\end{equation}
and then map back to $\Omega_7^{\xihet}$. However, as we noted in \cref{restricted_het_is_easy},
$\Omega_7^{\xi^{r,\mathrm{het}}}\cong\Omega_7^\Spin(K(\Z, 4))$, and Stong~\cite{Sto86} showed
$\Omega_7^{\Spin}(K(\Z, 4)) = 0$. Thus taking the product with $S^3$ is the zero map
$\Omega_4^{\xihet}\to\Omega_7^{\xihet}$, which is incompatible with $d_2(a)$ vanishing.
\end{proof}
\begin{proof}[Proof of \cref{4d_goings_on}]
Let $F$ be the fiber of the map $\phi\colon \mathit{MT\xi}^{r,\mathrm{het}}\to\MTxihet$, so that there is a long
exact sequence
\begin{equation}
\label{smith_unsaid}
	\dotsb \longrightarrow \Omega_4^{\xi^{r,\mathrm{het}}} \overset\phi\longrightarrow \Omega_4^{\xihet}
	\overset\partial\longrightarrow \pi_3(F) \longrightarrow \Omega_3^{\xi^{r,\mathrm{het}}}
	\longrightarrow\dotsb
\end{equation}
We will show $\pi_3(F)_2^\wedge = 0$, which implies the proposition statement by exactness. To do so, we must understand
$F$.

Let $V$ be the rank-zero stable vector bundle on $B\Ghet$ classified by the map $\xihet\colon B\Ghet\to B\O$ and
let $\sigma\to B\Ghet$ be the line bundle classified by the map quotienting by $\T[1]$, then by $\Spin$, then by
$\rE_8^2$:
\begin{equation}
\label{BGhet_to_Z2}
	B\Ghet\longrightarrow B(\rE_8^2\rtimes\Z/2)\longrightarrow B\Z/2.
\end{equation}
Then, inclusion of the zero section of $\sigma$ defines a map of spaces over $\Z\times B\O$: $\phi\colon (B\Ghet,
V)\to (B\Ghet, V\oplus\sigma)$. Here we use the notation $(B, \xi)$ to denote a space $B$ and a map $\xi\colon
B\to \Z\times B\O$, and we use $\Z\times B\O$ instead of $B\O$ because $\sigma$ is not rank $0$. Let $M^-$ denote
the Thom spectrum of $V\oplus\sigma\colon B\Ghet\to \Z\times B\O$, and let $\widetilde\phi\colon\MTxi\to M^-$
denote the map of Thom spectra induced by $\phi$; we claim $F\simeq \Sigma^{-1}M$. To see this, we
will use a theorem in~\cite{DDKLPTT} which identifies the fiber of $\widetilde\phi$ as the map
$\mathit{MT\xi}^{r,\mathrm{het}}\to\MTxihet$. Specifically, \cite{DDKLPTT} shows that the fiber of $\widetilde\phi$
is the Thom spectrum of the pullback of $V$ to the sphere bundle $S(\sigma)$ of $\sigma$. This sphere bundle is the
pullback of the universal sphere bundle over $B\Z/2$ by the classifying map of $\sigma$:
% https://q.uiver.app/?q=WzAsNCxbMCwwLCJTKFxcc2lnbWEpIl0sWzEsMCwiUyhMKVxcc2ltZXEgRVxcWi8yIl0sWzEsMSwiQlxcWi8yIl0sWzAsMSwiQlxcR2hldCJdLFszLDJdLFswLDNdLFsxLDJdLFswLDFdLFswLDIsIiIsMSx7InN0eWxlIjp7Im5hbWUiOiJjb3JuZXIifX1dXQ==
\begin{equation}
\label{ghet_sphere}
\begin{tikzcd}
	{S(\sigma)} & {S(L)\simeq E\Z/2} \\
	B\Ghet & {B\Z/2}
	\arrow[from=2-1, to=2-2]
	\arrow[from=1-1, to=2-1]
	\arrow[from=1-2, to=2-2]
	\arrow[from=1-1, to=1-2]
	\arrow["\lrcorner"{anchor=center, pos=0.125}, draw=none, from=1-1, to=2-2]
\end{tikzcd}
\end{equation}
The sphere bundle of the tautological line bundle $L\to B\Z/2$ is $E\Z/2\to B\Z/2$, which is contractible, so
the pullback diagram~\eqref{ghet_sphere} simplifies to a fiber diagram, and the sphere bundle is the fiber
of~\eqref{BGhet_to_Z2}. Since~\eqref{BGhet_to_Z2} was induced from a group homomorphism by taking classifying
spaces, one can compute its fiber by taking the classifying space of the kernel of the homomorphism, which is
$\mathcal S(\Spin\times \rE_8^2, c_1 + c_2 - \lambda)$. In \cref{no_Z2} we saw that applying the Thom spectrum
functor to $B\mathcal S(\Spin\times \rE_8^2, c_1 + c_2 - \lambda)\to B\Ghet$, i.e.\ to the map $S(\sigma)\to
B\Ghet$, produces the map $\mathit{MT\xi}^{r,\mathrm{het}}\to\MTxihet$, and therefore the fiber of this map is
$\Sigma^{-1}M^-$.

To finish the proof, attack $F$ with the Adams spectral sequence, using its description as the Thom spectrum
$\Sigma^{-1}M$ to get a description in terms of Ext of an $\cA(2)$-module by using~\cite{DY} again. Recall from
\cref{M2M5_LES}, left, the $\cA(2)$-module $\tau_{\le 13} M_5$; the result of the computation here is that the
$\cA(2)$-module relevant for computing $\pi_*(F)_2^\wedge$ agrees with $\Sigma^{-9}(\tau_{\le 13}M_5)$ in degrees
$4$ and below. Then, \cref{M2M5_LES}, right, computes $\Ext_{\cA(2)}(\Sigma^{-9}(\tau_{\le 13}M_5))$, which is the
$E_2$-page of the Adams spectral sequence computing $\pi_*(F)_2^\wedge$, in degrees $3$ and below (shift the
topological degree of everything in \cref{M2M5_LES}, right, down by $9$). The $E_2$-page vanishes in topological
degree $3$, which implies $\pi_3(F)^\wedge_2 = 0$.
\end{proof}

\begin{lem}
\label{purple_lem}
The differential~\ref{purple_d2} vanishes.
\end{lem}
\begin{proof}
The source of this differential is $E_2^{0,12}\cong \textcolor{Purple}{\Z/2}\cdot e$ in Adams filtration zero.
Classes $\alpha$ in Adams filtration $0$ are canonically identified with classes $c_\alpha$ forming a subgroup of
mod $2$ cohomology, and $\alpha$ survives to the $E_\infty$-page if and only if the bordism invariant $\int
c_\alpha$ is nonzero. Here, $\alpha = e$ and $c_\alpha = D_1D_2^2 + D_1^2D_2$, so
our differential vanishes if and only if $e$ survives to the $E_\infty$-page
if and only if the following invariant is nonzero:
%Elements in filtration zero are associated to mod $2$ cohomology classes, and looking
%at~\cref{heterotic_A2_module}, the generator of this $\textcolor{Purple}{\Z/2}$ corresponds to the cohomology class
%$D_1D_2^2 + D_1^2 D_2$. The corresponding element of $E_2^{0,12}$
%survives to the $E_\infty$-page if and only if the corresponding bordism invariant
\begin{equation}
	\int \paren{D_1D_2^2 + D_1^2 D_2}\colon \Omega_{12}^{\xihet}\longrightarrow\Z/2.
\end{equation}
We will produce a manifold on which this invariant is nonzero.

The quaternionic projective plane $\HP^2$ has $H^*(\HP^2;\Z)\cong\Z[x]/(x^3)$ with $\abs x = 4$ and $\lambda(\HP^2)
= x$~\cite[\S 15.5, \S15.6]{BH58} (see also~\cite[\S 5.2]{FH21b}). The Künneth formula tells us $H^*(\HP^2\times
S^4;\Z)\cong\Z[x, y]/(x^3, y^2)$, with $\abs y = 4$; since $TS^4$ is stably trivial, $\lambda(S^4)$ vanishes and
the Whitney sum formula (\cref{lambda_whitney}) implies $\lambda(\HP^2\times S^4) = x$.

To define a $\xihet$-structure on $\HP^2\times S^4$, it suffices to produce two $\rE_8$-bundles $P,Q\to\HP^2\times
S^4$ and a trivialization of $\lambda(\HP^2\times S^4) - c(P) -c(Q)$. Since we can freely prescribe $c(P)$ and
$c(Q)$, choose $P$ and $Q$ such that $c(P) = y$ and $c(Q) = x-y$; then $\lambda(\HP^2\times S^4) - c(P) - c(Q) =
0$, so we can choose a trivialization. Since $D_1 = c(P)\bmod 2$ and $D_2 = c(Q)\bmod 2$,
\begin{equation}
	\int_{\HP^2\times S^4}\paren{D_1D_2^2 + D_1^2D_2} = \paren{\int_{\HP^2\times S^4} (yx^2 + xy^2)}\bmod 2 = 1.
	\qedhere
\end{equation}
\end{proof}

Now we have to tackle extension questions. In this part of the computation, it will be helpful to reference
\cref{heterotic_SS}, as we will use the description of the $E_\infty$-page of this spectral sequence several times
while addressing extension questions.
\begin{lem}
\label{easier_extensions}
In degrees $10$ and below, all extension questions in the Adams spectral sequence for $\pi_*(\cQ)_2^\wedge$ either
split or are detected by $h_0$ on the $E_\infty$-page, except possibly for the extensions involving the classes
$c\in E_\infty^{0,9}$, $h_1^2 p_7\in E_\infty^{2,11}$, and $h_1b\in E_\infty^{3,12}$.
\end{lem}
The classes $h_1b$ and $c$ may vanish on the $E_\infty$-page, depending on the fate of the differentials
in~\ref{h1_orange_d2_family}.
\begin{proof}
The $h_0$-action alone solves all extensions in this range except in degrees $8$, $9$, and $10$.
%In all degrees except $9$, and maybe also $8$ and $10$ depending on the fate of~\ref{h1_orange_d2_family},
%the extension problem for $\pi_*(\cQ)_2^\wedge$ can be solved completely using the fact that an $h_0$-action on the
%$E_\infty$-page lifts to multiplication by $2$ in homotopy groups.
%In degree $11$, the $h_0$-action implies
%$\pi_{11}(\cQ)_2^\wedge$ is an extension of $\textcolor{PineGreen}{\Z/8}$ by $\textcolor{RedOrange}{\Z/8}$, which
%together with $\Omega_{11}^\Spin = 0$ gives the possibilities for $A$ in the theorem statement.

If the $d_2$s in~\ref{h1_orange_d2_family} vanish, there is an extension question in degree $8$. The
$h_0$-actions in the tower generated by $h_0a$ lift to produce a $\Z$ in $\Omega_8^\xihet$, so the only question is
whether there is an extension involving $h_1p_7$ and $b$. Suppose this
extension does not split, so $\pi_8(\cQ)_2^\wedge\cong\textcolor{Green}{\Z}\oplus \textcolor{RedOrange}{\Z/4}$. We
can choose a generator $x$ of this $\textcolor{RedOrange}{\Z/4}$ such that the image of $x$ in the Adams
$E_\infty$-page is $h_1p_7\in E_\infty^{1,9}$; since this is $h_1$ times another class on the $E_\infty$-page, $x$
is $\eta$ times a class $y\in \pi_7(\cQ)_2^\wedge$, where $\eta$ is the generator of $\pi_1(\mathbb S)\cong\Z/2$.
Since $2\eta = 0$, $2x = 2\eta y = 0$; since $x$ was supposed to generate a $\textcolor{RedOrange}{\Z/4}$, this is
a contradiction, and therefore this extension splits.

The same trick splits all extensions in degree $10$, and all extensions involving the class in $E_\infty^{4,13}$.
\end{proof}
\begin{prop}
\label{RP2_xtn}
All extension questions in $\pi_9(\cQ)_2^\wedge$ split, so $\pi_9(\cQ)_2^\wedge\cong (\Z/2)^{\oplus 4}$ if the
differentials in~\ref{h1_orange_d2_family} vanish, and $\pi_9(\cQ)_2^\wedge\cong (\Z/2)^{\oplus 2}$ if they do
vanish.
\end{prop}
\begin{proof}
If the differentials in~\ref{h1_orange_d2_family} do not vanish, this is a consequence of \cref{easier_extensions},
so assume that those differentials vanish.

First suppose we can split all extensions involving $c$. Then the only extension remaining is between $h_1^2 p_7$
and $h_1b$. In \cref{easier_extensions}, we split the extension between $h_1p_7$ and $b$, so the classes $h_1p_7$
and $b$ lift to classes $\underline{h_1p_7}$, resp.\ $\underline b$, which generate a
$\textcolor{RedOrange}{\Z/2}\oplus \textcolor{RedOrange}{\Z/2}\subset\pi_8(\cQ)_2^\wedge$.  The action by $h_1$
lifts to imply that the images of $\eta\cdot \underline{h_1p_7}$ and $\eta\cdot\underline b$ in the $E_\infty$-page
are $h_1^2p_7$, resp.\ $h_1b$, and $\eta$ carries the $\textcolor{RedOrange}{\Z/2} \oplus
\textcolor{RedOrange}{\Z/2}$ generated by $\underline{h_1p_7}$ and $\underline b$ to a $\textcolor{RedOrange}{\Z/2}
\oplus \textcolor{RedOrange}{\Z/2} \subset\pi_9(\cQ)^\wedge_2$ generated by $\eta \underline{h_1p_7}$ and
$\eta\underline b$, thus splitting the extension between $h_1^2p_7$ and $h_1b$.

Now we need to prove that $c$ lifts to a class $\underline c$ such that $2\underline c = 0$. Let $X$ be the pullback
% https://q.uiver.app/?q=WzAsNCxbMCwwLCJYIl0sWzEsMCwiQlxcR2hldCJdLFswLDEsIlxcUlBeMiJdLFsxLDEsIkJcXFovMiJdLFsyLDNdLFswLDJdLFsxLDNdLFswLDFdLFswLDMsIiIsMSx7InN0eWxlIjp7Im5hbWUiOiJjb3JuZXIifX1dXQ==
\begin{equation}\label{ghet_p2}
\begin{tikzcd}
	X & B\Ghet \\
	{\RP^2} & {B\Z/2}
	\arrow[from=2-1, to=2-2]
	\arrow[from=1-1, to=2-1]
	\arrow[from=1-2, to=2-2]
	\arrow[from=1-1, to=1-2]
	\arrow["\lrcorner"{anchor=center, pos=0.125}, draw=none, from=1-1, to=2-2]
\end{tikzcd}\end{equation}
and let $\xi\colon X\to B\O$ be the pullback of $\xihet$ to $X$. Both vertical arrows in~\eqref{ghet_p2} are
fibrations with fiber $B\rE_8^2$; using the induced map of Serre spectral sequences, we learn $H^*(X;\Z/2) \cong
H^*(B\Ghet;\Z/2)/(x^3)$, where $x\in H^1(B\Ghet;\Z/2)$ is the generator. One can replay the whole argument we ran
with $\xi$ in place of $\xihet$, defining $\xi'$ analogously to $\xihet'$, and deduce the following.
\begin{enumerate}
	\item The map $c\colon B\rE_8\to K(\Z, 4)$ induces an isomorphism $\Omega_*^\xi\to\Omega_*^{\xi'}$ in degrees
	$14$ and below,
	\item there is a spectrum $\cQ'$ and a splitting $\mathit{MT\xi}\simeq \MTSpin\vee \cQ'$, and
	\item the map $X\to B\Ghet$ induces a map $\mathit{MT\xi}'\to \MTxihet'$ which is the identity on the $\MTSpin$
factors and sends $\cQ\to\cQ'$.
\end{enumerate}
The analogue of \cref{e8_semi_coh} for $\xi'$ is exactly the same, except replacing
$\textcolor{RedOrange}{M_2}$ with $\textcolor{RedOrange}{\Sigma C2}$ and $\textcolor{PineGreen}{M_5}$ with
$\textcolor{PineGreen}{\Sigma^9 C2}$, where $C2$ is the $\cA(2)$-module $\Sigma^{-1}\widetilde H^*(\RP^2;\Z/2)$.
Bruner-Rognes~\cite[\S 6.1]{BR21} compute $\Ext_{\cA(2)}(C2)$, and using that we can draw the $E_2$-page of the
Adams spectral sequence computing $\pi_*(\cQ')_2^\wedge$ in \cref{RP2_fig}. For the classes $p_1$, $a$, and $c$ we
considered in the $E_2$-page of the Adams spectral sequence for $\cQ$, let $p_1'$, $a'$, and $c'$ be the
corresponding classes in the $E_2$-page for $\cQ'$: they live in the same bidegrees and the map $\cQ'\to\cQ$
carries $x'\to x$ for $x\in\set{p_1, a, c}$.

\begin{figure}[h!]
\includegraphics{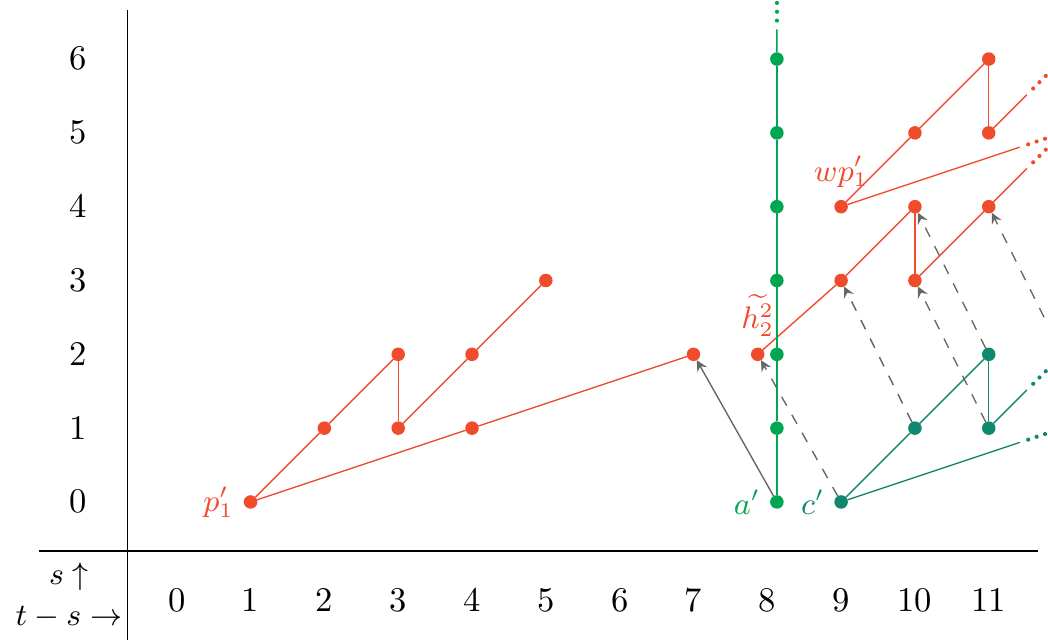}

\caption{The $E_2$-page of the Adams spectral sequence computing $\pi_*(\cQ')_2^\wedge$, where $\cQ'$ is the
spectrum defined in the proof of \cref{RP2_xtn}. By comparing with the Adams spectral sequence for $\cQ$, we learn
$d_2(a') = h_2^2 p_1'$ from \cref{g_to_o_nonzero}, and that the dashed differentials (e.g.\ $d_2(c')$,
$d_2(h_1c')$) vanish if and only if the differentials in~\ref{h1_orange_d2_family} vanish.}
\label{RP2_fig}
\end{figure}

The point of all of this is that if the differentials in~\ref{h1_orange_d2_family} vanish, then both $c$ and $h_1^2
p_7$ live to the $E_\infty$-page for $\cQ$, then both $c$ and $h_1^2p_7$ are in the image of the map $\Phi$ on
$E_\infty$-pages induced by $\cQ'\to\cQ$: $c = \Phi(c')$, and Bruner-Rognes~\cite[Corollary 4.3]{BR21} define a class
$\widetilde{h_2^2}\in\Ext_{\cA(2)}^{2,9}(C2) = \Ext_{\cA(2)}^{2,10}(\textcolor{RedOrange}{\Sigma C2})$ such that
$h_1^2 p_7 = \Phi(\widetilde{h_2^2})$. And looking at \cref{RP2_fig}, in the $E_\infty$-page for $\cQ'$,
$h_1(\widetilde{h_2^2})\ne 0$ and $h_1(wp_1')\ne 0$, so the $2\eta = 0$ trick from \cref{easier_extensions} splits
the extensions in $\pi_9(\cQ')_2^\wedge$. Thus there is a class $\underline c'\in\pi_9(\cQ')_2^\wedge$ such that
$2\underline c' = 0$ and the image of $\underline c'$ in the $E_\infty$-page is $c'$. Applying $\Phi(c') = c$, we
learn $c$ lifts to $\Phi(\underline c')$ in $\pi_9(\cQ)_2^\wedge$, and twice this class is $0$, as we wanted to
prove.
\end{proof}
We have therefore proven the following theorem.
\begin{thm}
\label{het_at_2}
Ignoring odd-primary torsion, there are isomorphisms
\begin{alignat*}{2}
	\Omega_0^{\xihet}  &\cong \Z & \qquad\qquad \Omega_6^\xihet &\cong \Z/2\\
	\Omega_1^{\xihet}  &\cong \Z/2\oplus\Z/2 &\Omega_7^\xihet &\cong \Z/16\\
	\Omega_2^{\xihet}  &\cong \Z/2\oplus\Z/2 &\Omega_8^\xihet &\cong \Z^3\oplus (\Z/2)^{\oplus i}\\
	\Omega_3^{\xihet}  &\cong \Z/8&\Omega_9^\xihet &\cong (\Z/2)^{\oplus j}\\
	\Omega_4^{\xihet}  &\cong \Z\oplus\Z/2& \Omega_{10}^\xihet &\cong (\Z/2)^{\oplus k}\\
	\Omega_5^{\xihet}  &\cong 0& \Omega_{11}^\xihet &\cong A,
%	\Omega_6^{\xihet}  &\cong \Z/2\\
%	\Omega_7^{\xihet}  &\cong \Z/16\\
%	\Omega_8^{\xihet}  &\cong \Z^3\oplus\Z/2\\
%	\Omega_9^{\xihet}  &\cong (\Z/2)^{\oplus 4}\\
%	\Omega_{10}^{\xihet}  &\cong (\Z/2)^{\oplus 4}\\
%	\Omega_{11}^{\xihet}  &\cong A,
\end{alignat*}
where:
\begin{itemize}
	\item $A$ is an abelian group of order $64$ isomorphic to one of $\Z/8\oplus\Z/8$, $\Z/16\oplus\Z/4$,
	$\Z/32\oplus\Z/2$, or $\Z/64$, and
	\item either $i = 1$, $j = 4$, and $j = 4$, or $i = 2$, $j = 6$, and $k = 5$.
\end{itemize}
\end{thm}
\subsubsection{Some manifold generators}
\label{xihet_gens}
We finish this section by giving manifold representatives for all the generators for the groups we found in
dimensions $10$ and below, except possibly for two classes in degrees $9$ and $10$ if the differentials
in~\ref{h1_orange_d2_family} vanish. We also give partial information in dimension $11$. In this list, we
implicitly localize at $2$, though we will soon see in \cref{no_odd_het} that this does not lose any information.

The map $\MTSpin\vee (\MTString\wedge B\Z/2)\to \MTxihet$ is surjective on homotopy groups in degrees $7$ and
below, quickly giving us many of the generators we need. The low-dimensional generators of spin bordism are
standard; for $\widetilde\Omega_*^\String(B\Z/2)$, we use the $h_2$-action on the $E_\infty$-page together with the
map $\widetilde\Omega_*^\String(B\Z/2)\to \widetilde\Omega_*^\Spin(B\Z/2)$, as in the proof of \cref{orange_diffs}
(see \cref{orange_diffs_pic}), to deduce generators.
\begin{enumerate}
	\setcounter{enumi}{-1}
	\item $\Omega_0^{\xihet}\cong \Z$, generated by the point.
	\item $\Omega_1^{\xihet}\cong \Z/2\oplus\textcolor{RedOrange}{\Z/2}$. The first summand comes from $\Omega_1^\Spin$,
	hence is generated by $\Snb$, the
	circle with $\xihet$-structure induced from its nonbounding framing. The other summand, corresponding to
	$p_1\in E_\infty^{0,1}$ of the Adams spectral sequence for $\cQ$, is in Adams
	filtration zero, hence corresponds to a mod $2$ cohomology class and is detected by that class. Looking at
	\cref{heterotic_A2_module}, this class is the generator of $H^1(B\Z/2;\Z/2)$ evaluated on the principal
	$\Z/2$-bundle associated to a $\xihet$-structure. Thus we can take as our generator $S^1$ with
	$\xihet$-structure induced by the nontrivial $\Z/2$-bundle and the inclusion $\Z/2\inj \rE_8^2\rtimes\Z/2$. We
	will call this generator $\RP^1$, so that we can represent its $\Z/2$-bundle by $S^1\to\RP^1$.
	\item An action by $h_1$ in the $E_\infty$-page of an Adams spectral sequence calculating bordism lifts to
	taking the product with $\Snb$ on manifold generators. Acting by $h_1$ defines an
	isomorphism from the $1$-line of the $E_\infty$-page to the $2$-line, so we can take $\Snb\times
	\Snb$ and $\RP^1\times \Snb$ to be our two generators of $\Omega_2^{\xihet}$.
	\item $\Omega_3^{\xihet}\cong\textcolor{RedOrange}{\Z/8}$; there is a generator whose image in the Adams
	$E_\infty$-page is $p_3$. The sequence of maps
	\begin{equation}
		\widetilde\Omega_3^\String(B\Z/2) \overset\iota\longrightarrow\Omega_3^\xihet \overset p\longrightarrow
		\Omega_3^\Spin(B\Z/2)
	\end{equation}
	consists of two isomorphisms $\Z/8\overset\cong\to\Z/8\overset\cong\to\Z/8$, so it suffices to find a generator
	of $\Omega_3^\Spin(B\Z/2)$ that admits a string structure. The standard generator is $\RP^3$ with principal
	$\Z/2$-bundle $S^3\to\RP^3$, and because $\RP^3$ is parallelizable, it admits a string structure.
	%induces a map of Adams spectral sequences for $\xihet$-bordism and the spin bordism of $B\Z/2$, and this is an
	%isomorphism in topological degree $3$ on the $E_\infty$-page, so we can choose as our generator any generator
	%of $\Omega_3^\Spin(B\Z/2)$ with a $\xihet$-structure, such as $\RP^3$ with bundle induced by $\Z/2\inj
	%\rE_8^2\rtimes\Z/2$.
	\item $\Omega_4^{\xihet}\cong \Z \oplus \textcolor{RedOrange}{\Z/2}$. The free summand comes from
	$\Omega_4^\Spin$, hence is generated by the K3 surface with trivial $\Z/2$-bundle, and $\rE_8$-bundles with
	characteristic classes $-\lambda(\mathrm{K3})$ and $0$.
	$\textcolor{RedOrange}{\Z/2}$ corresponds to $E_\infty^{1,5}\cong \textcolor{RedOrange}{\Z/2}\cdot h_2p_1$.
	Action by $h_2$ lifts to the product with $S^3$ with its Lie group framing, so we can generate this summand
	with $S^3\times\RP^1$.
	\begin{rem}
	\label{S4_generator}
	In \cref{4d_goings_on}, we showed $\Omega_4^\Spin(K(\Z, 4))\cong \Omega_4^{\xi^{r,\mathrm{het}}}\to
	\Omega_4^{\xihet}$ is surjective; using this, we can replace $S^3\times\RP^1$, which we will need later.
	Stong~\cite{Sto86} showed $\Omega_4^\Spin(K(\Z, 4))\cong\Z\oplus\Z$; one $\Z$ factor comes from
	$\Omega_4^\Spin$, hence is represented by the K3 surface with trivial map to $K(\Z, 4)$. The other is detected
	by the bordism invariant which, given a $4$-dimensional spin manifold $X$ and a map $f\colon X\to K(\Z, 4)$,
	sends $X\mapsto \int_X f^*c$, where $c\in H^4(K(\Z, 4);\Z)$ is the tautological class. For example, this
	invariant equals $1$ on $S^4$ with its standard orientation and unique spin structure inducing that
	orientation, with the map to $K(\Z, 4)$ given by the class $1\in H^4(S^4;\Z)\overset\cong\to\Z$.

	The images of the two classes $(\mathrm{K3}, 0)$ and $(S^4, 1)$ in $\Omega_4^{\xihet}$ must generate.
	Unsurprisingly, the K3 surface is sent to a generator of the $\Z$ summand we described above; this summand is
	detected by $\int p_1$. As this invariant vanishes on $(S^4, 1)$, surjectivity of the map on $\Omega_4$ implies
	that $(S^4, 1)$ maps to the class of $\RP^1\times S^3$.\footnote{We
	thank Justin Kaidi for informing of us of this fact.} Thus the
	$\textcolor{RedOrange}{\Z/2}$ summand in $\Omega_4^{\xihet}$ can be generated by $S^4$ with trivial
	$\Z/2$-bundle and two $\rE_8$-bundles with characteristic classes $c = \pm 1\in H^4(S^4;\Z)$.

	The map on Adams spectral sequences induced from the map of spectra
	$\mathit{MT\xi}^{r,\mathrm{het}}\to\MTxihet$ sends the class in the $E_\infty$-page representing $(S^4, 1)$ to
	$0$ (see Francis~\cite[\S 2]{Fra11} or Lee-Yonekura~\cite[\S 3.5]{LY22} for the Adams spectral sequence for
	$\Omega_*^{\xi^{r,\mathrm{het}}} = \Omega_*^\Spin(K(\Z, 4))$), so the fact that the image of $(S^4, 1)$ is
	nonzero in $\Omega_4^{\xihet}$ is analogous to a hidden extension.
	\end{rem}
%
%	The red summand looks like the K3 generating $\Omega_4^{\Spin}$, and indeed the map of Adams spectral sequences
%	induced by $B\Ghet\to B\Spin$ sends the generator of the $\textcolor{BrickRed}{\Z}\subset\Omega_4^{\xihet}$ to
%	the K3 surface. Thus our generator is a K3 surface with trivial $\rE_8$- and $\Z/2$-bundles, and with $B$-field
%	chosen to satisfy the Bianchi identity.
	\item $\Omega_5^{\xihet} = 0$.
	\item $\Omega_6^\xihet\cong\textcolor{RedOrange}{\Z/2}$, and the image of a generator on the $E_\infty$-page is
	$h_2p_3$, which lifts to imply that we can take $S^3\times\RP^3$ as a generator.
	\item $\Omega_7^{\xihet}\cong\textcolor{RedOrange}{\Z/16}$. This
	$\textcolor{RedOrange}{\Z/16}$ is detected by $\Omega_7^\Spin(B\Z/2)$ much like $\RP^3$ was, and we
	learn that this summand is generated by $\RP^7$ with $\Ghet$-bundle induced from the $\Z/2$-bundle
	$S^7\to\RP^7$, and is detected in the $E_\infty$-page by $p_7$.
%
%	The generator of the $\textcolor{RedOrange}{\Z/2}$ summand corresponds to a class in the $E_\infty$-page which
%	is $h_2^2$ times the generator of $E_\infty^{0,1}$, which lifts to imply that we can represent this generator
%	by $S^3\times S^3\times\RP^1$.
	\item $\Omega_8^\xihet\cong\Z^2\oplus\textcolor{Green}{\Z}\oplus \textcolor{RedOrange}{\Z/2}$ together with an
	additional $\textcolor{RedOrange}{\Z/2}$ summand if the differentials in~\ref{h1_orange_d2_family} do not
	vanish.
	\begin{itemize}
		\item The first two free summands come from $\Omega_*^\Spin$; their generators may be taken to be the
		quaternionic projective plane $\HP^2$ and a Bott manifold $B$. One can choose $B$ to have a string
		structure~\cite[\S 5.3]{FH21b} and we do so. In both cases, the $\Z/2$-bundle associated to the
		$\xihet$-structure is trivial; since $B$ is string, we give it the $\xihet$-structure in which both
		principal $\rE_8$-bundles are trivial. For $\HP^2$, $H^4(\HP^2; \Z)\cong\Z$ with generator $x$, as we
		discussed in the proof of \cref{purple_lem}; we choose a $\xihet$-structure on $\HP^2$ with principal
		$\rE_8$-bundles $P,Q\to\HP^2$ with $c(P) = -x$ and $Q$ trivial.
		\item The third free summand comes from the green $h_0$-tower in topological degree $8$ in the Adams
		spectral sequence for $\pi_*(\cQ)$. This summand is detected by the bordism invariant
		\begin{equation}
		\label{8d_green_detector}
			f\colon \int c(P)c(Q)\colon \Omega_8^\xihet\longrightarrow\Z,
		\end{equation}
		because this quantity can be nonzero (as we show below), it vanishes on the two generators we discovered
		for the other two free summands, and because it must vanish on the remaining, torsion summand. It is a
		consequence of \cref{g_to_o_nonzero} that the mod $2$ reduction of~\eqref{8d_green_detector}, which is
		$\int D_1D_2$, vanishes. This is because every class $x\in E_2^{0,t}$ has an associated degree-$t$ $\Z/2$
		cohomology class $c_x$, and $x$ lives to the $E_\infty$-page if and only if the bordism invariant $\int
		c_x$ is nonvanishing. Thus the minimum nonzero value of $\abs{f(M)}$, where $M$ is a closed, $8$-dimensional
		$\xihet$-manifold, is at least $2$.

		Recall from the proof of \cref{purple_lem} that $H^*(\HP^2;\Z)\cong\Z[x]/(x^3)$ with $\abs x = 4$ and
		$\lambda(\HP^2) = x$. Consider the two $\rE_8$-bundles $P,Q\to \HP^2$ prescribed by $c(P) = 2x$ and $c(Q) =
		-x$; then $\lambda(\HP^2) - c(P) - c(Q) = 0$, so this data lifts to a $\xihet$-structure, and
		\begin{equation}
			\int_{\HP^2} c(P) c(Q) = 2,
		\end{equation}
		achieving the minimum. Therefore $\HP^2$ with these two principal $\rE_8$-bundles generates the final free
		summand.
		\item The $\textcolor{RedOrange}{\Z/2}$ summand that we know is present independent of any unresolved
		differentials is generated by $h_1p_7$, so as usual lifts to $\Snb\times\RP^7$.
		\item If $d_2(c)\ne 0$, there is an additional $\textcolor{RedOrange}{\Z/2}$ summand represented in the
		$E_\infty$-page by $b$. We will discuss this summand, and its generator $X_8$, in \S\ref{s:X8}.
	\end{itemize}
	\item $\Omega_9^\xihet\cong (\Z/2)^{\oplus 2}\oplus \textcolor{RedOrange}{(\Z/2)^{\oplus 2}}$, and if the
	differentials in~\ref{h1_orange_d2_family} vanish, there is an additional $\textcolor{RedOrange}{\Z/2}\oplus
	\textcolor{PineGreen}{\Z/2}$ summand.
	\begin{itemize}
		\item Two of the $\Z/2$ summands come from $\Omega_9^\Spin\cong (\Z/2)^{\oplus 2}$, where they are
		represented by the generators $\HP^2\times \Snb$ and $B\times \Snb$, with
		$\xihet$-structure induced from the corresponding generators in $\Omega_8^\xihet$.
		\item The other two $\Z/2$ summands that are present no matter the value of the undetermined differentials
		are in the image of the map $\iota\colon\widetilde\Omega_9^\String(B\Z/2)\to\Omega_9^\xihet$. The generator
		of the summand in lower Adams filtration has image in the $E_\infty$-page equal to $h_1^2p_7$, so we obtain
		$\Snb\times\Snb\times\RP^7$.

		The summand in higher Adams filtration has nonzero image in
		$\widetilde\Omega_9^\Spin(B\Z/2)\cong\Z/2\oplus\Z/2$, by inspection of \cref{orange_diffs_pic}. The two
		generators of $\widetilde\Omega_9^\Spin(B\Z/2)$ can be taken to be $\HP^2\times\RP^1$ and $B\times\RP^1$;
		to determine which we get, compose further with the Atiyah-Bott-Shapiro~\cite{ABS} map
		$\widetilde\Omega_9^\Spin(B\Z/2)\to\widetilde\ko_9(B\Z/2)\cong\Z/2$, which sends $[\HP^2\times\RP^1]\mapsto
		0$ and $[B\times\RP^1]$ to the generator. The image of the map of Adams spectral sequence in
		\cref{orange_diffs_pic} is contained in the summand whose image under the Atiyah-Bott-Shapiro map is
		nonzero, the image of our generator in $\widetilde\Omega_9^\Spin(B\Z/2)$ is bordant to $B\times\RP^1$;
		finally, since $B$ and $\RP^1$ are both string, we can take $B\times\RP^1$ as our last generator in this
		dimension.
		\item If $d_2(h_1 c) = 0$, there is another $\textcolor{RedOrange}{\Z/2}$ summand whose image in the
		$E_\infty$-page is $h_1b$. Thus as usual it lifts to $\Snb\times X_8$, where $X_8$ is the manifold we
		describe in \S\ref{s:X8}.
		\item If $d_2(c) = 0$, there is another $\textcolor{PineGreen}{\Z/2}$ summand whose image in the
		$E_\infty$-page is $c$. We were unable to find a manifold $X_9$ representing this generator. Because $c$ is
		in Adams filtration $0$, corresponding to the mod $2$ cohomology class $D_1D_2x$, if $X_9$ exists then one
		can detect it by showing $\int_{X_9} D_1D_2x = 1$.
	\end{itemize}
%	\item We do not know the isomorphism type of $\Omega_8^{\xihet}$, but it contains three $\Z$ summands. Two of
%	them are detected by the map $\MTxihet\to\MTSpin$, so like we argued for the K3 surface above, the generators
%	of those two summands can be represented by $\mathbb{HP}^2$ and a Bott manifold with specific choices of
%	$B$-field to cancel $\lambda$. The remaining $\Z$ summand is detected by the quantity $\int c_1c_2\in\Z$; we
%	do not know a generator. There could also be torsion coming from $\widetilde\Omega_8^\String(B\Z/2)$.
%	\item We currently know very little about $\Omega_9^\xihet$; there could be several differentials or hidden
%	extensions. Nonetheless, the $\textcolor{Goldenrod!67!black}{\Z/2}$ in $E_2^{1,10}$ does not admit
%	differentials, so survives to the $E_\infty$-page, and lifts to be generated by $\mathrm{Bott}\times
%	\Snb$. The $\textcolor{PineGreen}{\Z/2}$ in $E_2^{0,9}$ also survives to the $E_\infty$-page,
%	where it is detected by $\int x c_1c_2$, where $c_i$ are the classes of the associated $\rE_8$-bundles and $x$
%	is the class of the associated $\Z/2$-bundles. It would be interesting to find an explicit manifold
%	representative of this class.
	\item $\Omega_{10}^\xihet\cong (\Z/2)^{\oplus 3}\oplus \textcolor{RedOrange}{\Z/2}$, together with potentially
	another $\textcolor{PineGreen}{\Z/2}$ summand if the differentials in~\ref{h1_orange_d2_family} vanish.
	\begin{itemize}
		\item Three of the $\Z/2$ summands in $\Omega_{10}^\xihet$ come from $\Omega_{10}^\Spin\cong (\Z/2)^{\oplus
		3}$. Their generators are known to be $B\times \Snb\times\Snb$, $\HP^2\times\Snb\times\Snb$, and a
		\term{Milnor hypersurface} $X_{10}$, defined to be a smooth degree-$(1, 1)$ hypersurface in
		$\CP^2\times\CP^4$. Milnor~\cite[\S 3]{Mil65} showed that $X_{10}$ generates the last $\Z/2$ summand in
		$\Omega_{10}^\Spin$.
%		. Any closed spin $10$-manifold $M$ with $\int_M w_4(M)w_6(M) = 1$ can be chosen to be the third generator;
%		$X_{10}$ refers to a , which is a common choice (e.g.~\cite[\S 3.3]{DMW02}).
		\item The next $\textcolor{RedOrange}{\Z/2}$ summand is detected by the maps
		$\widetilde\Omega_{10}^\String(B\Z/2)\to \Omega_{10}^\xihet$ and
		$\Omega_{10}^\xihet\to\Omega_{10}^\Spin(B\Z/2)$, and by a similar argument to the one we gave for the
		higher-filtration orange $\textcolor{RedOrange}{\Z/2}$ summand in degree $9$, we may choose
		$B\times\RP^1\times \Snb$ as the generator.
		\item If $d_2(h_1c) = 0$, then there is an additional $\textcolor{PineGreen}{\Z/2}$ summand whose image in
		the $E_\infty$-page is $h_1c$. Thus we can take $\Snb\times X_9$ for a manifold representative, though as
		discussed above we do not know what $X_9$ is.
	\end{itemize}
	\item We have not determined generators for $\Omega_{11}^\xihet$, nor even its isomorphism type. This is a
	question whose answer would be useful for anomaly cancellation for the $\rE_8\times\rE_8$ heterotic string; see
	\cref{ques3} and \S\ref{heterotic_anomaly}. Nonetheless, the Adams argument we gave above implies
	$\Omega_{11}^\xihet$ contains a $\textcolor{RedOrange}{\Z/8}$ subgroup, the image of
	$\iota\colon\widetilde\Omega_{11}^\String(B\Z/2)\to\Omega_{11}^\xihet$. By comparing with the map
	$\widetilde\Omega_{11}^\String(B\Z/2)\to\widetilde\Omega_{11}^\Spin(B\Z/2)$ as in \cref{orange_diffs_pic}, one
	learns that the class of $B\times\RP^3$ generates this $\textcolor{RedOrange}{\Z/8}$.
\end{enumerate}
\subsubsection{$X_8$, a potentially nonzero class in $\Omega_8^\xihet$}
\label{s:X8}
Though we were unable to determine if the class $b\in E_2^{2,10}$ survives to the $E_\infty$-page, we are able to
write down a manifold representative $X_8$ of the class it determines in $\Omega_8^\xihet$; if $b$ does survive,
$X_8$ should be added to the list of generators above.
\begin{defn}
Let $\Z/2$ act on $S^3\times S^3\times S^2$ by the antipodal map on $S^2$ and the first copy of $S^3$, and a
reflection through a plane on the second $S^3$. This is a free action; let $X_8$ denote the quotient, which is a
smooth manifold.
\end{defn}
$X_8$ is a generalized Dold manifold of the sort studied by Nath-Sankaran~\cite{NS19}. Manifolds similar to $X_8$
frequently appear as generators of bordism groups: see~\cite[\S 5.5.1]{FH21b} and~\cite[\S 14.3.3]{DDHM22a} for
related examples.
\begin{lem}
\label{X8_is_string}
$X_8$ admits a string structure, and one can choose a string structure on $X_8$ so that the induced string
structure on $S^3\times S^3$ is the one induced by the Lie group framing on $S^3\times S^3 \cong\SU_2\times\SU_2$.
\end{lem}
\begin{proof}
Adding the normal bundles for $S^{k-1}\inj\R^k$ defines an isomorphism
\begin{equation}
\label{stab_split}
	T(S^3\times S^3\times S^2)\oplus\underline\R^3 \overset\cong\longrightarrow \underline\R^4\oplus
	\underline\R^4\oplus\underline\R^3.
\end{equation}
To understand $TX_8$, we will study~\eqref{stab_split} when we introduce the $\Z/2$-action on $S^3\times S^3\times
S^2$ whose quotient is $X_8$. Since the outward unit normal vector field on $S^k$ is $\O_{k+1}$-invariant, $\Z/2$
acts trivially on the $\underline\R^3$ on the left side of~\eqref{stab_split}, since the outward unit normal vector
field provides the trivializations of the normal bundles giving that $\underline\R^3$ factor. On the right-hand
side, $\Z/2$ by the antipodal map on the first factor of $S^3$, so acts by $-1$ on each $\underline\R$ summand of
the first $\underline\R^4$. The reflection on the second $S^3$ factor means $\Z/2$ acts on the second
$\underline\R^4$ by $-1$, $1$, $1$, and $1$ on the four $\underline\R$ summands. Finally, the antipodal map on
$S^2$ implies $\Z/2$ acts by $-1$ on the remaining three $\underline\R$ summands.

Passing from equivariant vector bundles on $S^3\times S^3\times S^2$ to nonequivariant vector bundles on the
quotient, \eqref{stab_split} induces an isomorphism
\begin{equation}
	TX_8\oplus\underline\R^3 \overset\cong\longrightarrow \sigma^{\oplus 8}\oplus \underline\R^3,
\end{equation}
where $\sigma\to X_9$ is pulled back from the tautological line bundle $\sigma\to\RP^2$. The Whitney sum formula
implies $\sigma^{\oplus 8}\to\RP^2$ is spin, and since the string obstruction lives in $H^4(\RP^2;\Z) = 0$,
$\sigma^{\oplus 8}$ is string. Thus the pullback to $X_8$ is also string, so $TX_8$ is string.

For the Lie group framing string structure, use the fact that the involutions on each $S^3$ summand can be
described in terms of Lie groups: since the quotient of $S^3 \cong \SU_2$ by the antipodal map is $\RP^3\cong
\SO_3$, the Lie group framing on $S^3$ is equivariant for the antipodal map. Compatibility for the reflection comes
from the action of a reflection in $\Pin_3^+\supset\SU_2$.
\end{proof}
\begin{prop}
\label{X8_transfer}
With the string structure described in \cref{X8_is_string} and the $\Z/2$-bundle $\sigma\to X_8$, $[X_8]$ is
linearly independent from $[\Snb\times\RP^7]$ in $\widetilde\Omega_8^\String(B\Z/2)\cong\Z/2\oplus\Z/2$, so the
image of $[X_8]$ in the $E_\infty$-page for $\Omega_*^\String(B\Z/2)$ is the nonzero class in
$E_\infty^{2,10}\cong\Z/2$ (perhaps plus a term in lower filtration).
\end{prop}
\begin{proof}
Let $f\colon\widetilde\Omega_8^\String(\RP^2)\to\widetilde\Omega_8^\String(B\Z/2)$ be the map induced by
$\RP^2\inj\RP^\infty\simeq B\Z/2$. The map this induces on Adams spectral sequences is not hard to analyze:
Bruner-Rognes~\cite[\S 4.4, Chapter 6, \S 12.1]{BR21} run the whole Adams spectral sequence for
$\widetilde\tmf_*(\RP^2)$, using the identification $\Sigma^\infty\RP^2\simeq \Sigma\Sph/2$, and as discussed above
$\tmf$- and $\MTString$-homology agree in degrees $14$ and below.\footnote{See also the closely related work of
Beaudry-Bobkova-Pham-Xu~\cite{BBPX22}, who compute $\tmf_*(\RP^2)$ using the elliptic spectral sequence.} Likewise,
Davis-Mahowald~\cite[Table 3.2]{DM78} compute the $E_2$-page of the Adams spectral sequence for
$\widetilde\Omega_*^\String(B\Z/2)$ in the range we need, and with their calculation, $h_i$-linearity of
differentials, and the $2\eta = 0$ trick from the proof of \cref{easier_extensions} one sees that
$\widetilde\Omega_8^\String(B\Z/2)\cong\Z/2\oplus\Z/2$. As discussed in \S\ref{xihet_gens}, one of the $\Z/2$
summands is detected by $\RP^7\times\Snb$, whose image in the $E_\infty$-page is in filtration $1$. Consider the
map
\begin{equation}
\label{RP2Inf}
	\Psi\colon \Ext_{\cA(2)}(\widetilde H^*(\RP^2;\Z/2)) \longrightarrow \Ext_{\cA(2)}(\widetilde H^*(B\Z/2;\Z/2))
\end{equation}
induced by $\RP^2\to \RP^\infty\simeq B\Z/2$; we draw this map in \cref{2_to_infty}. $\Psi$ is also the map between
the $E_2$-pages of these two Adams spectral sequences; looking at \cref{2_to_infty}, $\Psi$ is injective in
topological degree $8$, with image containing the nonzero element of $E_2^{2,10}$ but not the nonzero class in
$E_2^{1,9}$. As both of these elements survive to the $E_\infty$-page, this lifts to imply that $f\colon
\widetilde\Omega_8^\String(\RP^2)\to \widetilde\Omega_8^\String(B\Z/2)$ is injective and that if one wants to find
a class in $\widetilde\Omega_8^\String(B\Z/2)$ linearly independent from $\RP^7\times\Snb$, it suffices to find a
nonzero class in $\widetilde\Omega_8^\String(\RP^2)$.
\begin{figure}[h!]
\begin{subfigure}[c]{0.48\textwidth}
%\begin{sseqdata}[name=SpinRP, classes = fill, scale=0.5, xrange={1}{10}, yrange={0}{6}, Adams grading,
%>=stealth,
%x label = {$\displaystyle{s\uparrow \atop t-s\rightarrow}$},
%x label style = {font = \small, xshift = -22.5ex, yshift=6ex}]
%\begin{scope}[RedOrange]
%	\class(1, 0)
%	\class(2, 1)\structline
%	\class(3, 2)\structline
%	\class(3, 1)\structline
%	\class[fill=none](4, 2)\structline
%	\class[fill=none](5, 3)\structline
%	\class(4, 1)\structline(1, 0)(4, 1)
%	\class(7, 2)\structline
%
%	\class(8, 2)
%	\class(9, 3)\structline
%	\class[fill=none](10, 4)\structline
%	\class[fill=none](10, 3)\structline
%	\class[fill=none](11, 4)\structline
%
%	\class(9, 4)
%	\class(10, 5)\structline
%	\class(11, 6)\structline
%	\class(11, 5)\structline
%	\class[fill=none](12, 6)\structline
%\end{scope}
%\end{sseqdata}
%\printpage[name=SpinRP, page=2]
\includegraphics{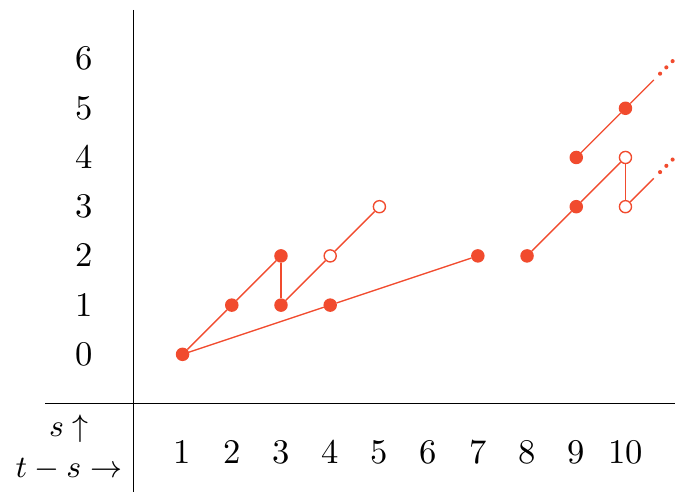}
\end{subfigure}
\begin{subfigure}[c]{0.48\textwidth}
%\begin{sseqdata}[name=StringZmod, classes = fill, scale=0.5, xrange={1}{10}, yrange={0}{6}, Adams grading,
%>=stealth,
%x label = {$\displaystyle{s\uparrow \atop t-s\rightarrow}$},
%x label style = {font = \small, xshift = -22.5ex, yshift=6ex}]
%\begin{scope}[RedOrange]
%	\class(1, 0)
%	\class(2, 1)\structline
%	\class(3, 2)\structline
%	\class(3, 1)\structline
%	\class[fill=gray!40!white](3, 0)\structline
%
%	\class(4, 1)\structline(1, 0)(4, 1)
%	\class[fill=gray!40!white](6, 1)\structline(3, 0)(6, 1)
%	\class[fill=gray!40!white](7, 0)
%	\class[fill=gray!40!white](7, 1)\structline
%	\class(7, 2)
%	\class[fill=gray!40!white](7, 2)
%	\class[fill=gray!40!white](7, 3)
%	\structline(7, 1)(7, 2, -1)
%	\structline(7, 1)(7, 2, 1)
%	\structline(7, 2, -1)(7, 3)
%	\structline(4, 1)(7, 2, 1)
%	\class[fill=gray!40!white](8, 1)\structline(7, 0)(8, 1)
%	\class[fill=gray!40!white](9, 2)\structline
%	\class(8, 2)
%	\class(9, 3)\structline
%	\structline(6, 1)(9, 2)
%	\class(9, 4)
%	\class(10, 5)\structline
%	\class(11, 6)\structline
%	\class(11, 5)\structline
%	\class[fill=gray!40!white](11, 4)\structline
%	\class[fill=gray!40!white](14, 5)\structline
%\end{scope}
%\end{sseqdata}
%\printpage[name=StringZmod, page=2]
\includegraphics{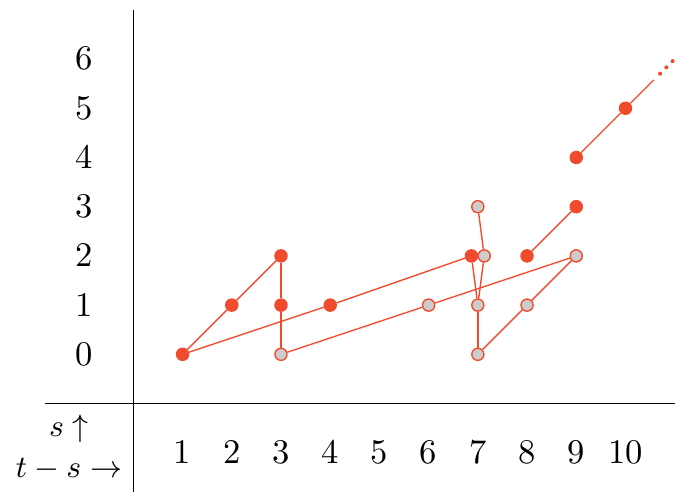}
\end{subfigure}
\caption{The map of Adams spectral sequences for reduced string bordism induced by the map
$\RP^2\inj\RP^\infty\simeq B\Z/2$, which we use in the proof of \cref{X8_transfer}. Left: $\Ext_{\cA(2)}(\widetilde
H^*(\RP^2; \Z/2), \Z/2)$, the $E_2$-page of the Adams spectral sequence computing
$\widetilde\Omega_*^\String(\RP^2)_2^\wedge$. Filled dots have nonzero image after mapping to $\RP^\infty$;
unfilled dots are the kernel. Bruner-Rognes~\cite[\S 4.4, Chapter 6, \S 12.1]{BR21} compute these Ext groups and
run this Adams spectral sequence; from their work we learn there are no differentials in this range (though there
are hidden $\nu$-extensions that do not enter into our argument; see (\textit{ibid.}, Theorem 12.5)). Right:
$\Ext_{\cA(2)}(\widetilde H^*(B\Z/2; \Z/2), \Z/2)$, a summand of the $E_2$-page of the Adams spectral sequence
computing $\widetilde\Omega_*^\String(B\Z/2)_2^\wedge$. Filled dots are in the image of the map from $\RP^2$; gray
dots are the cokernel. The $E_2$-page was computed by Davis-Mahowald~\cite[Table 3.2]{DM78}, and during the proof
of \cref{X8_transfer} we argue that there are no differentials or hidden extensions by $2$ in the range depicted.}
\label{2_to_infty}
\end{figure}

The map $\sigma\colon X_8\to B\Z/2$ factors through $\RP^2$ by definition, so we are done if we can show $X_8$,
with its map to $\RP^2$, is nonbounding. To do so, consider the transfer map $\Sigma^\infty\RP^2\to\Sigma^\infty
S^2$ associated to the double cover $S^2\to\RP^2$; this induces on string bordism a map
$\widetilde\Omega_*^\String(\RP^2)\to\widetilde\Omega_*^\String(S^2)$ sending $(M, f\colon M\to\RP^2)$ to the
double cover $M'\to M$ associated to the line bundle $f^*\sigma$, together to the map $M'\to S(\sigma) = S^2$.

The map $\Omega_k^\xi\to\widetilde\Omega_{k+\ell}^\xi(S^\ell)$ sending $M\mapsto (M\times S^\ell,
\mathrm{proj}_2\colon M\times S^\ell\to S^\ell)$ (where $S^\ell$ carries the bounding stable framing, which with
the $\xi$-structure on $M$ induces a $\xi$-structure on $M\times S^\ell$) is always an
isomorphism (e.g.\ check this with the Atiyah-Hirzebruch spectral sequence), and
$\Omega_6^\String\cong\Z/2\times\Z/2$~\cite[\S 3, \S 4]{Gia71}, generated by $S^3\times S^3$ with its Lie group
framing, because it is represented by $h_2^2$ in the Adams spectral sequence. Therefore
$\widetilde\Omega_8^\String(S^2)\cong\Z/2$ is generated by $S^3\times S^3\times S^2$, with the map to $S^2$
given by projection onto the third factor. The image of $X_8$ under the transfer is its double cover, which is
$S^3\times S^3\times S^2$, with the correct string structure and map to $S^2$, so $[X_8]\ne 0$ in
$\widetilde\Omega_8^\String(\RP^2)$, which suffices to prove the theorem.
\end{proof}
Finally, by looking at the map $\widetilde\Omega_*^\String(B\Z/2)\to\Omega_*^\xihet$, we conclude:
\begin{cor}
\label{X8_generate}
Suppose $d_2(c) = 0$ in the Adams spectral sequence for $\xihet$. Then $[X_8]\ne 0$ in $\Omega_8^{\xihet}$, and its
image in the $E_\infty$-page is the class $b\in E_\infty^{2,10}$ (perhaps plus some elements in lower filtration).
\end{cor}

%% file: odd_primary.tex
\label{xihet_odd}
\begin{thm}
\label{no_odd_het}
$\Omega_\ast^\xihet$ has no odd-primary torsion in degrees $11$ and below.
\end{thm}
\begin{proof}
This amounts to a direct computation with the Adams spectral sequence. We will go over the case $p = 3$ in detail;
for $p = 5,7$ the story is similar but easier, and for $p\ge 11$ it is trivial because the degrees of the Steenrod
powers are too high for the Adams spectral sequence to produce torsion.

First we compute $H^\ast(B\Ghet;\Z/3)$ as a module over the Steenrod algebra $\cA$ in low degrees in
\cref{ghet_coh}, then we do the same for $H^\ast(\MTxihet;\Z/3)$ in \cref{M3_module}. Once we have this, we can run
the Adams spectral sequence, and do so in \cref{Z3_adams}.

Throughout this subsection, $\cP^i$ refers to the $i^{\mathrm{th}}$ Steenrod power, a degree-$4i$ operation on mod
$3$ cohomology, and $\beta$ is the Bockstein homomorphism for the sequence $0\to\Z/3\to\Z/9\to\Z/3\to 0$.
\begin{lem}
Let $C\in H^3(K(\Z, 3); \Z/3)$ denote the mod $3$ reduction of the tautological class. Then
\begin{equation}
	H^*(K(\Z, 3);\Z/3)\cong\Z/3[C, \cP^1C, \beta\cP^1 C, \dotsc]/(C^2, \dotsc),
\end{equation}
where all missing generators and relations are in degrees $14$ and above.
\end{lem}
\begin{proof}
This is a standard application of the Serre spectral sequence for the fibration $K(\Z, 2)\to *\to K(\Z, 3)$,
so we will be succinct. $E_2^{0, *}\cong H^*(K(\Z, 2); \Z/3)\cong\Z/3[x]$,
with $\abs x = 2$; by the $E_\infty$-page, all powers of $x$ must be killed by differentials.

The only way to kill $x$ is with a transgressing $d_3\colon E_3^{0, 2}\to E_3^{3, 0}$. Let $C\coloneqq d_3(x)$.
$C^2 = 0$ follows by graded commutativity. The Leibniz rule for differentials means that when $3\nmid k$, $d_3(x^k)
= \pm x^{k-1}C$, and if $3\mid k$, $d_3(x^k) = 0$.

So $x^3$ survives to the $E_4$-page. The only remaining differential that can kill $x^3$ is the transgressing $d_7\colon
E_7^{0, 6}\to E_7^{7, 0}$, so $d_7(x^3)\ne 0$; by the Kudo transgression theorem~\cite{Kud56}, because $x^3 =
\cP^1(x)$, $d_7(x^3) = \cP^1 C$. The Leibniz rule then implies $d_7(x^6) = x^3\cP^1 C$, so by the $E_8$-page,
everything on the line $p = 0$ in total degree less than $18$ has been killed.

Because $d_3(x^3) = 0$, $x^2C$ survives to the $E_4$-page; the only remaining way for it to support a differential
is to have a new class $w\in H^8(K(\Z, 3);\Z/3)$ such that $d_5(x^2C) = w$. To see that $\beta(\cP^1 C) = \pm w$,
compare with the analogous spectral sequence for $\Z/9$-valued cohomology to see that $\cP^1C$ is not in the image
of the mod $3$ reduction map from $\Z/9$ cohomology to $\Z/3$ cohomology.\footnote{Alternatively, one could deduce
this Bockstein by setting up the Serre spectral sequence for $K(\Z, 3)\to *\to K(\Z, 4)$ and Hill's
calculation~\cite[Corollary 2.9, Figure 1(a)]{Hil09} of the low-degree mod $3$ cohomology of $K(\Z, 4)$ as an
$\cA$-module: $C$ transgresses to a degree-$4$ class $D$, and Hill shows $\beta(\cP^1 D)\ne 0$, so by the
Kudo  transgression theorem~\cite{Kud56}, $\beta(\cP^1 C)\ne 0$ in $H^*(K(\Z, 3);\Z/3)$.}
\end{proof}
\begin{prop}
\label{ghet_coh}
Let $D\in H^4(B\rE_8;\Z/3)$ be the mod $3$ reduction of the class $c$ from \cref{char_class_e8}, and let $D_1$ and
$D_2$ be the two copies of $D$ in $H^*(B\rE_8^2;\Z/3)$ coming from the two factors of $B\rE_8$.
In degrees $13$ and below, the pullback map on $\Z/3$ cohomology induced by $\phi\colon B\Ghet\to B\Spin\times
B(\rE_8^2\rtimes\Z/2)$ is the quotient ring homomorphism sending $\lambda - D_1 - D_2\mapsto 0$, $-p_2 - \cP^1(D_1
+ D_2)\mapsto 0$, and $\beta\cP^1(D_1 + D_2)\mapsto 0$.
%Hence there is an isomorphism of $\cA$-modules
%\begin{equation}
%	\widetilde H^*(B\Ghet;\Z/3)\cong \Sigma^4 M \oplus \Sigma^8\Z/3\oplus \Sigma^{12} C\alpha
%		\oplus \Sigma^{12} C\alpha
%		\oplus \Sigma^{12} C\alpha \oplus P,
%\end{equation}
%where $P$ is concentrated in degrees $14$ and above.
\end{prop}
Here $\phi$ is the map we constructed in~\eqref{forget_string} which forgets the B-field.
\begin{proof}
Throw the Serre spectral sequence at the fibration
\begin{equation}
\label{Ghet_serre}
	K(\Z, 3)\longrightarrow B\Ghet \longrightarrow B\Spin\times B(\rE_8^2\rtimes\Z/2).
\end{equation}
The base space is not simply connected, so we might have to worry about local coefficients, but this turns out not
to be the case, because the $\Z/2$ symmetry swapping the two $\rE_8$ factors, which is the origin of the $\pi_1$ in
the base, acts trivially on the B-field, which gives us the fiber in~\eqref{Ghet_serre}.

In order to run the Serre spectral sequence for~\eqref{Ghet_serre}, we need to know the cohomology of $B\Spin$ and
$B(\rE_8^2\rtimes\Z/2)$. The former is the polynomial ring on the mod $3$ reductions of the Pontrjagin classes,
which is a theorem of Borel-Hirzebruch~\cite[\S 30.2]{BH59}; for the latter, run the Serre spectral sequence for
the fibration
\begin{equation}
	B\rE_8^2\longrightarrow B(\rE_8^2\rtimes\Z/2)\longrightarrow B\Z/2.
\end{equation}
Because $H^*(B\Z/2;\Z/3)$ vanishes in positive degrees, this Serre spectral sequence collapses to imply
\begin{equation}
	H^*(B(\rE_8^2\rtimes\Z/2);\Z/3)\overset\cong\longrightarrow H^*(B\rE_8^2;\Z/3)^{\Z/2}.
\end{equation}
The answer now follows from the Künneth formula, the fact that we can replace $B\rE_8$ with $K(\Z, 4)$ in the range
we need by the result of Bott-Samelson~\cite[Theorems IV, V(e)]{BS58} we mentioned in \S\ref{xihet_at_2}, and the
mod $3$ cohomology of $K(\Z, 4)$ in low degrees, worked out by Cartan~\cite{Car54} and Serre~\cite{Ser52}, and
stated explicitly by Hill~\cite[Corollary 2.9]{Hil09}.

Now back to~\eqref{Ghet_serre} and its Serre spectral sequence.
The fibration~\eqref{Ghet_serre} is classified by the degree-$4$ cohomology class $\lambda - D_1 - D_2$, i.e.\ it
is the pullback of the universal $K(\Z, 3)$-bundle
\begin{equation}
\label{universal_Z3}
	K(\Z, 3)\longrightarrow E K(\Z, 3)\longrightarrow B K(\Z, 3) \simeq K(\Z, 4)
\end{equation}
by the map $B\Spin\times B(\rE_8^2\rtimes\Z/2)\to K(\Z, 4)$ classified by $\lambda - D_1 - D_2$.\footnote{This map,
and hence also the fibration, is only determined up to homotopy, but any two choices of representative give
isomorphic answers.} In the Serre spectral sequence for~\eqref{universal_Z3}, the class $C\in E_2^{0, 3} =
H^3(K(\Z, 3); \Z/3)$ must transgress to the generator of $E_2^{4, 0} = H^4(K(\Z, 4);\Z/3)$, and this generator
pulls back to $\lambda - D_1 - D_2$, enforcing the relation $\lambda - D_1 - D_2 = 0$ in the $E_5$-page.

The other two pullbacks to zero in the theorem statement then follow from the Kudo transgression
theorem~\cite{Kud56}: $\cP^1 C\in E_2^{0, 7} = H^7(K(\Z, 3);\Z/3)$ must transgress to $\cP^1(\lambda - D_1 - D_2)$,
and analogously for $\beta\cP^1 C$. To compute these, we must determine how $\cP^1$ acts on the mod $3$ reductions
of Pontrjagin classes. Shay~\cite{Sha77} proves a formula for Steenrod powers of Chern classes, which yields the
formula for Pontrjagin classes by pullback. Hence, as worked out by Nordström~\cite{MO_Wu_3}, $\cP^1 p_1 = p_2$;
then an Adem relation tells us
\begin{equation}
	\cP^1 p_2 = \cP^1\cP^1 p_1 = -\cP^2 p_1 = p_1^3,
\end{equation}
the last equality because $\cP^i$ is the cup product cube on classes of degree $2i$. Thus we see that $\cP^1 C$
transgresses to $-p_2 - \cP^1(D_1 + D_2)$ and $\beta\cP^1 C$ transgresses to $\beta\cP^1(D_1 + D_2)$, killing those
classes by the $E_{10}$-page.

Now, the Leibniz rule cleans up the rest of the Serre spectral sequence in total degree at most $13$: by the
$E_{10}$-page, everything in this range is concentrated on the line $q = 0$. Therefore on the $E_\infty$-page,
the extension question is trivial in this range, and we conclude.
\end{proof}
\begin{prop}
\label{M3_module}
Let $\mathcal M_3$ denote the quotient of $H^*(\MTxihet;\Z/3)$ by all elements of degree $14$ or higher, $\mathcal
M^\SO_3$ denote the quotient of $H^*(\MTSO;\Z/3)$ by all elements of degree $14$ or higher, and $C\alpha$ denote the $\cA$-module which consists of two $\Z/3$ summands in degrees $0$ and $4$ linked by
$\cP^1$. %; if $a\colon \Sph_{(3)}\to\Sph_{(3)}$ is a generator of $\pi_3(\Sph_{(3)})\cong\Z/3$, then $C\alpha$ is
%the cohomology of the cofiber of $a$. 
Then, there is an isomorphism of $\cA$-modules
\begin{equation}
	\mathcal M_3\cong \mathcal M_3^\SO\oplus \Sigma^8 C\alpha \oplus \Sigma^{12}\Z/3.
\end{equation}
\end{prop}
\begin{proof}
In \cref{ghet_coh}, we discovered that the map $\phi\colon B\Ghet\to B\Spin\times B(\rE_8^2\rtimes\Z/2))$ induces a
surjection on mod $3$ cohomology in degrees $13$ and below. As $\phi$ commutes with the maps down to $B\O$ that are
part of the definition of these tangential structures, $\phi$ induces a map on Thom spectra
\begin{equation}
\label{forget_B_3}
	\MTxihet\to \MTSpin\wedge B(\rE_8^2\rtimes\Z/2)_+.
\end{equation}
Both of these tangential structures' maps to $B\O$ factor through $B\SO$, so the Thom isomorphism for mod $3$ cohomology
untwists. The Thom isomorphism is natural for maps of tangential structures, so we conclude that the pullback map on mod
$3$ cohomology induced by~\eqref{forget_B_3} is a surjection in degrees $13$ and below --- and therefore that we
can compute Steenrod powers in the cohomology of the latter Thom spectrum. And the map $\MTSpin\to\MTSO$ is an
equivalence away from $2$, so we may work with $\MTSO$ in place of $\MTSpin$. Milnor~\cite[Theorem 4]{Mil60}
computed the Steenrod module structure on $H^\ast(\MTSO;\Z/3)$, showing that it is a free $\cA/\beta$-module. Using
this, we can determine the Steenrod powers of $Up_i$, where $U$ is the Thom class; and this and the Cartan formula
finish the proof.
\end{proof}
\begin{prop}
\label{Z3_adams}
In topological degrees $12$ and below, the Adams $E_2$-page computing $(\Omega_*^\xihet)_3^\wedge$ consists of
$h_0$-towers concentrated in even topological degrees, and therefore this Adams spectral sequence collapses in
degrees $12$ and below.
\end{prop}
\begin{proof}
The direct-sum decomposition in \cref{M3_module} means that it suffices to prove the statement about $h_0$-towers
for $\mathcal M_3^\SO$, $\Sigma^8 C\alpha$, and $\Sigma^{12}\Z/3$ separately. As usual, with $M$ an $\cA$-module,
we write $\Ext(M)$ to denote $\Ext_\cA^{*, *}(M, \Z/3)$. The first ingredient we need is $\Ext(\Z/3)$ itself; the
computation of $\Ext_\cA(\Z/3)$ in degrees $t-s\le 11$ is due to Gershenson~\cite{Ger63}; May~\cite{May65, May66}
expanded this computation to $t-s\le 88$. In topological degrees $2$ and below, $\Ext(\Z/3)$ consists of a single
$h_0$-tower in topological degree $0$, implying the conclusion for $\Sigma^{12}\Z/3$.

Next, we compute $\Ext(C\alpha)$ using the fact that a short exact sequence of $\cA$-modules induces a long exact
sequence in Ext groups. Specifically, factor $C\alpha$ as an extension of $\cA$-modules
\begin{equation}
\label{cof_alf_ext}
	\shortexact{\textcolor{RubineRed}{\Sigma^4\Z/3}}{C\alpha}{\textcolor{Periwinkle}{\Z/3}},
\end{equation}
which we draw in \cref{cof_alf_fig}, left, and compute the corresponding long exact sequence in Ext in
\cref{cof_alf_fig}, right. There is one potentially nonzero boundary map in range: $\partial\colon\Ext_\cA^{0,
4}(\textcolor{RubineRed}{\Z/3})\to\Ext_\cA^{1,4}(\textcolor{Periwinkle}{\Z/3})$. This map must be nonzero because
$\Ext_\cA^{0, 4}(C\alpha) = \Hom_\cA(C\alpha, \Sigma^4\Z/3) = 0$. We see that in degrees $6$ and below,
$\Ext(C\alpha)$ consists solely of $h_0$-towers in even degrees, which implies the part of the corollary statement
coming from $\Sigma^8C\alpha$.
%The boundary maps in the long exact sequence
%c%ommute with the $\Ext_\cA(\Z/3)$-action on the domain and codomain, so one can deduce that the other boundary map
%is an isomorphism by acting on the first boundary map by a generator of $\Ext_\cA^{2,9}(\Z/3)\cong\Z/3$.

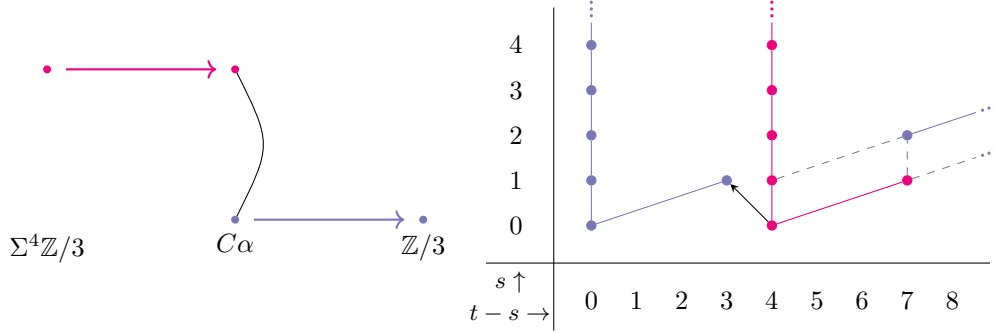
\begin{figure}[h!]
\begin{subfigure}[c]{0.4\textwidth}
\begin{tikzpicture}[scale=0.5]
	\PoneR(0, 0);
	\begin{scope}[RubineRed]
		\foreach \x in {-5, 0} {
			\tikzpt{\x}{4}{}{};
		}
		\draw[->, thick] (-4.5, 4) -- (-0.5, 4);
	\end{scope}
	\begin{scope}[Periwinkle]
		\foreach \x in {0, 5} {
			\tikzpt{\x}{0}{}{};
		}
		\draw[->, thick] (0.5, 0) -- (4.5, 0);
	\end{scope}
	\node[below=2pt] at (-5, 0) {$\Sigma^4\Z/3$};
	\node[below=2pt] at (0, 0) {$C\alpha$};
	\node[below=2pt] at (5, 0) {$\Z/3$};
\end{tikzpicture}
\end{subfigure}
\begin{subfigure}[c]{0.5\textwidth}
%\begin{sseqdata}[name=Calf, Adams grading, classes=fill, xrange={0}{8}, yrange={0}{4}, scale=0.6,
%	x label = {$\displaystyle{s\uparrow \atop t-s\rightarrow}$},
%	x label style = {font = \small, xshift = -23ex, yshift=6ex}, >=stealth]
%\begin{scope}[RubineRed]
%	\class(4, 0)\AdamsTower{}
%	\class(7, 1)\structline(4, 0)(7, 1)
%	\class(11, 2)
%	\class(14, 3)\structline
%\end{scope}
%\begin{scope}[draw=none, fill=none]
%\class(11, 1)
%\class(11, 3)
%\end{scope}
%\begin{scope}[Periwinkle]
%	\class(0, 0)\AdamsTower{}
%	\class(3, 1)\structline(0, 0)(3, 1)
%	\class(7, 2)
%	\class(10, 3)\structline
%	\class(10, 2)\structline
%	\class(13, 3)\structline
%	\class(11, 1)
%	\class(11, 2)\structline
%	\class(11, 3)\structline
%\end{scope}
%\d1(4, 0)
%\d1(11, 2, 1)
%\begin{scope}[gray, dashed]
%	\structline(4, 1)(7, 2)
%	\structline(7, 1)(7, 2)
%	\structline(7, 1)(10, 2)
%\end{scope}
%\end{sseqdata}
%\printpage[name=Calf, page=1]
\includegraphics{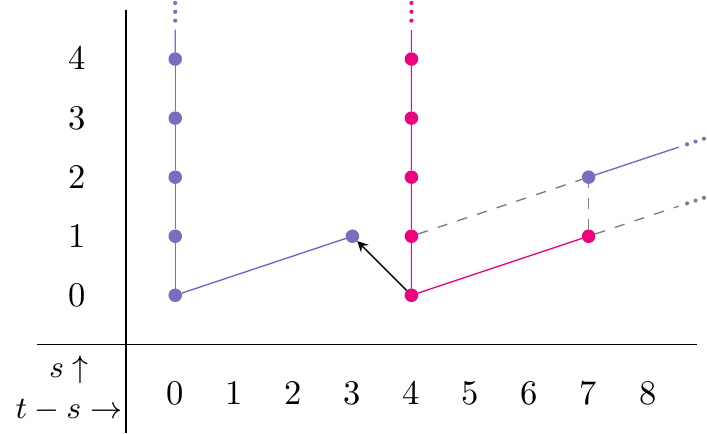}
\end{subfigure}
\caption{Left: the extension~\eqref{cof_alf_ext} of $\cA$-modules at $p = 3$. Right: the associated long exact
sequence in Ext. The dashed gray lines are actions by elements of $\Ext_\cA(\Z/3)$ that cannot be seen from this
long exact sequence and must be deduced another way; we do not need them in this paper, so do not go into the
details.}
\label{cof_alf_fig}
\end{figure}

Finally, $\mathcal M_3^\SO$. Milnor~\cite[Theorem 4]{Mil60} showed that this module coincides with a free
$\cA/\beta$-module in degrees $13$ and below, and proves (\textit{ibid.}, Lemma 5) that the Ext groups of such a
module consist solely of $h_0$-towers in even topological degree. Therefore in topological degrees $12$ and below,
$\Ext(\mathcal M_3^\SO)$ also consists solely of $h_0$-towers in even topological degrees.
\end{proof}
This suffices to prove \cref{no_odd_het} for $p = 3$: $h_0$-towers on the $E_\infty$-page lift to $\Z_3$ (i.e.\ the
$3$-adic integers) summands in $(\Omega_*^\xihet)_3^\wedge$, so there is no $3$-torsion in this range.
\end{proof}
\begin{rem}
The change-of-rings technique we used at $p = 2$ has an analogue at $p = 3$ for twists of $\tmf$ (hence also
$3$-local twisted string bordism in degrees $15$ and below, because the Ando-Hopkins-Rezk map~\cite{AHR10}
$\MTString_{(3)}\to\tmf_{(3)}$ is $15$-connected~\cite{HR95, Hil09}): using Baker-Lazarev's version of the Adams
spectral sequence~\cite{BL01}, we can take Ext over the algebra
\begin{equation}
	\cA^\tmf\coloneqq \pi_{-*}\mathrm{Map}_{\tmf}(H\Z/3, H\Z/3),
\end{equation}
where $H\Z/3$ is made into a $\tmf$-algebra spectrum by the ring spectrum maps $\tmf\overset{\tau_{\le 0}}\to
H\Z\to H\Z/3$, where the first map is the Postnikov $0$-connected quotient and the second map is induced from
$\Z\twoheadrightarrow\Z/3$. The algebra $\cA^\tmf$ was explicitly calculated by Henriques and Hill, using work of
Behrens~\cite{Beh06} and unpublished work of Hopkins-Mahowald; see Henriques~\cite[\S 13.3]{DFHH14},
Hill~\cite{Hil07}, and Bruner-Rognes~\cite[\S 13]{BR21} for computations with this Adams spectral sequence.

Just like at $p = 2$, there is a little more work to do apply this spectral sequence to twisted string bordism when
the twist does not arise from a vector bundle. We will take up this question in future work joint with Matthew
Yu~\cite{DY}, where we will see how to work over $\cA^{\tmf}$ for non-vector-bundle twists and that it simplifies
the $3$-primary computation of $\Omega_*^\xihet$ in degrees relevant to string theory.
\end{rem}

%% file: chl_bordism.tex
\label{s_CHL_bord}
In this section, we compute the $\xiCHL$ bordism groups. Just like for the $\xihet$ bordism groups, we use the
change-of-rings trick from \cref{fake_shearing} at $p = 2$ and work more directly with the Adams spectral sequence
at odd primes. This time, however, we can deduce a lot of information from abstract isomorphisms with the Adams
spectral sequences for the string bordism of $B\rE_8$, which has been studied by Hill~\cite{Hil09}.

%bordism groups for the symmetry type which is the string cover of
%$\Spin\times\rE_8$ classified by the degree-$4$ class $\lambda - 2c$. By \TODO{}, these are the string bordism
%groups of $(B\rE_8)^{2V - ???}$, where $V$ is the vector bundle associated to the standard representation.
\subsubsection{$2$-primary computation}
\begin{thm}
\label{2_loc_CHL}
In degrees $11$ and below, the $2$-completions of $\Omega_*^{\xiCHL}$ and $\Omega_*^\String(B\rE_8)$ are abstractly
isomorphic.
\end{thm}
\begin{proof}
By \cref{fake_shearing}, the Adams $E_2$-page in this range coincides with the Ext of $T(-2c)$ over $\cA(2)$. The
$\cA(2)$-module structure on $T(\mu)$ only depends on the underlying group $BG$ and on $\mu\bmod 2$, and $2c\bmod 2
= 0$, so as $\cA(2)$-modules, $T(-2c)\cong T(0) = H^*(B\rE_8;\Z/2)$. So the Adams $E_2$-page coincides in the range
we care about with the $E_2$-page for $\mathit{MT\xi}^0 = \MTString\wedge (B\rE_8)_+$. Hill~\cite[Figure 3]{Hil09}
computes the $E_2$-page corresponding for the reduced string bordism of $B\rE_8$, which we use to draw the full
$E_2$-page for $\Omega_\ast^{\xiCHL}$ in \cref{CHL_fig_2}.
\begin{figure}[h!]
%\begin{sseqdata}[name=CHLss, classes = fill, scale=0.75, xrange={0}{12}, yrange={0}{8},
%x label = {$\displaystyle{s\uparrow \atop t-s\rightarrow}$},
%x label style = {font = \small, xshift = -37ex, yshift=6ex}]
%\begin{scope}[gray!75!white]
%	\class(0, 0)\AdamsTower{}
%	\class(1, 1)\structline(0, 0)(1, 1)
%	\class(2, 2)\structline
%	\class(3, 3)\structline
%	\class(3, 2)\structline
%	\class(3, 1)\structline
%	\class(6, 2)\structline
%	\structline(0, 0)(3, 1)
%	\structline(0, 1)(3, 2)
%	\structline(0, 2)(3, 3)
%
%	\class(8, 4)\AdamsTower{}
%	\class(9, 5)\structline(8, 4)(9, 5)
%	\class(10, 6)\structline
%	\class(11, 7)\structline
%	\class(11, 6)\structline
%	\class(11, 5)\structline
%	\class(14, 6)\structline
%	\structline(8, 4)(11, 5)
%	\structline(8, 5)(11, 6)
%	\structline(8, 6)(11, 7)
%
%	\class(8, 3)
%	\class(9, 4)\structline
%	\class(12, 3)\AdamsTower{}
%	\class(15, 4)\structline(12, 3)(15, 4)
%	\class(15, 5)\structline(12, 4)(15, 5)
%\end{scope}
%\begin{scope}[draw=none, fill=none]
%	\class(8, 1)
%	\class(8, 2)
%	\class(12, 3)
%	\foreach \y in {0, 1, 2} {
%		\foreach \z in {1, 2, 3} {
%			\class(12, \y)
%		}
%%	}
%\end{scope}
%\class(4, 0)\AdamsTower{}
%\class(8, 1)\AdamsTower{}
%\class(9, 1)
%\class(10, 2)\structline
%\class(9, 2)\structline(8, 1, -1)(9, 2)
%\class(10, 3)\structline
%\class(10, 0)
%\class(11, 1)\structline
%\class(12, 4)\AdamsTower{}
%\class(12, 3)\AdamsTower{}
%\class(12, 0)\AdamsTower{}
%\end{sseqdata}
%\printpage[name=CHLss, page=2]
\includegraphics{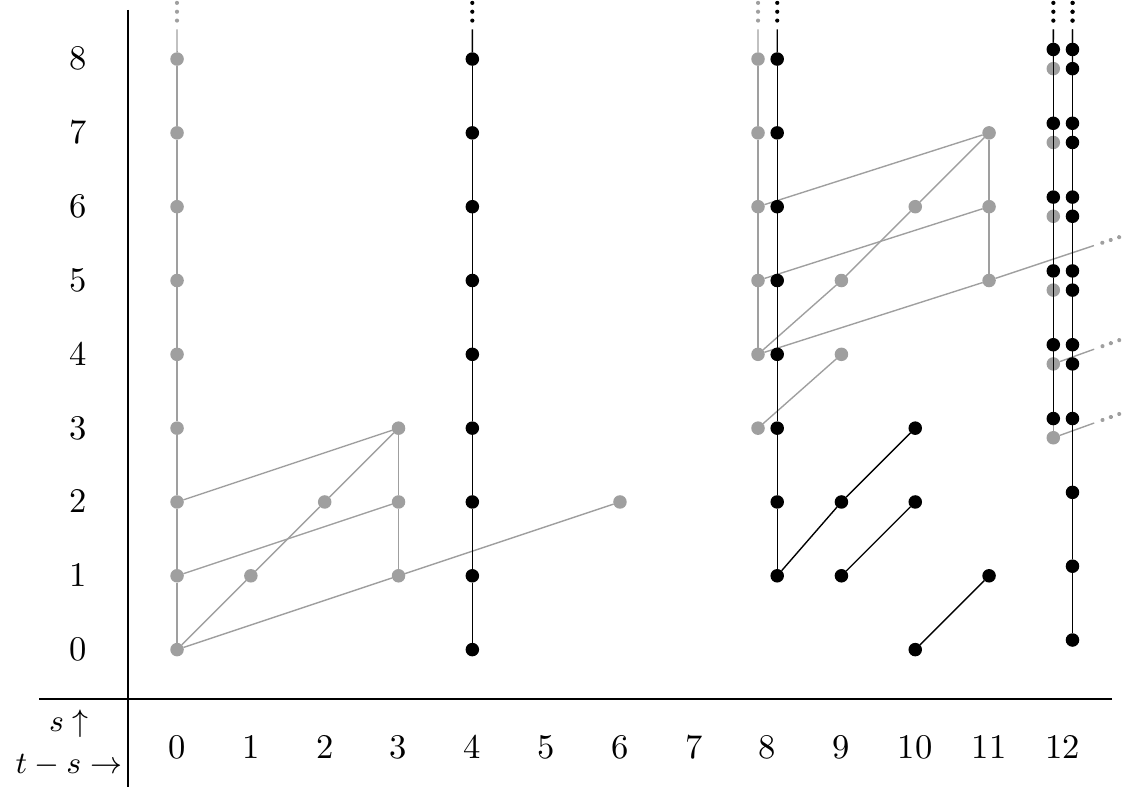}
\caption{The $E_2$-page for the Adams spectral sequence computing $2$-completed $\xiCHL$ bordism. The gray summands
correspond to classes with trivial $\rE_8$-bundle. See \cref{2_loc_CHL} for more information. This figure is
adapted from~\cite[Figure 3]{Hil09}.}
\label{CHL_fig_2}
\end{figure}

This is an abstract isomorphism and does not a priori tell us about differentials or extensions. However,
quotienting by $\T[1]$ defines a map $\GCHL\to \Spin\times\rE_8$, which induces a map on Adams spectral sequences
for Thom spectra of classifying spaces, and this map of Adams spectral sequences is identified with the map induced
by $\MTString\wedge (B\rE_8)_+\to\MTSpin\wedge (B\rE_8)_+$, so any differential for the string bordism of $B\rE_8$
deduced by pulling back from the Adams spectral sequence for $\MTSpin\wedge B\rE_8$ remains valid in our Adams
spectral sequence for $\xiCHL$ bordism.

Moreover, we can realize the part of $\Omega_*^{\xiCHL}$ corresponding to the gray summands in \cref{CHL_fig_2} by
string manifolds with trivial $\rE_8$-bundle, so the gray summands split off of the rest of the Adams spectral
sequence.

Looking at the black summands in \cref{CHL_fig_2}, linearity of differentials with respect to the
$\Ext_\cA(\Z/2)$-action on the $E_2$-page means the only possible nonzero differentials in the range we care about
are $d_2\colon E_2^{0, 10}\to E_2^{2, 11}$ and $d_2\colon E_2^{1, 12}\to E_2^{3, 13}$. Hill~\cite[\S 3.3]{Hil09}
uses the map to $\MTSpin\wedge (B\rE_8)_+$ to show that these two differentials are nontrivial, so as we noted
above, the same is true for $\xiCHL$ bordism.

As there are no more differentials, and all extensions by $2$ in range follow from $\Ext_\cA(\Z/2)$-action without
additional information, we have proven the theorem.
\end{proof}
\begin{rem}
As described in \cref{spinw4}, the map $c\colon B\rE_8\to K(\Z, 4)$ defines a map from $\xiCHL$ structures to
$\Spin\ang{w_4}$ structures, i.e.\ the data of a spin structure and a trivialization of $w_4$. This is the CHL
analogue of the passage from $\xihet$ structures to $\xihet'$ structures from \S\ref{xihet_at_2} --- and just as in
that case, because $c$ is $15$-connected, the induced map $\Omega_k^{\xiCHL}\to \Omega_k^{\Spin\ang{w_4}}$ is an
isomorphism for $k\le 14$, so the computations in this section also give $\Spin\ang{w_4}$ bordism groups.

An alternate point of view due to Sati-Schreiber-Stasheff~\cite[(2.17)]{SSS12} is that $\Spin\ang{w_4}$ structures
are twisted string structures in the sense of \cref{bordism_twists}: the trivialization of $w_4(M)$ is equivalent
data to a class $\mu\in H^4(M;\Z)$ and an identification of $2\mu$ and $\lambda(M)$, so a
$\Spin\ang{w_4}$-structure is a twisted string structure for the map $-2\colon K(\Z, 4)\to K(\Z, 4)$
(corresponding to the classifying space Sati-Schreiber-Stasheff denote $B\mathrm{String}^{2\mathrm{DD}_2}$). See
also~\cite[Remark C.18]{FH21b}.
\end{rem}
The proof of \cref{2_loc_CHL} took advantage of an abstract isomorphism, so it tells us nothing about the generators. The elements of
$\Omega_*^\String(B\rE_8)$ coming from $\Omega_*^\String(\pt)$ are represented by string manifolds with trivial
$\rE_8$-bundle; these vacuously satisfy the condition $2c = \lambda$, so define classes in $\Omega_\ast^\xiCHL$
representing the same elements under the abstract isomorphism with $\Omega_\ast^\String(B\rE_8)$.

That leaves a few elements left: copies of $\Z$ in degrees $4$ and $8$, and copies of $\Z/2$ in degrees $9$ and
$10$. We can represent the generator of $\Omega_4^{\xiCHL}\cong\Z$ by a K3 surface with an $\rE_8$-bundle chosen to
satisfy the Bianchi identity; it would be interesting to determine generators of $\Omega_k^{\xiCHL}$ for $k =
8,9,10$.

\subsubsection{Odd-primary computation}
\begin{thm}
\label{no_odd_CHL}
For $k\le 12$, $\Omega_k^{\xiCHL}$ has no odd-primary torsion.
\end{thm}
\begin{proof}
First we show the result for $p = 3$. The mod $3$ cohomology, as an $\cA$-module, of the string cover $\mathcal
S(G, \lambda)$ only depends on $\lambda\bmod 3$.
%[\TODO: think through this again].
Therefore in the CHL case,
where $\lambda = 2c$, we might as well work with $\lambda = -c$ --- or replacing our $K(\Z, 4)$ class with its
opposite, $\lambda = c$. This string cover corresponds to the universal twist of $\MTString$ over $K(\Z, 4)$ from
\cref{bordism_twists}, which means that by \cref{universal_Thom}, the Thom spectrum for this twist is $\MTSpin$
again! That is, the $E_2$-page of the $3$-primary Adams spectral sequence for CHL bordism coincides with the
$E_2$-page for spin bordism --- or for oriented bordism, because the forgetful map $\MTSpin\to\MTSO$ is a
$3$-primary equivalence.

Milnor~\cite[Theorem 4]{Mil60} shows that the mod $3$ cohomology of $\MTSO$ is free as an $\cA/\beta$-module on
even-degree generators, where $\beta$ is the mod $3$ Bockstein; then, he proves (\textit{ibid.}, Theorem 1) that
for any spectrum with that property and satisfying a finiteness condition, there is no odd-primary torsion in
homotopy. The CHL bordism spectrum satisfies these conditions, so we conclude.

For $p\ge 5$, the argument is essentially the same as in \cref{no_odd_het}.
\end{proof}
%\subsubsection{Computation with $6$ inverted}
%Again following Hill~\cite[\S 3.1]{Hil09}, run the Atiyah-Hirzebruch spectral sequence for $\MTString[1/6]\wedge
%((B\rE_8)^{2V - ???})$. The $E_2$-page only depends on the ordinary cohomology of $(B\rE_8)^{2V - ???}$
%with various coefficient groups, which thanks to the Thom isomorphism is isomorphic to the cohomology of
%$B\rE_8$. Hill shows that in degrees $11$ and below, the Atiyah-Hirzebruch spectral sequence lacks torsion and
%is too sparse for differentials, so we're done.

%% file: string_theory_section.tex
\label{s_phys}
There are a few different uses of bordism groups in theories of quantum gravity. In this section, we discuss
applications and questions raised by the computations in the previous section. Though we stay mostly mathematical,
some of what we state in this section is only known at a physical level of rigor.
\subsection{The cobordism conjecture}
\label{s_cobordism_conjecture}
As part of the Swampland program in quantum gravity, McNamara-Vafa~\cite{MV19} made the following conjecture, a
consequence of the generally believed fact that theories of quantum gravity should not have global symmetries:
\begin{conj}[McNamara-Vafa cobordism conjecture~\cite{MV19}]
\label{cobordism_conjecture}
Suppose we have a consistent $n$-dimensional theory of quantum gravity in which the spacetime backgrounds that are
summed over carry a $\xi$-structure. Then, for $3\le k\le n-1$, $\Omega_k^\xi = 0$.
\end{conj}
The key here is the meaning of ``the spacetime backgrounds carry a $\xi$-structure'' --- we do not mean just that
one could sum over $\xi$-manifolds, but that $\xi$ is in some to-be-specified sense the maximally general structure
for which the theory makes sense. String theorists often work with singular manifolds and even Deligne-Mumford
stacks on $\cat{Man}$~\cite{PS05, PS06a, PS06b, DFM11a, DFM11b}, and the notion of $\xi$-bordism appearing in
\cref{cobordism_conjecture} is expected to take this into account, as some sort of bordism theory of generalized
manifolds.

The tangential structures $\xi$ currently known for various theories of quantum gravity do not satisfy the vanishing
criterion in \cref{cobordism_conjecture}, so there must be additional data or conditions on these theories'
backgrounds modifying $\xi$ so as to kill its bordism groups. These modifications often take the form of additional
extended objects in the theory. This leads to a common application of the cobordism conjecture: compute the bordism
groups for the tangential structure $\xi$ as we currently understand it, and use any nonvanishing groups as beacons
illuminating novel objects in the theory, which one then studies. This idea has been applied in~\cite{MV19, BKRU20,
GEMSV20, DM21, MV21, Sch21, ACC22, BC22, BCKM22, DHMT22, MR22, Wit22, DDHM22a, MVDBL22};\footnote{Despite all of
this work, there are still plenty of already-worked-out computations of bordism groups relevant to various string
and supergravity theories whose corresponding defects have not been determined. This includes
$\Omega_*^\Spin(B\rE_8)$~\cite{Sto86, Edw91}, applicable to the $\rE_8\times\rE_8$ heterotic string in the absence
of the $\Z/2$ swapping symmetry; $\Omega_*^{\mathrm{DPin}}$~\cite[Appendices E, F]{KPMTD20}, relevant for type I
string theory; and $\Omega_\ast^{\mathfrak m_c}$~\cite{FH21b}, useful for the low-energy limit of M-theory.}
in this subsection, we will use our computations from \S\ref{s_computations} and see what we can learn about the
$\rE_8\times\rE_8$ heterotic string and the CHL string.

Despite the $k\ge 3$ bound in \cref{cobordism_conjecture}, modifying $\xi$ to kill classes in $\Omega_1^\xi$ and
$\Omega_2^\xi$ is often physically meaningful, and can predict useful new objects in the theory. This is a common
technique in the study of the cobordism conjecture, and we will do this too.
\subsubsection{The $\rE_8\times\rE_8$ heterotic string}
McNamara-Vafa~\cite[\S 4.5]{MV19} discussed predictions of their conjecture to the $\rE_8\times\rE_8$ heterotic
string theory, but after making the simplifying assumption that the gauge $(\rE_8\times\rE_8)\rtimes\Z/2$-bundle is
trivial; the corresponding tangential structure is then $B\String$. For example, their conjecture must account for
$\Omega_3^\String\cong\Z/24$, generated by $S^3$ with its Lie group framing, and they explain how this is
trivialized by taking into account the NS5-brane.

With $\Omega_*^\xihet$ in hand, we can predict more objects. Recall the generators we found for $\xihet$-bordism
groups, and our notation for them, from \S\ref{xihet_gens}.
\begin{exm}
\label{RP1_brane}
$\Omega_1^\xihet\cong{\Z/2}\oplus\textcolor{RedOrange}{\Z/2}$, with generators
$\Snb$ and $\RP^1$. McNamara-Vafa already considered $\Snb$, but the latter is new. If one
allows manifolds with singularities, $\RP^1$ bounds $D^2/(\Z/2)$, i.e.\ the disc with a principal $\Z/2$-bundle
that is singular at the origin, inflated to a singular $\Ghet$-bundle via the inclusion $\Z/2\inj\Ghet$.

This class corresponds to a $7$-brane in the $\rE_8\times\rE_8$ heterotic string. The worldvolume of this brane is
eight-dimensional, so the link around it in ten-dimensional spacetime is a circle. The monodromy around this circle
exchanges the two $\rE_8$-bundles. This is exactly the non-supersymmetric $7$-brane recently introduced and
discussed by Kaidi-Ohmori-Tachikawa-Yonekura~\cite{KOTY23}.

Related $7$-branes in different theories are studied by Distler-Freed-Moore~\cite{DFM11a} and
Dierigl-Heckman-Montero-Torres~\cite{DHMT22}; the latter study a $7$-brane in type IIB string theory, called an
R7-brane, which in the cobordism conjecture corresponds to $[\RP^1]\in\Omega_1^{\Spin\text{-}\GL_2^+(\Z)}$.

As a way of better understanding Kaidi-Ohmori-Tachikawa-Yonekura's $7$-brane, we can try
to identify where it is sent under dualities between different string theories. For example,
Hořava-Witten~\cite{HW96a, HW96b, Wit96} identified (a certain limit) of $\rE_8\times\rE_8$ heterotic string theory
with a theory predicted to be the
%\todo{M-theory means expected to work, not actually done!}
low-energy limit of a compactification of M-theory on the unit interval. Under this identification, the
Kaidi-Ohmori-Tachikawa-Yonekura $7$-brane ought to correspond to a defect in M-theory associated to a $2$-dimensional bordism class
by the cobordism conjecture. Because the passage from M-theory to heterotic string theory requires compactifying on
the interval, which is
a manifold with boundary, one should use a theory of bordism of compact manifolds which are not necessarily
closed.\footnote{McNamara-Vafa~\cite[\S 5]{MV19} hint at this generalization, though from the perspective of
manifolds with singularities rather than manifolds with boundary.} The bordism class should be represented by an
interval bundle over $\RP^1$, so we conjecture that the bordism class of the Möbius strip corresponds to the avatar
of this brane in M-theory. As a check, M-theory compactified on a Möbius strip is expected to coincide with
$\rE_8\times\rE_8$ heterotic string theory compactified on $\RP^1$ --- they are both predicted to be the CHL
string, as we discussed in \S\ref{CHL_intro}, though as usual only a statement about low-energy supergravity limits
is known. We will not attempt to fully resolve this question in this paper:
among other things, this would require finding ``the right'' notion of bordism for manifolds with boundary for this
application.

Before we leave heterotic/M-duality behind, we point out a notion of bordism of manifolds with boundary, due to
Conner-Floyd~\cite[\S 16]{CF66}, for which the Möbius strip is nonbounding; we optimistically conjecture that this
is the correct kind of bordism of manifolds with boundary for applications to the cobordism conjecture.
\begin{defn}
Let $\xi_1\colon B_1\to B\O$ and $\xi_2\colon B_2\to B\O$ be tangential structures and $\eta\colon B_1\to B_2$ be a
\term{map of tangential structures}, i.e.\ $\eta$ commutes with the maps $\xi_i$. A \term{$\xi_2/\xi_1$-manifold}
is a compact manifold $M$ with $\xi_2$-structure together with
\begin{enumerate}
	\item a $\xi_1$-structure $\mathfrak x$ on $\partial M$, and
	\item an identification of the $\xi_2$-structure $\eta(\mathfrak x)$ on $\partial M$ with the $\xi_2$-structure
	induced by taking the boundary on $M$.
\end{enumerate}
\end{defn}
Conner-Floyd~\cite[\S 16]{CF66} introduce a notion of bordism for $\xi_2/\xi_1$-manifolds,\footnote{Conner-Floyd
only consider a few examples of $\xi_1$ and $\xi_2$. The works~\cite{Sto68, Ale75, Mit75, RST77, Lau00, Bun15}
consider some more tangential structures.} which we write $\Omega_*^{\xi_2/\xi_1}$, such that the Thom spectrum
corresponding to this notion of bordism is $\MTxi_2/\MTxi_1$, the cofiber of $\eta\colon \MTxi_1\to\MTxi_2$. This
implies the existence of a long exact sequence
\begin{equation}
\label{CF_bord_LES}
	\dotsb\longrightarrow \Omega_k^{\xi_1}\overset{\eta}{\longrightarrow}
	\Omega_k^{\xi_2}\overset{j}{\longrightarrow} \Omega_k^{\xi_1/\xi_2}\overset{\partial}{\longrightarrow}
	\Omega_{k-1}^{\xi_1} \longrightarrow\dots
\end{equation}
where $j$ regards a $\xi_2$-manifold as a $\xi_2/\xi_1$-manifold with empty boundary.
\begin{lem}
\label{mob_CF}
The class of the Möbius strip $M$ is nonzero in $\Omega_2^{\Pin^+/\Spin}$.\footnote{Strictly speaking,
$\Pin^+/\Spin$ is not the correct tangential structure: one should replace $\Pin^+$ with something like $\mathfrak
m_c$~\cite{Wit97, Wit16, FH21b}, and should replace $\Spin$ with something like $\xihet$, though $\mathfrak m_c$-
and \pinp structures on $2$-manifolds are equivalent data~\cite[\S 8.5.1]{FH21b}.}
\end{lem}
\begin{proof}
By~\eqref{CF_bord_LES}, it suffices to prove that $[\partial M]\ne 0$ in $\Omega_1^\Spin$. The boundary of
the Möbius strip is a circle, and for any \pinp structure on $M$, the boundary circle has the nonbounding spin
structure, i.e.\ is nonzero in $\Omega_1^\Spin$. This is because if $\partial M$ had the bounding
spin structure, one could glue the disc with its standard \pinp structure to $M$ along $\partial M$ and thereby
obtain a \pinp structure on $\RP^2$, but $\RP^2$ does not admit a \pinp structure.
\end{proof}
\Cref{mob_CF} suggests that Conner-Floyd's notion of bordism of manifolds with boundary could be the correct one
for our application in heterotic/M-theory duality. 
\end{exm}
\begin{exm}
\label{omega_2}
Moving onto higher-codimension objects predicted by higher-dimensional bordism groups, $\Omega_2^\xihet$ is
nonzero, but can be generated by products of $S_{\mathit{nb}}^1$ and $\RP^1$. This means that if we trivialize
$[\RP^1],[\Snb]\in\Omega_1^{\xihet}$ in the sense above, namely by allowing $\rE_8\times\rE_8$ heterotic string
theory to be defined on singular manifolds whose boundaries are $\RP^1$ and $\Snb$, then we can realize our chosen
generators of $\Omega_2^{\xihet}$ as boundaries of singular $3$-manifolds: for example, we used
$D^2/(\Z/2)$ to realize $\RP^1$ as a boundary, so we can use $\Snb\times D^2/(\Z/2)$ to realize $\Snb\times\RP^1$
as a boundary. Thus accounting for $\Omega_2^{\xihet}$ does not require adding any new kinds of defects or
singularities beyond what we used for $\Omega_1^\xihet$.
\end{exm}
\begin{exm}
\label{RP3_brane}
$\Omega_3^\xihet\cong \textcolor{RedOrange}{\Z/8}$, generated by $\RP^3$. As in \cref{RP1_brane}, we can bound
$\RP^3$ by $B^4/(\Z/2)$ by allowing a singularity at the origin. This bordism class should correspond to a
$5$-brane distinct from the NS5-brane.
\end{exm}
\begin{exm}
\label{K3_brane}
$\Omega_4^\xihet\cong{\Z}\oplus\textcolor{RedOrange}{\Z/2}$. The $\textcolor{RedOrange}{\Z/2}$
summand is generated by $S^3\times\RP^1$, where $S^3$ carries the Lie group framing, so its bordism class can be
trivialized using the objects we have already discussed, like in \cref{omega_2}. By \cref{S4_generator}, because
$S^3\times S^1$ is bordant as $\xihet$-manifolds to $S^4$ with trivial $\Z/2$-bundle and $\rE_8$-bundles with
characteristic classes $\pm 1\in H^4(S^4;\Z)\cong\Z$, this bordism class corresponds to the $4$-brane recently
found by Kaidi-Ohmori-Tachikawa-Yonekura~\cite{KOTY23}.

The ${\Z}$ summand in $\Omega_4^{\xihet}$ is new to us. It is generated by the K3
surface with trivial $\Z/2$-bundle; one $\rE_8$-bundle is trivial, and the other cancels $\lambda(\mathrm{K3})$.
McNamara-Vafa~\cite[\S 4.2.1]{MV19} address the K3 surface without data of $\rE_8$-bundles or a nontrivial B-field,
using it to exhibit a higher-form $\T$-symmetry. Our K3 surface corresponds to a different bordism class, but
McNamara-Vafa's argument still applies: as the K3 surface is believed to be a valid background for
$\rE_8\times\rE_8$ heterotic string theory, this higher-form $\T$-symmetry must be broken or gauged in some way.
We do not know what this would look like.
\end{exm}
$\Omega_5^\xihet$ vanishes and $\Omega_6^\xihet\cong\textcolor{RedOrange}{\Z/2}$ is generated by $\RP^3\times S^3$,
so as in \cref{omega_2} we can realize it as a boundary without introducing any new kinds of singularities.
\begin{exm}
\label{RP7_brane}
$\Omega_7^\xihet\cong\textcolor{RedOrange}{\Z/16}$, generated by $\RP^7$.
%$\RP^1\times S^3\times S^3$, respectively. The factors in $\RP^1\times S^3\times S^3$ have already been addressed,
%so we focus on $\RP^7$.
This bordism class is closely analogous to \cref{RP1_brane,RP3_brane}; this time, we have a $1$-brane, i.e.\ a
string.
\end{exm}
\begin{rem}[Relating bordism classes by compactification\protect\footnotemark]
\label{RP_cpt_rmk}
\footnotetext{We thank Markus Dierigl for pointing this out to us.}
For the cobordism conjecture for type IIB string theory considered on spin-$\GL_2^+(\Z)$ manifolds,
$[\RP^k]\in\Omega_k^{\Spin\text{-}\GL_2^+(\Z)}$ is nonzero for $k = 1$, $3$, and $7$~\cite[\S 14.3.2]{DDHM22a}, so
we would expect these classes to correspond to three different kinds of extended objects, akin to
\cref{RP1_brane,RP3_brane,RP7_brane}. However, in~\cite[\S 7]{DDHM22a}, it is shown that the two higher-codimension
objects can be expressed as compactifications of the R7-brane corresponding to $\RP^1$, so there is really only one
novel object. We suspect something similar happens here: that in $\rE_8\times\rE_8$ heterotic string theory, the
extended objects corresponding to $\RP^3$ and $\RP^7$ can be accounted for using previously known branes and
Kaidi-Ohmori-Tachikawa-Yonekura's
$7$-brane from \cref{RP1_brane} corresponding to $\RP^1$.

From a bordism point of view, we are saying that if we allow singular $\xihet$-manifolds which locally look like
$\R^k\times D^2/(\Z/2)$, it should be possible to not just bound $\RP^1$, but also to bound $\RP^3$ and $\RP^7$. We
leave this as a conjecture.
\end{rem}
\begin{exm}
$\Omega_8^\xihet$, which corresponds to codimension-$9$ objects, is isomorphic to either $\Z^3\oplus\Z/2$ or
$\Z^3\oplus (\Z/2)^{\oplus 2}$, depending on the fate of the differential~\ref{h1_orange_d2_family}. The generators
of these four or five summands that we found are:
\begin{itemize}
	\item $\HP^2$ with two different $\xihet$-structures, giving two $\Z$ summands;
	\item the Bott manifold, generating another free summand;
	\item $\RP^7\times\Snb$ generating the $\Z/2$ summand that is present even if~\ref{h1_orange_d2_family} does
	not vanish; and
	\item the manifold $X_8$ that we discussed in \S\ref{s:X8}, an $S^3\times S^3$-bundle over $\RP^2$. If the
	differential \ref{h1_orange_d2_family} is nonzero, then $X_8$ bounds as a $\xihet$-manifold.
\end{itemize}
$\RP^6\times\Snb$ is already accounted for in the sense of \cref{omega_2}, so we focus on the other generators.

Both $B$ and $\HP^2$ are nonbounding in the bordism group $\Omega_8^{\Spin\text{-}\mathrm{Mp}_2(\Z)}$, which
appears in the study of the cobordism conjecture for type IIB string theory; see~\cite[\S 6.9]{DDHM22a} for a
discussion of defects in type IIB corresponding to these bordism classes. Like in \cref{K3_brane}, the story in
$\rE_8\times\rE_8$ heterotic string theory is presumably not exactly the same, but it may be analogous.

Finally, $X_8$. Following the arguments in~\cite[\S 4.5]{MV19} and~\cite[\S 7.6, \S 7.8]{DDHM22a} the description of
$X_8$ as a fiber bundle over $\RP^2$ with fiber $S^3\times S^3$ suggests the following string-theoretic
construction: use the singular manifold corresponding to the NS5-brane to bound for the first $S^3$, compactify on
the second $S^3$, and then fiber over $D^3/(\Z/2)$ to make $X_8$ a boundary of a singular manifold. We do not know
whether this is a valid background for the $\rE_8\times\rE_8$ heterotic string; an argument for or against it could
provide an example of a use of \cref{cobordism_conjecture} to make a mathematical conjecture for the fate of $X_8$
based on string-theoretic predictions.
\end{exm}
\begin{exm}
$\Omega_9^\xihet$ corresponds to zero-dimensional objects, i.e.\ point defects, and is isomorphic to either
$(\Z/2)^{\oplus 4}$ or $(\Z/2)^{\oplus 6}$, depending on the fate of~\ref{h1_orange_d2_family}. Three of the
generators we found in \S\ref{xihet_gens} are of the form $\Snb$ times a $\xihet$-manifold, so have already been
accounted for in the sense of \cref{omega_2}. The fourth generator is $B\times\RP^1$, so it is also already
accounted for.

The remaining two manifolds that might or might not be necessary are $X_8\times\Snb$, which as usual is already
taken care of, and a manifold $X_9$ which we did not determine.
%\TODO: codim-$9$ (not interesting, because $\eta\cdot$ is surjective. but point out these are point defects)
\end{exm}
%\begin{exm}
%We did not fully determine $\Omega_9^\xihet$ in \S\ref{xihet_at_2}; we did find a $\textcolor{RedOrange}{\Z/2}$
%summand corresponding to $E_\infty^{4,12}\cong\textcolor{RedOrange}{\Z/2}$, generated by
%$\mathrm{Bott}\times\RP^1$. There is also possibly a $\textcolor{PineGreen}{\Z/2}$ summand generated by a
%$\xihet$-manifold $W_9$ with
%\begin{equation}
%	\int_{W_9} c(P_1)c(P_2)x \equiv 1\bmod 2,
%\end{equation}
%where $P_1$ and $P_2$ are the two associated principal $\rE_8$-bundles and $x$ is the characteristic class of the
%associated $\Z/2$-bundle; and there are likely some other summands. The defects predicted by these bordism groups
%are zero-dimensional, i.e.\ they are localized. For $\mathrm{Bott}\times\RP^1$, one should understand the line
%defect corresponding to the Bott manifold, which is not known (see~\cite{DDHM22a}). It would be particularly
%interesting to describe the point defect associated to $W_9$, if indeed such a bordism class exists: this is the
%first bordism class for which the $\Z/2$-bundle and both $\rE_8$-bundles must be nontrivial, so its presence may
%tell us something about the $\Z/2$ symmetry exchanging the two $\rE_8$-bundles.
%\end{exm}

\subsubsection{The CHL string}
In \cref{2_loc_CHL,no_odd_CHL}, we saw that $\Omega_*^\xiCHL$ is abstractly isomorphic to
$\Omega_*^\String(B\rE_8)$. Thus there is a summand corresponding to $\Omega_*^\String(\pt)$, and as we saw above,
these classes can be represented by string manifolds with trivial $\rE_8$-bundle. Some of these manifolds were
accounted for by McNamara-Vafa~\cite[\S 4.5]{MV19} in heterotic string theory, e.g.\ killing $S^3$ with its
nonbounding framing using the fivebrane, and presumably a similar defect is present in the CHL string.
McNamara-Vafa leave plenty of string bordism classes' interpretations in terms of defects open to address, and this
would be interesting to understand more in the setting of the CHL string.

We also found a few more classes in $\Omega_*^\xiCHL$. For example, $\Omega_4^{\xiCHL}\cong\Z$, generated by a K3
surface with $\rE_8$-bundle chosen to satisfy the Bianchi identity. Like in \cref{K3_brane}, this corresponds to
some codimension-$4$ object, though we do not know what it will look like.
\subsection{Is the $\Z/2$ symmetry on the $\rE_8\times\rE_8$ heterotic string anomalous?}
\label{anomalies_general}
Quantum field theories can come with the data of an \term{anomaly}, a mild inconsistency in which
key quantities in the field theory are not defined absolutely without fixing additional
data. For example, one wants the partition function of a QFT on a manifold $M$ to be a complex number, but an
anomaly signals that the partition function is only an element of a complex line which has not been trivialized.
The process of resolving this inconsistency, when necessary, is called \term{anomaly cancellation}.
%part of the procedure for defining a quantum field theory is a consistency check called \term{anomaly
%ancellation}: a priori, key quantities in the field theory are not defined absolutely without fixing additional
%data. For example, one wants the partition function of a QFT on a manifold $M$ to be a complex number, but before
%anomaly cancellation, the partition function is only an element of a complex line which has not been trivialized.

Freed-Teleman~\cite{FT14} describe anomaly cancellation for a broad class of quantum field theories as follows: an
$n$-dimensional quantum field theory $Z$ lives at the boundary of an $(n+1)$-dimensional invertible field theory
$\alpha$, called the \term{anomaly field theory} of $Z$. The tangential structures of $Z$ and $\alpha$ match.
Anomaly cancellation is the procedure of trivializing $\alpha$, i.e.\ establishing an isomorphism from $\alpha$ to
the trivial theory.

We think of this from Atiyah-Segal's approach~\cite{AtiyahTFT, SegalCFT} that field theories are symmetric monoidal
functors from (potentially geometric) bordism categories into categories such as $\cat{Vect}_\C$. The perspective
of extended field theory means these are often $(\infty, n)$-categories. If $\cat C$ and $\cat D$ are two symmetric
monoidal $(\infty, n)$-categories, the $(\infty, n)$-category of symmetric monoidal functors $F\colon\cat C\to\cat
D$ acquires the symmetric monoidal structure of ``pointwise tensor product,'' specified by the formula
\begin{equation}
	(F_1\otimes F_2)(x) \coloneqq F_1(x) \otimes_{\cat D} F_2(x), 
\end{equation}
where $x$ is an object, morphism, higher morphism, etc.
\begin{defn}[{Freed-Moore~\cite[Definition 5.7]{FM06}}]
Let $\cat C$ be a symmetric monoidal $(\infty, n)$-category. An \term{invertible field theory} is a field theory
$\alpha\colon\Bord_n^\xi\to\cat C$ such that there is another field theory $\alpha^{-1}\colon\Bord_n^\xi\to\cat C$
such that $\alpha\otimes\alpha^{-1}\simeq \boldsymbol 1$, the trivial theory.
\end{defn}
The trivial theory $\boldsymbol 1\colon\Bord_n^\xi\to\cat C$ is defined to send all objects to the monoidal unit in
$\cat C$ and all morphisms and higher morphisms to identity morphisms, resp.\ identity higher morphisms.

Therefore the classification of anomalies follows from the classification of invertible field theories, and anomaly
cancellation is an isomorphism from an invertible field theory to $\boldsymbol 1$.
Freed-Hopkins-Teleman~\cite{FHT10} classify invertible \emph{topological} field theories using stable homotopy
theory, and Grady-Pavlov~\cite[\S 5]{GP21} generalize this in the nontopological setting.

In most cases, including the supergravity theories studied in this paper, the QFT under study is unitary, so their
anomaly theories have the Wick-rotated analogue of unitarity, \term{reflection positivity}.
Freed-Hopkins~\cite{FH21} classify reflection-positive invertible field theories.

Let $I_\Z$ denote the \term{Anderson dual of the sphere spectrum}~\cite{And69, Yos75}.
\begin{thm}[{Freed-Hopkins~\cite[Theorem 1.1]{FH21}}]
Let $\xi$ be a tangential structure. There is a natural isomorphism from the group of deformation classes of
$(n+1)$-dimensional reflection-positive invertible topological field theories on $\xi$-manifolds to the torsion
subgroup of $[\mathit{MT\xi}, \Sigma^{n+2}I_\Z]$.
\end{thm}
Freed-Hopkins then conjecture (\textit{ibid.}, Conjecture 8.37) that the entire group classifies all
reflection-positive invertible field theories, topological or not.

$I_\Z$ satisfies a universal property which leads to the existence of a natural short exact sequence
\begin{equation}
\label{anomaly_SES}
	\shortexact{\mathrm{Tors}(\Hom(\Omega_{n+1}^\xi, \T))}{[\mathit{MT\xi},
	\Sigma^{n+2}I_\Z]}{\Hom(\Omega_{n+2}^\xi, \Z)},
\end{equation}
and this sequence carries physical meaning for the classification of possible anomalies for an $n$-dimensional QFT
$Z$. For example, $\Hom(\Omega_{n+2}^\xi, \Z)$ is a group of $\Z$-valued degree-$(n+2)$ characteristic classes of
$\xi$-manifolds, and the quotient map in~\eqref{anomaly_SES} sends the anomaly field theory of $Z$ to its
\term{anomaly polynomial}. This data can often be computed using perturbative techniques for $Z$, and is referred
to as the \term{local anomaly}. Consequently, one can use bordism computations to assess what the group of possible
anomalies of a QFT is, and whether a specific anomaly field theory is trivializable; see~\cite{FH21b, TY21, DDHM22,
LY22, DY22, Tac22, DOS23} for recent anomaly cancellation theorems in string and supergravity theories using this
technique.% [\TODO: who else?]

\subsubsection{Anomalies for the $\rE_8\times\rE_8$ heterotic string}
\label{heterotic_anomaly}
For the $\rE_8\times\rE_8$ heterotic string, the anomaly field theory is an element of the group $[\MTxihet,
\Sigma^{12} I_\Z]$: the free part is noncanonically isomorphic to the free part of $\Omega_{12}^\xihet$, and the
torsion part is noncanonically isomorphic to the torsion subgroup of $\Omega_{11}^\xihet$. Though we have not
completely determined these groups, $\Omega_{11}^{\xihet}$ is nonzero, as we showed in \cref{het_at_2}, so there is
the possibility of a nontrivial anomaly to cancel. One generally expects that the anomaly field theory itself is
trivial, because physicists have undertaken many consistency checks on $\rE_8\times\rE_8$ heterotic string theory,
but sometimes there is a surprise: in joint work with Dierigl, Heckman, and Montero~\cite{DDHM22}, we found that
the anomaly theory for the duality symmetry in type IIB string theory is nonzero, and requires a modification of
the theory to be trivialized.

For the $\rE_8\times\rE_8$ heterotic string, there has been a fair amount of work already cancelling the anomaly in
special cases, but for the full tangential structure $\xihet$, the question of anomaly cancellation is open. The
original work of Green-Schwarz~\cite{GS84} shoes that the anomaly polynomial vanishes, so by~\eqref{anomaly_SES},
we only need to look at bordism invariants out of $\Omega_{11}^\xihet$. If one ignores the $\Z/2$ swapping
symmetry, the anomaly is known to be trivial: Witten~\cite[\S 4]{Wit86} showed that the global anomaly is
classified by a bordism invariant $\Omega_{11}^\Spin(B\rE_8)\to\C^\times$, and Stong~\cite{Sto86} showed that
$\Omega_{11}^\Spin(B\rE_8) = 0$ (see \cref{restricted_het_is_easy}). Sati~\cite{Sat11a} studies a closely related
question in terms of $\Omega_{11}^\String(B\rE_8)$.

Recent work of Tachikawa-Yamashita~\cite{TY21} (see also Tachikawa~\cite{Tac22} and Yonekura~\cite[\S 4]{Yon22})
cancels anomalies in a large class of compactifications of heterotic string theory using an ingenious
$\mathit{TMF}$-based argument. Their work does not take into account the $\Z/2$ swapping symmetry. It would be
interesting to address the full anomaly on $\xihet$-manifolds, either by directly computing it on generators of
$\Omega_{11}^\xihet$ or by adapting Tachikawa-Yonekura's argument. If this symmetry does have a nontrivial anomaly,
this would have consequences for the CHL string, either requiring a modification of the theory or showing that it
is inconsistent.
\subsubsection{Anomalies for the CHL string}
Anomaly cancellation for the CHL string has been studied less. In \cref{2_loc_CHL,no_odd_CHL}, we saw that
$\Omega_{11}^\xiCHL$ is torsion, so the anomaly polynomial vanishes; and we saw
$\Omega_{10}^\xiCHL\cong\Z/2\oplus\Z/2$, so there is a potential for the anomaly field theory to be nontrivial,
which would be interesting to check.
%\subsection{$\Theta$-angles and the uniqueness conjecture}
%\TODO: mention the Swampland uniqueness conjecture, and how tensoring with invertible field theories means there's
%something to work out. Cite [DDHM2], Miguel's most recent paper.

%% file: bbl_bib.tex
\newcommand{\etalchar}[1]{$^{#1}$}